\documentclass[12pt]{article}
\usepackage{amsmath,amsthm,amssymb,stmaryrd,xy}
\usepackage{mathabx}
\usepackage[a4paper,textwidth=160mm,textheight=225mm]{geometry}

\theoremstyle{definition} %\newcommand{\fixme}[1]{}
\newtheorem{remark}{Remark}[section]
\newtheorem{remarks}[remark]{Remarks}
\newtheorem{example}[remark]{Example}
 \theoremstyle{plain}
\newtheorem{definition}[remark]{Definition}
\newtheorem{theorem}[remark]{Theorem}
\newtheorem{proposition}[remark]{Proposition}
\newtheorem{corollary}[remark]{Corollary}
\newtheorem{lemma}[remark]{Lemma}
\newtheorem*{assumption*}{Assumption}

\newcommand{\intd}{\,\mathrm{d}}

\newcommand{\complex}{\mathbb{C}}
\newcommand{\reals}{\mathbb{R}}

\DeclareMathOperator{\Id}{id}

\DeclareMathOperator{\supp}{supp}

\DeclareMathOperator{\Ad}{Ad}

\DeclareMathOperator{\Ind}{Ind}

\newcommand{\smalldiagram}{}%{\def\labelstyle{\mapstyle}}%\def\objectstyle{\spacestyle}}

\newcommand{\bffell}{\mathbf{Fell}}
\newcommand{\bfcoact}{\mathbf{Coact}}

\newcommand{\bfFhA}{\mathbf{\hat{F}}}
\newcommand{\bfGhA}{\mathbf{\hat G}}
\newcommand{\bfFA}{\mathbf{\check F}}
\newcommand{\bfGA}{\mathbf{\check G}}
\newcommand{\bfF}{\mathbf{F}}
\newcommand{\bfG}{\mathbf{G}}
\newcommand{\bfC}{\mathbf{C}}
\newcommand{\bfD}{\mathbf{D}}

\newcommand{\bfCG}{\mathbf{C_{0}(G^{0})\textrm{-}alg}}
\newcommand{\bfb}{\mathbf{C^{*}\textrm{-}\frakb\textrm{-}alg}}
\newcommand{\bfact}{\mathbf{G\textrm{-}act}}

\newcommand{\hDelta}{\widehat{\Delta}}
\newcommand{\hhdelta}{\widehat{\widehat{\delta}}}

\newcommand{\frakA}{\mathfrak{A}}
\newcommand{\frakB}{\mathfrak{B}}
\newcommand{\frakC}{\mathfrak{C}}

\newcommand{\frakAo}{\mathfrak{A}^{\dag}}
\newcommand{\frakBo}{\mathfrak{B}^{\dag}}
\newcommand{\frakCo}{\mathfrak{C}^{\dag}}

\newcommand{\fraka}{\mathfrak{a}}
\newcommand{\frakb}{\mathfrak{b}}
\newcommand{\frakc}{\mathfrak{c}}
\newcommand{\frakd}{\mathfrak{d}}
\newcommand{\frake}{\mathfrak{e}}

\newcommand{\frakbo}{\mathfrak{b}^{\dag}}

\newcommand{\frakdo}{\mathfrak{d}^{\dag}}

\newcommand{\frakH}{\mathfrak{H}}
\newcommand{\frakK}{\mathfrak{K}}
\newcommand{\frakL}{\mathfrak{L}}

\newcommand{\cbasel}[2]{(\mathfrak{#2},\mathfrak{#1},
\mathfrak{#1}^{\dag})}
\newcommand{\cbases}[2]{\cbasel{#1}{#2}}
\newcommand{\cbaseos}[2]{(\mathfrak{#2},\mathfrak{#1}^{\dag},
\mathfrak{#1})} 
\newcommand{\cbasesb}{\cbases{B}{K}}

\newcommand{\cbaseosb}{\cbaseos{B}{K}}

\newcommand{\hdelta}{\widehat\delta}

\newcommand{\aHb}{{_{\alpha}H_{\beta}}}

\newcommand{\cKd}{{_{\gamma}K_{\delta}}}
\newcommand{\cKhd}{{_{\gamma}K_{\hdelta}}}

\newcommand{\lt}{\smalltriangleleft}
\newcommand{\rt}{\smalltriangleright}

\newcommand{\halpha}{\widehat{\alpha}}
\newcommand{\hbeta}{\widehat{\beta}}

\newcommand{\hcA}{\widehat{\mathcal{ A}}}
 \newcommand{\cA}{\mathcal{ A}}
 \newcommand{\cC}{\mathcal{ C}}

\newcommand{\hA}{\widehat{A}} \newcommand{\ha}{\widehat{a}}

\newcommand{\mycong}{\xrightarrow{\cong}}

\newcommand{\coact}{\mathbf{Coact}}

\newcommand{\rtensor}[3]{ {_{#1}\! \underset{#2}{\otimes}\!
{}_{#3}}}

\newcommand{\cgtensor}[2]{{}_{#1}\boxtimes_{#2}}

\newcommand{\rtensorab}{\rtensor{\alpha}{\frakb}{\beta}}
\newcommand{\rtensorcb}{\rtensor{\gamma}{\frakb}{\beta}}

\newcommand{\tensor}[1]{\underset{#1}{\otimes}}

\newcommand{\btensor}{\underset{\frakb}{\otimes}}
\newcommand{\botensor}{\underset{\frakbo}{\otimes}}

\newcommand{\rtensorh}{\underset{\frakb}{\otimes}}

\newcommand{\fibre}[3]{ {_{#1}\! \underset{#2}{\ast}\!
{}_{#3}}} 
\newcommand{\fib}[3]{ {_{#1}\! \underset{#2}{\ast}\!
{}^{#3}}}

\newcommand{\bfibre}{\underset{\frakb}{\ast}}

\newcommand{\fsource}{\rtensor{\hbeta}{\frakbo}{\alpha}}
\newcommand{\frange}{\rtensor{\alpha}{\frakb}{\beta}}

\newcommand{\tl}{\ensuremath \olessthan}
\newcommand{\tr}{\ensuremath \ogreaterthan}

\newcommand{\Hsource}{H \fsource H} \newcommand{\Hrange}{H
\frange H}

\newcommand{\Hone}{H \fsource H \fsource H}
\newcommand{\Htwo}{H \frange H \fsource H}
\newcommand{\Hthree}{H \frange H \frange H}
\newcommand{\Hfour}{(\Hsource) \rtensor{\alpha \lt
\alpha}{\frakb}{\beta} H} \newcommand{\Hfive}{H
\rtensor{\hbeta}{\frakbo}{\alpha \rt \alpha} (\Hrange)}

\newcommand{\HtensorK}{H \rtensor{\beta}{\frakb}{\gamma} K}
\newcommand{\KtensorH}{K \rtensor{\gamma}{\frakbo}{\beta} H}

\newcommand{\rfsource}{\rtensor{\hdelta}{\frakbo}{\alpha}}
\newcommand{\rfrange}{\rtensor{\gamma}{\frakb}{\beta}}

\newcommand{\rHsource}{K \rfsource H}
\newcommand{\rHrange}{K \rfrange H}

\newcommand{\checkHsource}{H
  \rtensor{\halpha}{\frakb}{\hbeta} H}

\newcommand{\checkHrange}{H
  \rtensor{\hbeta}{\frakbo}{\alpha} H}

\newcommand{\hatHsource}{H \rtensor{\alpha}{\frakb}{\beta}
  H}

\newcommand{\hatHrange}{H \rtensor{\beta}{\frakbo}{\halpha}
  H}

\newcommand{\checkV}{\widecheck{V}}
\newcommand{\hatV}{\widehat{V}}

\newcommand{\fibreab}{\fibre{\alpha}{\frakb}{\beta}}
\newcommand{\fibrebc}{\fibre{\beta}{\frakb}{\gamma}}
\newcommand{\AfibreA}{A \fibreab A}

\newcommand{\AfibreB}{A \fibrebc B}
\newcommand{\AfibB}{A \fibre{\beta}{\frakb}{\gamma} B}

 \newcommand{\Cl}{C_{\lambda}}
\newcommand{\gB}{{_{\gamma}B}}

\newcommand{\kalpha}[1]{|\alpha\rangle_{{#1}}}
\newcommand{\balpha}[1]{\langle\alpha|_{{#1}}}
\newcommand{\kbeta}[1]{|\beta{}\rangle_{{#1}}}
\newcommand{\bbeta}[1]{\langle\beta|_{{#1}}}
\newcommand{\khbeta}[1]{|\hbeta{}\rangle_{{#1}}}
\newcommand{\bhbeta}[1]{\langle\hbeta|_{{#1}}}
\newcommand{\khalpha}[1]{|\halpha{}\rangle_{{#1}}}

\newcommand{\kgamma}[1]{|\gamma\rangle_{{#1}}}
\newcommand{\bgamma}[1]{\langle\gamma|_{{#1}}}

\newcommand{\GrrG}{G {_{r}\times_{r}} G}
\newcommand{\GsrG}{G {_{s}\times_{r}} G}
\newcommand{\FsrF}{{\cal F} {_{sp}\times_{rp}} {\cal F}}
 \newcommand{\GtrG}{G {_{t}\times_{r}} G}

\newcommand{\leg}[1]{#1} \title{Coactions of Hopf
  $C^{*}$-bimodules}

\author{Thomas Timmermann\\[1ex]
\texttt{timmermt@math.uni-muenster.de}\\ 
Universit\"at M\"unster, \\
Einsteinstrasse 62, 48149 M\"unster, Germany, \\
phone: ++49 251 8332724, fax: ++49 251 8332708 \\
email: \texttt{timmermt@math.uni-muenster.de}
}

\date{\today}

\begin{document} \xyrequire{matrix} \xyrequire{arrow}
\xyrequire{curve} \CompileMatrices

\maketitle

\abstract{Coactions of Hopf $C^{*}$-bimodules simultaneously
  generalize coactions of Hopf $C^{*}$-algebras and actions
  of groupoids. Following an approach of Baaj and Skandalis,
  we construct reduced crossed products and establish a
  duality for fine coactions. Examples of coactions arise
  from Fell bundles on groupoids and actions of a groupoid
  on bundles of $C^{*}$-algebras.  Continuous
  Fell bundles on an \'etale groupoid  correspond to
  coactions of the reduced groupoid algebra, and actions of
  a groupoid on a continuous bundle of $C^{*}$-algebras
  correspond to coactions of the function algebra.



 }

%\tableofcontents

\section{Introduction and preliminaries}
\label{section:introduction}

Actions of quantum groupoids that simultaneously generalize
actions of quantum groups and actions of groupoids have been
studied in various settings, including that of weak Hopf
algebras or finite quantum groupoids
\cite{kornel:weak-hopf,kornel:galois}, Hopf algebroids or
algebraic quantum groupoids
\cite{bohm:actions,kadison:actions}, and Hopf-von Neumann
bimodules or measured quantum groupoids
\cite{enock:action,enock:2,vallin:actions-finite}.  In this
article, we introduce and investigate coactions of Hopf
$C^{*}$-bimodules or reduced locally compact quantum
groupoids within the framework  developed in
\cite{timmermann:fiber, timmermann:cpmu-hopf}.

In the first part of this article, we construct reduced
crossed products and dual coactions, and show that the
bidual of a fine coaction is Morita equivalent to the
initial coaction. These constructions apply to pairs of Hopf
$C^{*}$-bimodules that appear as the left and the right leg
of a (weak) $C^{*}$-pseudo-Kac system, which consists of a
$C^{*}$-pseudo-multiplicative unitary
\cite{timmermann:cpmu-hopf} and an additional symmetry.  We
associate such a $C^{*}$-pseudo-Kac system to every groupoid
and to every compact $C^{*}$-quantum groupoid and expect
that the same can be done for every reduced locally compact
quantum groupoid once this concept has been defined
properly.  The constructions in this part generalize
corresponding constructions of Baaj and Skandalis
\cite{baaj:2} for coactions of Hopf $C^{*}$-algebras.

Coactions of the two Hopf $C^{*}$-bimodules associated to a
locally compact Hausdorff groupoid --- the function algebra
on one side and the reduced groupoid algebra on the other
--- are studied in detail in the second part of this
article. We show that actions of the groupoid on continuous
bundles of $C^{*}$-algebras correspond to coactions of the
first Hopf $C^{*}$-bimodule, and that continuous Fell
bundles on $G$  naturally yield coactions of the second Hopf
$C^{*}$-bimodule. Generalizing results of Quigg \cite{quigg}
and Baaj and Skandalis \cite{baaj:1} from groups to
groupoids, we show that if the groupoid is \'etale, every
coaction of the reduced groupoid algebra arises from a Fell
bundle.

This work was supported by the SFB 478 ``Geometrische
Strukturen in der Mathematik''\footnote{funded by the
  Deutsche Forschungsgemeinschaft (DFG)}.

This article is organized as follows.  The first part is
concerned with coactions of Hopf $C^{*}$-bimodules and
associated reduced crossed products. Section
\ref{section:rtp} summarizes the relative tensor product of
$C^{*}$-modules and the fiber product of $C^{*}$-algebras
over $C^{*}$-bases \cite{timmermann:fiber} which are
fundamental to everything that follows, and introduces
coactions of Hopf $C^{*}$-bimodules.  Section
\ref{section:kac} is concerned with $C^{*}$-pseudo-Kac
systems.  Every $C^{*}$-pseudo-Kac system gives rise to two
Hopf $C^{*}$-bimodules, called the legs of the system, which
are dual to each other in a suitable sense. Coactions of
these legs on $C^{*}$-algebras, associated reduced crossed
products, dual coactions and a duality theorem concerning
iterated crossed products are discussed in Section
\ref{section:coactions}.  Section \ref{section:kac-groupoid}
gives the construction of the $C^{*}$-pseudo-Kac system of a
locally compact Hausdorff groupoid $G$. The associated Hopf
$C^{*}$-bimodules are the function algebra on one side and
the reduced groupoid $C^{*}$-algebra of $G$ on the other
side.  The second part of the article relates coactions of
these Hopf $C^{*}$-bimodules to well-known
notions. Section \ref{section:actions} shows that actions of
a  groupoid $G$ on continuous bundles of $C^{*}$-algebras
correspond to certain fine coactions of the function algebra
of $G$.  Section \ref{section:prelim} contains
preliminaries on Fell bundles, their morphisms and
multipliers. Section \ref{section:fell} shows that
continuous Fell bundles on $G$ give rise to coactions
of the reduced groupoid $C^{*}$-algebra of $G$, and section
\ref{section:etale} gives a reverse construction that
associates to every sufficiently nice coaction of the
groupoid algebra a Fell bundle provided that the groupoid $G$ is
\'etale.

\paragraph{Preliminaries} \label{section:preliminaries}

We use the following
notation.  Given a subset $Y$ of a normed space $X$, we
denote by $[Y] \subset X$ the closed linear span of $Y$.
All sesquilinear maps like inner products of Hilbert spaces
are assumed to be conjugate-linear in the first component
and linear in the second one.  Given a Hilbert space $H$, we
use the ket-bra notation and define for each $\xi \in H$
operators $|\xi\rangle \colon \complex \to H$, $\lambda
\mapsto \lambda \xi$, and $\langle\xi|=|\xi\rangle^{*}
\colon H \to \complex$, $\xi' \mapsto \langle
\xi|\xi'\rangle$.  Given a $C^{*}$-algebra $A$ and a
subspace $B \subset A$, we denote by $A \cap B'$ the
relative commutant $\{a \in A \mid [a,B]=0\}$.

We shall make extensive use of (right) Hilbert
$C^{*}$-modules; see \cite{lance}.  In particular, we use
the internal tensor product and the KSGNS-construction. Let
$E$ be a Hilbert $C^{*}$-module over a $C^{*}$-algebra $A$,
let $F$ be a Hilbert $C^{*}$-module over a $C^{*}$-algebra
$B$, and let $\phi \colon A \to {\cal L}(F)$ be a completely
positive map. We denote by $E \tr_{\phi} F$ the Hilbert
$C^{*}$-module over $B$ which is the closed linear span of
elements $\eta \tr_{\phi} \xi$, where $\eta \in E$ and $\xi
\in F$ are arbitrary, and $\langle \eta \tr_{\phi} \xi|\eta'
\tr_{\phi} \xi'\rangle = \langle
\xi|\phi(\langle\eta|\eta'\rangle)\xi'\rangle$ and $(\eta
\tr_{\phi} \xi)b=\eta \tr_{\phi} \xi b$ for all $\eta,\eta'
\in E$, $\xi,\xi' \in F$, and $b \in B$. If $\phi$ is a
$*$-homomorphism, this is the usual internal tensor product;
if $F=B$, this is the KSGNS-construction.  If $S\in {\cal
  L}(E)$ and $T \in {\cal L}(F) \cap \phi(A)'$, then there
exists a unique operator $S \tr_{\phi} T \in {\cal L}(E
\tr_{\phi} E)$ such that $(S \tr_{\phi} T)(\eta \tr_{\phi}
\xi)=S\eta \tr_{\phi} T\xi$ for all $\eta \in E$, $\xi \in
F$; see \cite[Proposition 1.34]{echter}.  We sloppily write
``$\tr_{A}$'' or ``$\tr$'' instead of ``$\tr_{\phi}$'' if no
confusion may arise.  We also define a flipped product $F
{_{\phi}\tl} E$ as follows. We equip the algebraic tensor
product $F \odot E$ with the structure maps $\langle \xi
\odot \eta | \xi' \odot \eta'\rangle := \langle \xi|
\phi(\langle \eta|\eta'\rangle) \xi'\rangle$, $(\xi \odot
\eta) b := \xi b \odot \eta$, form the separated completion,
and obtain a Hilbert $C^{*}$-module $F {_{\phi}\tl} E$ over
$B$ which is the closed linear span of elements $\xi
{_{\phi}\tl} \eta$, where $\eta \in E$ and $\xi \in F$ are
arbitrary, and $\langle \xi {_{\phi}\tl} \eta|\xi'
{_{\phi}\tl} \eta'\rangle = \langle
\xi|\phi(\langle\eta|\eta'\rangle)\xi'\rangle$ and $(\xi
{_{\phi}\tl} \eta)b=\xi b {_{\phi}\tl} \eta$ for all
$\eta,\eta' \in E$, $\xi,\xi' \in F$, and $b\in B$. Again,
we sloppily write ``${_{A}}\tl$'' or ``$\tl$'' instead of
``${_{\phi}\tl}$'' if no confusion may arise.  Evidently,
there exists a unitary $\Sigma \colon F \tr E \mycong E \tl
F$, $\eta \tr \xi \mapsto \xi \tl \eta$.

\section{Hopf $C^{*}$-bimodules and  coactions}

\label{section:rtp}

A groupoid differs from a group in the non-triviality of its
unit space. In almost every approach to quantum groupoids,
the unit space is replaced by a nontrivial algebra, and a
relative tensor product of modules and a fiber product of
algebras over that algebra become fundamentally important. We
shall use the corresponding constructions for
$C^{*}$-algebras introduced in \cite{timmermann:fiber} and
briefly summarize the main definitions and results below.
For additional details and motivation, see
\cite{timmermann:fiber,timmermann:cpmu-hopf}.

\paragraph{The relative tensor product}
A {\em $C^{*}$-base} is a triple $\cbasesb$ consisting of a
Hilbert space $\frakK$ and two commuting nondegenerate
$C^{*}$-algebras $\frakB,\frakBo \subseteq
\mathcal{L}(\frakK)$.  It should be thought of
as a $C^{*}$-algebraic counterpart to pairs consisting of a
von Neumann algebra and its commutant.  Let
$\frakb=\cbasesb$ be a $C^{*}$-base. Its {\em opposite} is
the $C^{*}$-base $\frakbo:=\cbaseosb$.

A {\em $C^{*}$-$\frakb$-module} is a pair
$H_{\alpha}=(H,\alpha)$, where $H$ is a Hilbert space and
$\alpha \subseteq \mathcal{L}(\frakK,H)$ is a closed
subspace satisfying $[\alpha \frakK]=H$, $[\alpha \frakB]=
\alpha$, and $[\alpha^{*}\alpha] = \frakB \subseteq
\mathcal{L}(\frakK)$.  If $H_{\alpha}$ is a
$C^{*}$-$\frakb$-module, then $\alpha$ is a Hilbert
$C^{*}$-module over $B$ with inner product $(\xi,\xi')
\mapsto \xi^{*}\xi'$ and there exist isomorphisms
\begin{align} \label{eq:rtp-iso}
  \alpha \tr \frakK &\to H, \ \xi \tr \zeta
  \mapsto \xi \zeta, & \frakK \tl \alpha &\to H, \
  \zeta \tl \xi \mapsto \xi\zeta,
\end{align}
and a nondegenerate representation
\begin{align*}
  \rho_{\alpha} \colon \frakBo \to \mathcal{L}(H), \quad
  \rho_{\alpha}(b^{\dagger})(\xi \zeta)= \xi b^{\dagger} \zeta
  \quad \text{for all } b^{\dagger} \in \frakBo, \xi \in
  \alpha, \zeta \in \frakK.
\end{align*}
A {\em semi-morphism} between $C^{*}$-$\frakb$-modules
$H_{\alpha}$ and $K_{\beta}$ is an operator $T \in
\mathcal{L}(H,K)$ satisfying $T\alpha \subseteq \beta$. If
additionally $T^{*}\beta \subseteq \alpha$, we call $T$ a
{\em morphism}. We denote the set of all (semi-)morphisms by
$\mathcal{ L}_{(s)}(H_{\alpha},K_{\beta})$. If $T \in
\mathcal{ L}_{s}(H_{\alpha},K_{\beta})$, then
$T\rho_{\alpha}(b^{\dagger}) = \rho_{\beta}(b^{\dagger})T$
for all $b^{\dagger} \in \frakBo$, and if additionally $T
\in \mathcal{ L}(H_{\alpha},K_{\beta})$, then left
multiplication by $T$ defines an operator in $\mathcal{
  L}(\alpha,\beta)$ which we again denote by $T$.

We shall use the following notion of $C^{*}$-bi- and
$C^{*}$-$n$-modules.  Let $\frakb_{1},\ldots,\frakb_{n}$ be
$C^{*}$-bases, where
$\frakb_{i}=(\frakK_{i},\frakB_{i},\frakBo_{i})$ for each
$i$.  A {\em
  $C^{*}$-$(\frakb_{1},\ldots,\frakb_{n})$-module} is a
tuple $(H,\alpha_{1},\ldots,\alpha_{n})$, where $H$ is a
Hilbert space and $(H,\alpha_{i})$ is a
$C^{*}$-$\frakb_{i}$-module for each $i$ such that
$[\rho_{\alpha_{i}}(\frakBo_{i})\alpha_{j}]=\alpha_{j}$
whenever $i\neq j$.  In the case $n=2$, we abbreviate
$\aHb:=(H,\alpha,\beta)$.  If
$(H,\alpha_{1},\ldots,\alpha_{n})$ is a
$C^{*}$-$(\frakb_{1},\ldots,\frakb_{n})$-module, then
$[\rho_{\alpha_{i}}(\frakBo_{i}),\rho_{\alpha_{j}}(\frakBo_{j})]=0$
whenever $i \neq j$.  The set of {\em (semi-)morphisms}
between $C^{*}$-$(\frakb_{1},\ldots,\frakb_{n})$-modules
$\mathcal{H}=(H,\alpha_{1},\ldots,\alpha_{n})$ and
$\mathcal{K}=(K,\beta_{1},\ldots,\beta_{n})$ is $\mathcal{
  L}_{(s)}(\mathcal{H}, \mathcal{K}) := \bigcap_{i=1}^{n}
\mathcal{ L}_{(s)}(H_{\alpha_{i}},K_{\beta_{i}}) \subseteq
\mathcal{L}(H,K)$.

Let $\frakb=\cbasesb$ be a $C^{*}$-base, $H_{\beta}$
$C^{*}$-$\frakb$-module, and $K_{\gamma}$ a
$C^{*}$-$\frakbo$-module. The {\em relative tensor product}
of $H_{\beta}$ and $K_{\gamma}$ is the Hilbert space
\begin{align*}
  \HtensorK:=\beta \tr \frakK \tl \gamma.
\end{align*}
It is spanned by elements $\xi \tr \zeta \tl \eta$, where
$\xi \in \beta$, $\zeta \in \frakK$, $\eta \in \gamma$, and
the inner product is given by
$\langle \xi \tr \zeta \tl \eta|\xi' \tr \zeta' \tl
\eta'\rangle = \langle \zeta | \xi^{*}\xi' \eta^{*}\eta'
\zeta'\rangle = \langle \zeta|\eta^{*}\eta' \xi^{*}\xi'
\zeta'\rangle$ for all $\xi,\xi' \in \beta$, $\zeta,\zeta'
\in \frakK$, $\eta,\eta' \in \gamma$.  Obviously, there
exists a unitary {\em flip}
\begin{align*}
  \Sigma \colon \HtensorK \to \KtensorH, \quad \xi \tr \zeta
  \tl \eta \mapsto \eta \tr \zeta \tl \xi.
\end{align*}
Using the unitaries in \eqref{eq:rtp-iso} on $H_{\beta}$ and
$K_{\gamma}$, respectively, we shall make the following
identifications without further notice:
\begin{gather*}
  H {_{\rho_{\beta}} \tl} \gamma \cong \HtensorK
  \cong \beta \tr_{\rho_{\gamma}} K, \quad \xi \zeta \tl \eta
  \equiv \xi \tr \zeta \tl \eta \equiv \xi \tr \eta\zeta.
\end{gather*}
For all $S \in \rho_{\beta}(\frakBo)'$ and $T \in
\rho_{\gamma}(\frakB)'$, we have
operators
\begin{align*}
  S \tl \Id &\in \mathcal{L}(H
  {_{\rho_{\beta}}\tl} \gamma) = \mathcal{L}(\HtensorK), &
  \Id \tr T &\in \mathcal{L}(\beta
  \tr_{\rho_{\gamma}} K) = \mathcal{L}(\HtensorK).
\end{align*}
If $S \in \mathcal{ L}_{s}(H_{\beta})$ or $T \in \mathcal{
  L}_{s}(K_{\gamma})$, then $(S \tl \Id)(\xi
\tr \eta\zeta) = S\xi \tr
\eta\zeta$ or $(\Id \tr T)(\xi\zeta \tl
\eta) = \xi\zeta \tl T\eta$, respectively,
for all $\xi \in \beta$, $\zeta \in
\frakK$, $\eta \in \gamma$, so that we can define
\begin{multline*}
  S \btensor T :=(S \tl \Id) (\Id \tr T) = (\Id \tr T)(S \tl
  \Id) \in \mathcal{ L}(\HtensorK)  \\
  \text{ for all } (S,T) \in \left(\mathcal{ L}_{s}(H_{\beta})
    \times \rho_{\gamma}(\frakB)'\right) \cup
  \left(\rho_{\beta}(\frakBo)' \times \mathcal{
      L}_{s}(K_{\gamma})\right).
\end{multline*}

For each $\xi \in \beta$ and $\eta \in \gamma$, there exist
bounded linear operators
\begin{align*}
  |\xi\rangle_{{1}} \colon K &\to \HtensorK, \ \omega
  \mapsto \xi \tr \omega, & 
  |\eta\rangle_{{2}} \colon H &\to \HtensorK, \ \omega
  \mapsto \omega \tl \eta,
\end{align*}
whose adjoints $ \langle \xi|_{1}:=|\xi\rangle_{1}^{*}$ and
$\langle\eta|_{{2}} := |\eta\rangle_{{2}}^{*}$ are given by
\begin{align*}
  \langle \xi|_{1}\colon \xi' \tr
  \omega &\mapsto \rho_{\gamma}(\xi^{*}\xi')\omega, &
  \langle\eta|_{{2}} \colon \omega
  \tl\eta' &\mapsto \rho_{\beta}( \eta^{*}\eta')\omega.
\end{align*}
We write $\kbeta{1} := \{ |\xi\rangle_{{1}} \,|\, \xi \in
\beta\} \subseteq \mathcal{ L}(K,\HtensorK)$ and similarly
define $\bbeta{1}$, $\kgamma{2}$, $\bgamma{2}$.

Let $\mathcal{H}=(H,\alpha_{1},\ldots,\alpha_{m},\beta)$ be
a $C^{*}$-$(\fraka_{1},\ldots,\fraka_{m},\frakb)$-module and let
$\mathcal{K}=(K,\gamma,\delta_{1}$, $\ldots,\delta_{n})$ be a
$C^{*}$-$(\frakbo,\frakc_{1},\ldots,\frakc_{n})$-module,
where $\fraka_{i}=(\frakH_{i},\frakA_{i},\frakAo_{i})$ and
$\frakc_{j}=(\frakL_{j},\frakC_{j},\frakCo_{j})$ are
$C^{*}$-bases for all $i,j$. We define
\begin{align*}
  \alpha_{i} \lt \gamma &:= [\kgamma{2}\alpha_{i}] \subseteq
  \mathcal{ L}(\frakH_{i}, \HtensorK), & \beta \rt
  \delta_{j} &:= [\kbeta{1}\delta_{j}] \subseteq \mathcal{
    L}(\frakL_{j}, \HtensorK)
\end{align*}
for all $i,j$.  Then $(\HtensorK,
\alpha_{1} \lt \gamma, \ldots, \alpha_{m} \lt \gamma, \beta
\rt \delta_{1},\ldots, \beta \rt \delta_{n})$ is a
$C^{*}$-$(\fraka_{1},\ldots,\fraka_{m}$, $
\frakc_{1},\ldots,\frakc_{n})$-module, called the {\em
  relative tensor product} of $\mathcal{H}$ and
$\mathcal{K}$ and denoted by $ \mathcal{H} \btensor
\mathcal{K}$. For all $i,j$ and  $a^{\dagger} \in \frakAo_{i}$, $c^{\dagger} \in \frakCo_{j}$,
\begin{align*}
  \rho_{(\alpha_{i} \lt \gamma)}(a^{\dagger}) &=
  \rho_{\alpha_{i}}(a^{\dagger}) \btensor \Id, &
  \rho_{(\beta \rt \delta_{j})}(c^{\dagger}) &= \Id \btensor
  \rho_{\delta_{j}}(c^{\dagger}).
\end{align*}

The relative tensor product is functorial in the following sense. Let
$\tilde{\mathcal{H}}=(\tilde H,\tilde
\alpha_{1},\ldots,\tilde \alpha_{m},\tilde \beta)$ be a
$C^{*}$-$(\fraka_{1},\ldots,\fraka_{m},\frakb)$-module,
$\tilde{\mathcal{K}}=(\tilde K,\tilde \gamma$, $\tilde
\delta_{1},\ldots,\tilde \delta_{n})$ a
$C^{*}$-$(\frakbo,\frakc_{1},\ldots,\frakc_{n})$-module, and
$S \in \mathcal{L}_{(s)}(\mathcal{H},\tilde{\mathcal{H}})$,
$T \in
\mathcal{L}_{(s)}(\mathcal{K},\tilde{\mathcal{K}})$. Then
there exists a unique operator $S \btensor T \in
\mathcal{L}_{(s)}(\mathcal{H} \btensor \mathcal{K},
\tilde{\mathcal{H}} \btensor \tilde{\mathcal{K}})$
satisfying $(S \btensor T)(\xi \tr \zeta \tl \eta) = S\xi
\tr \zeta \tl T\eta$ for all $\xi \in \beta, \, \zeta \in
\frakK, \, \eta \in \gamma$.

Finally, the relative tensor product is associative in the following sense.
Let $\frakd,\frake_{1},\ldots,\frake_{l}$ be $C^{*}$-bases,
$\hat{\mathcal{K}}
=(K,\gamma,\delta_{1},\ldots,\delta_{n},\epsilon)$ a
$C^{*}$-$(\frakbo,\frakc_{1},\ldots,\frakc_{n}$,
$\frakd)$-module and
$\mathcal{L}=(L,\phi$, $\psi_{1},\ldots,\psi_{l})$ a
$C^{*}$-$(\frakdo,\frake_{1},\ldots,\frake_{l})$-module. Then
there exists a canonical isomorphism
  \begin{align} \label{eq:rtp-associative}
    a_{\mathcal{ H},\mathcal{ K},\mathcal{ L}} \colon (\HtensorK) \rtensor{\beta \rt
    \epsilon}{\frakd}{\phi} L \to \beta
  \tr_{\rho_{\gamma}} K {_{\rho_{\epsilon}}\tl} \, \phi
  \to H \rtensor{\beta}{\frakb}{\gamma \lt \phi} (K
  \rtensor{\epsilon}{\frakd}{\phi} L)
\end{align}
which is an isomorphism of
$C^{*}$-$(\fraka_{1},\ldots,\fraka_{m},\frakc_{1},\ldots,\frakc_{n},\frake_{1},\ldots,\frake_{l})$-modules
$(\mathcal{H} \btensor \hat{\mathcal{K}}) \tensor{\frakd}
\mathcal{L} \to \mathcal{H} \btensor (\hat{\mathcal{K}}
\tensor{\frakd} \mathcal{L})$. From now on, we identify the
Hilbert spaces in \eqref{eq:rtp-associative} and denote them
by $\HtensorK \rtensor{\epsilon}{\frakd}{\phi} L$.

\paragraph{The fiber product of $C^{*}$-algebras}
Let $\frakb_{1},\ldots,\frakb_{n}$ be $C^{*}$-bases, where
$\frakb_{i}=(\frakK_{i},\frakB_{i},\frakBo_{i})$ for each $i$.  A {\em
  (nondegenerate) $C^{*}$-$(\frakb_{1},\ldots,\frakb_{n})$-algebra} consists
of a $C^{*}$-$(\frakb_{1},\ldots,\frakb_{n})$-mod\-ule
$(H,\alpha_{1},\ldots,\alpha_{n})$ and a (nondegenerate) $C^{*}$-algebra $A
\subseteq \mathcal{L}(H)$ such that $\rho_{\alpha_{i}}(\frakBo_{i})A$ is
contained in $A$ for each $i$. We shall only be interested in the cases
$n=1,2$, where we abbreviate $A^{\alpha}_{H}:=(H_{\alpha},A)$,
$A^{\alpha,\beta}_{H}:=(\aHb,A)$.  Given a
$C^{*}$-$(b_{1},\ldots,b_{n})$-algebra ${\cal
  A}=((H,\alpha_{1},\ldots,\alpha_{n}),A)$, we identify $M(A)$ with a
$C^{*}$-subalgebra of ${\cal L}([AH]) \subseteq {\cal L}(H)$ and obtain
$C^{*}$-$(b_{1},\ldots,b_{n})$-algebra $M({\cal
  A})=((H,\alpha_{1},\ldots,\alpha_{n}),M(A))$.

We need several natural notions of a morphism. Let $\mathcal{ A}=(\mathcal{
  H},A)$ and $\mathcal{ C}=(\mathcal{ K},C)$ be
$C^{*}$-$(\frakb_{1},\ldots,\frakb_{n})$-algebras, where $\mathcal{
  H}=(H,\alpha_{1},\ldots,\alpha_{n})$ and $\mathcal{
  K}=(K,\gamma_{1},\ldots,\gamma_{n})$.  A $*$-homo\-morphism $\pi \colon A
\to C$ is called a {\em jointly (semi-)normal morphism} or briefly {\em
  (semi-)morphism} from $\mathcal{ A}$ to $\mathcal{ C}$ if $[\mathcal{
  L}^{\pi}_{(s)}(\mathcal{ H},\mathcal{ K})\alpha_{i}] = \gamma_{i}$ for each
$i$, where
\begin{align*}
  \mathcal{ L}^{\pi}_{(s)}(\mathcal{ H},\mathcal{ K}) &= \{
  T \in \mathcal{ L}_{(s)}(\mathcal{ H},\mathcal{ K}) \mid
  Ta=\pi(a)T \text{ for all } a \in A\}.
\end{align*}
One easily verifies that every (semi-)morphism $\pi$ between
$C^{*}$-$\frakb$-algebras $A_{H}^{\alpha}$ and
$C_{K}^{\gamma}$ satisfies
$\pi(\rho_{\alpha}(b^{\dagger}))=\rho_{\gamma}(b^{\dagger})$
for all $b^{\dagger} \in \frakBo$.

We construct a fiber product of $C^{*}$-algebras over
$C^{*}$-bases as follows. Given Hilbert spaces $H,K$, a
closed subspace $E \subseteq {\cal L}(H,K)$, and a
$C^{*}$-algebra $A \subseteq {\cal L}(H)$, we define a
$C^{*}$-algebra
\begin{align*}
  \Ind_{E}(A) &:= \{ T \in {\cal L}(K) \mid TE \subseteq
  [EA] \text{ and } T^{*}E
  \subseteq [EA]\} \subseteq {\cal L}(K).
\end{align*}
Let $\frakb$ be a $C^{*}$-base, $A_{H}^{\beta}$ a
$C^{*}$-$\frakb$-algebra, and $B_{K}^{\gamma}$ a
$C^{*}$-$\frakbo$-algebra. The {\em fiber product} of
$A_{H}^{\beta}$ and $B_{K}^{\gamma}$ is the $C^{*}$-algebra
\begin{align*}
  A \fibre{\beta}{\frakb}{\gamma} B := \Ind_{\kbeta{1}}(B)
  \cap \Ind_{\kgamma{2}}(A) \subseteq
  \mathcal{L}(\HtensorK).
\end{align*}
To define coactions, we also need to consider the 
$C^{*}$-algebra
\begin{align*}
  A \fib{\beta}{\frakb}{\gamma} B :=
  \Ind_{[\kbeta{1}B]}(B) \cap \Ind_{\kgamma{2}}(A) \subseteq
  \mathcal{L}(\HtensorK),
\end{align*}
which evidently contains $A \fibre{\beta}{\frakb}{\gamma}
B$.  If $A$ and $B$ are unital, so is $A
\fibre{\beta}{\frakb}{\gamma} B$, but otherwise, $A
\fibre{\beta}{\frakb}{\gamma} B$ and $A
\fib{\beta}{\frakb}{\gamma} B$ may be degenerate.  Clearly,
conjugation by the flip $\Sigma \colon \HtensorK \to
\KtensorH$ yields an isomorphism
  \begin{align*}
    \Ad_{\Sigma} \colon A \fibre{\beta}{\frakb}{\gamma} B
    \to B \fibre{\gamma}{\frakbo}{\beta} A.
  \end{align*}
  If $\fraka,\frakc$ are $C^{*}$-bases,
  $A^{\alpha,\beta}_{H}$ is a
  $C^{*}$-$(\fraka,\frakb)$-algebra and
  $B^{\gamma,\delta}_{K}$ a
  $C^{*}$-$(\frakbo,\frakc)$-algebra, then
  \begin{align*}
    A^{\alpha,\beta}_{H} \bfibre B^{\gamma,\delta}_{K} :=
    ({_{\alpha}H_{\beta}} \btensor {_{\gamma}K_{\delta}},\,
    A \fibre{\beta}{\frakb}{\gamma} B)
\end{align*} is a
  $C^{*}$-$(\fraka,\frakc)$-algebra, called the {\em fiber
    product} of $A^{\alpha,\beta}_{H}$ and
  $B^{\gamma,\delta}_{K}$ \cite[Proposition
  3.18]{timmermann:fiber}, and likewise
  $({_{\alpha}H_{\beta}} \btensor {_{\gamma}K_{\delta}},\, A
  \fib{\beta}{\frakb}{\gamma} B)$ is a
  $C^{*}$-$(\fraka,\frakc)$-algebra.

The construction of the fiber product and of the algebra
above is functorial with respect to (semi-)morphisms
\cite[Theorem 3.2]{timmermann:cpmu-hopf} in the following
sense.

\begin{lemma}
 \label{lemma:fp-c-morphism}
 Let $\frakc$ be a $C^{*}$-base, $\pi$ a semi-morphism of
 $C^{*}$-$\frakb$-algebras $A_{H}^{\beta}$,
 $C^{\lambda}_{L}$, and $\cKd$ a
 $C^{*}$-$(\frakbo,\frakc)$-module. Let 
$   I:=\mathcal{ L}_{s}^{\pi}(H_{\beta},L_{\lambda}) \btensor
   \Id \subseteq {\cal L}(\HtensorK, L
   \rtensor{\lambda}{\frakb}{\gamma} K)$ and
 \begin{align*}
     X&:=(I^{*}I)' \subseteq {\cal L}(\HtensorK), &
     Y&:=(II^{*})' \subseteq {\cal L}(L
   \rtensor{\lambda}{\frakb}{\gamma} K).
   \end{align*}
   \begin{enumerate}
 \item ${\cal X}:=(\HtensorK_{\delta},X)$ and ${\cal Y}:=(L
   \rtensor{\lambda}{\frakb}{\gamma} K_{\delta}, Y)$ are
   $C^{*}$-$\frakc$-algebras.
 \item There exists a semi-morphism $\Ind_{\kgamma{2}}(\pi)
   \colon {\cal X} \to {\cal Y}$ such that
   $(\Ind_{\kgamma{2}}(\pi))(x)z=zx$ for all $x \in X$ and
   $z \in I$.
 \item Assume that $B^{\gamma}_{K}$ is a $C^{*}$-$\frakbo$-algebra.
   Then we have $\AfibreB \subseteq A \fib{\beta}{\frakb}{\gamma} B
   \subseteq X$,
   $(\Ind_{\kgamma{2}}(\pi))(\AfibreB) \subseteq \Cl \bfibre
   \gB$ and $(\Ind_{\kgamma{2}}(\pi))(A
   \fib{\beta}{\frakb}{\gamma} B) \subseteq C
   \fib{\lambda}{\frakb} {\gamma} B$.
 \item $[\kgamma{2}A\bgamma{2}] \subseteq X$ and
   $(\Ind_{\kgamma{2}}(\pi))([\kgamma{2}A\bgamma{2}]) =
   [\kgamma{2}\pi(A)\bgamma{2}]$.
 \end{enumerate}
\end{lemma}
\begin{proof} 
  The existence of the $*$-homomorphism
  $\Ind_{\kgamma{2}}(\pi)$ follows as in \cite[Lemma
  3.1]{timmermann:cpmu-hopf}. The proof of the remaining
  assertions is straightforward; compare \cite[\S
  3.4]{timmermann:fiber} and \cite[\S 3.1]{timmermann:cpmu-hopf}.
\end{proof}
\begin{theorem}
  Let $\fraka,\frakb,\frakc$ be $C^{*}$-bases, $\phi$ a
  (semi-)morphism of $C^{*}$-$(\fraka,\frakb)$-algebras
  $\mathcal{ A}=A^{\alpha,\beta}_{H}$ and $\mathcal{
    C}=C^{\kappa,\lambda}_{L}$, and $\psi$ a (semi-)morphism
  of $C^{*}$-$(\frakbo,\frakc)$-algebras $\mathcal{
    B}=B_{K}^{\gamma,\delta}$ and $\mathcal{
    D}=D^{\mu,\nu}_{M}$. Then there exists a unique
  (semi-)morphism of $C^{*}$-$(\fraka,\frakc)$-algebras
  $\phi \ast \psi$ from $ ({_{\alpha}H_{\beta}} \btensor
  {_{\gamma}K_{\delta}},\, A \fib{\beta}{\frakb}{\gamma}
  B)$ to $ ({_{\kappa}L_{\lambda}} \btensor
  {_{\mu}M_{\nu}},\, C \fib{\lambda}{\frakb}{\mu} D)$ such
  that
  \begin{align*}
    (\phi \ast
    \psi)(x)R = R x \quad \text{for all } x \in \AfibB
    \text{ and }R \in I_{M}J_{H} + J_{L}I_{K},
  \end{align*}
  where $I_{X}=\mathcal{ L}^{\phi}(H,L) \btensor \Id_{X}$,
  $J_{Y}= \Id_{Y} \btensor \mathcal{ L}^{\psi}(K,M)$ for
  $X\in\{K,M\},Y\in\{H,L\}$, and $\phi \ast \psi$
  restricts to a (semi-)morphism from $ A^{\alpha,\beta}_{H}
  \bfibre B^{\gamma,\delta}_{K}$ to $C^{\kappa,\lambda}_{L}
  \bfibre D^{\mu,\nu}_{M}$.
\end{theorem}
\begin{proof}
  This follows from Lemma \ref{lemma:fp-c-morphism} and
  a similar argument as in the proof of \cite[Theorem
  3.13]{timmermann:fiber}.
\end{proof}

Unfortunately, the fiber product need not be associative,
but in our applications, it will only appear as the target
of a comultiplication or coaction whose coassociativity will compensate
the non-associativity of the fiber product.

\paragraph{Hopf $C^{*}$-bimodules and coactions}
The notion of a Hopf $C^{*}$-bimodule was introduced in
\cite{timmermann:cpmu-hopf}.
\begin{definition}
  Let $\frakb=\cbasesb$ be a $C^{*}$-base.  A {\em Hopf
    $C^{*}$-bimodule over $\frakb$} is a
  $C^{*}$-$(\frakbo,\frakb)$-algebra $A^{\beta,\alpha}_{H}$
  together with a morphism $\Delta$ from
  $A^{\beta,\alpha}_{H}$ to $A^{\beta,\alpha}_{H} \bfibre
  A^{\beta,\alpha}_{H}$ that is coassociative in the sense
  that $(\delta \ast \Id) \circ \delta=(\Id \ast
  \Delta) \circ \delta$ as maps from $A$ to $\mathcal{L}(H
  \rtensor{\alpha}{\frakb}{\beta} H
  \rtensor{\alpha}{\frakb}{\beta} H)$.

    Let $(\cA,\Delta)$ be a Hopf $C^{*}$-bimodule, where
    $\cA=A_{H}^{\beta,\alpha}$.  A {\em coaction} of
    $(\cA,\Delta)$ consists of a $C^{*}$-$\frakb$-algebra
    $C_{K}^{\gamma}$ and a semi-morphism $\delta$ from
    $(K_{\gamma},C)$ to $(K_{\gamma} \btensor
    {_{\beta}H_{\alpha}}, C \fib{\gamma}{\frakb}{\beta} A)$
    such that $(\delta \ast \Id) \circ \delta=(\Id
    \ast \Delta) \circ \delta$ as maps from $C$ to $
    \mathcal{L}(K \rtensor{\gamma}{\frakb}{\beta} \Hrange)$.
    We call such a coaction $(C^{\gamma}_{K},\delta)$
    \begin{itemize}
  \item {\em left-full} if
    $[\delta(C)\kgamma{1}A]=[\kgamma{1}A]$, and {\em
      right-full} if   $[ \delta(C)\kbeta{2}] = [
    \kbeta{2}C]$;
  \item {\em fine} if $\delta$ is injective, a morphism, and
    right-full, and if $[\rho_{\gamma}(\frakBo)C]=C$;
  \item {\em very fine} if it is fine and if $\delta^{-1}\colon
    \delta(C) \to C$ is a morphism of
    $C^{*}$-$\frakb$-algebras from $(\rHrange_{\alpha},
    \delta(C))$ to $(K_{\gamma},C)$.
    \end{itemize}
    A {\em morphism} between coactions
    $(C^{\gamma}_{K},\delta_{C})$ and
    $(D^{\epsilon}_{L},\delta_{D})$ is a semi-morphism $\rho$
    of $C^{*}$-$\frakb$-algebras from $C^{\gamma}_{K}$ to
    $M(D)^{\epsilon}_{L}$ that is nondegenerate in the sense
    that $[\rho(C)D]=D$ and equivariant in the sense that
    $\delta_{D}(d) \cdot (\rho \ast \Id)(\delta_{C}(c)) =
    \delta_{D}(d\rho(c))$ for all $d \in D$, $c \in C$.  We
    denote the category of all coactions of $(\cA,\Delta)$
    by $\coact_{(\cA,\Delta)}$.
\end{definition}
Examples of Hopf $C^{*}$-bimodules and coactions will be
discussed in detail in Sections \ref{section:kac-groupoid},
\ref{section:actions}, and \ref{section:fell}.

\section{Weak $C^{*}$-pseudo-Kac systems}

\label{section:kac}

To form a reduced crossed product for a coaction of a Hopf
$C^{*}$-bimodule $(\cA,\Delta)$ and to equip this reduced
crossed product with a dual coaction, one needs a second
Hopf $C^{*}$-bimodule $(\hcA,\hDelta)$ that is dual to
$(\cA,\Delta)$ in a suitable sense. We shall see that a good
notion of duality is that $(\cA,\Delta)$ and
$(\hcA,\hDelta)$ are the legs of a weak $C^{*}$-pseudo-Kac
system, which is a generalization of the balanced
multiplicative unitaries and Kac systems introduced by Baaj
and Skandalis \cite{baaj:10,baaj:2}.

\paragraph{$C^{*}$-pseudo-multiplicative unitaries} 
A weak $C^{*}$-pseudo-Kac system consists of a well-behaved
$C^{*}$-pseudo-multiplicative unitary $V$ and a
symmetry $U$ satisfying a number of axioms. Before we state
these axioms, we recall from \cite{timmermann:cpmu-hopf} the
notion of a $C^{*}$-pseudo-multiplicative unitary and the
construction of the associated Hopf $C^{*}$-bimodules.

Let $\frakb$ be a $C^{*}$-base.  A {\em
  $C^{*}$-pseudo-multiplicative unitary} over $\frakb$
consists of a $C^{*}$-$(\frakbo,\frakb,\frakbo)$-module
$(H,\hbeta,\alpha,\beta)$ and a unitary $V \colon \Hsource
\to \Hrange$ such that
  \begin{align} \label{eq:pmu-intertwine} 
 V(\alpha \lt \alpha) &= \alpha \rt
\alpha, & V(\hbeta \rt \beta) &= \hbeta \lt \beta, &
V(\hbeta \rt \hbeta) &= \alpha \rt \hbeta, & V(\beta \lt
\alpha) &= \beta \lt \beta
  \end{align}
  in ${\cal L}(\frakK,\Hrange)$ and 
  $V_{12}V_{13}V_{23}=V_{23}V_{12}$ in the sense that the
  following diagram commutes,
  \begin{gather} \label{eq:pmu-pentagon}
     \xymatrix@R=12pt@C=20pt{ {\Hone}
        \ar[r]^{V_{12}} \ar[d]^{V_{23}} & {\Htwo}
        \ar[r]^{V_{23}} & {\Hthree,} \\ {\Hfive}
        \ar[rr]^{V_{13}} & & {\Hfour} \ar[u]^{V_{12}}
        } 
  \end{gather} 
  where $V_{ij}$ is the leg notation for the operator that
  acts like $V$ on the $i$th and $j$th factor in the
  relative tensor product; see \cite{timmermann:cpmu-hopf}.

  Let $V$ be a $C^{*}$-pseudo-multiplicative unitary as above, let
  \begin{align*}
    \hA= \hA_{V} &= [\bbeta{2}V\kalpha{2}] \subseteq {\cal
      L}(H), & \hDelta = \hDelta_{V} &\colon \hA \to
    {\cal
      L}(\Hsource), \ \ha \mapsto V^{*}(1 \btensor \ha)V, \\
    A= A_{V} &= [\balpha{1}V\khbeta{1}] \subseteq {\cal
      L}(H), & \Delta= \Delta_{V} &\colon A \to {\cal
      L}(\Hrange), \ a\mapsto V(a \botensor 1)V^{*},
  \end{align*}
  and let $\hcA=\hA^{\alpha,\hbeta}_{H}$,
  $\cA=A^{\beta,\alpha}_{H}$. We say that $V$ is {\em
    well-behaved} if $(\hcA,\hDelta)$ and $(\cA,\Delta)$ are
  Hopf $C^{*}$-bimodules.  This happens for example if $V$
  is {\em regular} in the sense that
  $[\balpha{1}V\kalpha{2}]=[\alpha\alpha^{*}] \subseteq
  {\cal L}(H)$ \cite[Theorem 5.7]{timmermann:cpmu-hopf}.

The {\em opposite} of $V$ is the
  $C^{*}$-pseudo-multi\-plicative unitary
 \begin{align*}
   V^{op}:=\Sigma V^{*} \Sigma
    \colon H \rtensor{\beta}{\frakbo}{\alpha} H
    \xrightarrow{\Sigma} \Hrange \xrightarrow{V^{*}}
    \Hsource \xrightarrow{\Sigma} H
    \rtensor{\alpha}{\frakb}{\hbeta} H.
 \end{align*}
 If $V$ is well-behaved or regular, then the same is true
 for $V^{op}$, and then
\begin{align} \label{eq:pmu-op} \hA_{V^{op}} &= A_{V}, &
  \hDelta_{V^{op}} &= \Ad_{\Sigma} \circ \Delta_{V}, &
  A_{V^{op}} &= \hA_{V}, & \Delta_{V^{op}} &=
  \Ad_{\Sigma} \circ \hDelta_{V}.
\end{align}

\paragraph{Balanced $C^{*}$-pseudo-multiplicative unitaries} 
Let $(H,\halpha,\hbeta,\alpha,\beta)$ be a
$C^{*}$-$(\frakb,\frakbo,\frakb,\frakbo)$-module and $U \in
{\cal L}({_{\halpha}H_{\hbeta}},{_{\alpha}H_{\beta}})$ a
symmetry, that is, $U=U^{*}=U^{-1}$. Then
$U\halpha=\alpha$, $U\hbeta=\beta$, $U\alpha=\halpha$,
$U\beta=\hbeta$, and the following diagram commutes,
\begin{align*}
  \xymatrix@R=15pt@C=60pt{ {\checkHsource} \ar@{<->}[r]^{(1
      \rtensor{}{\frakb^{(\dagger)}}{} U)\Sigma}
    \ar@{<->}[d]_{(U \rtensor{}{\frakb^{(\dagger)}}{}
      1)\Sigma} \ar@{<->}[rd] & {\Hsource} \ar@{<->}[d]^{(1
      \rtensor{}{\frakb^{(\dagger)}}{}
      U)\Sigma} \ar@{<->}[ld] \\
    {\hatHrange} \ar@{<->}[r]_{(U
      \rtensor{}{\frakb^{(\dagger)}}{} 1)\Sigma} & {\Hrange}
  }
\end{align*}
where each arrow can be read in both directions and the
diagonal maps are $U \rtensor{}{\frakb^{(\dagger)}}{} U$. We
adopt the leg notation and write $U_{1}$ for $U
\rtensor{}{\frakb^{(\dagger)}}{} 1$ and $U_{2}$ for $1
\rtensor{}{\frakb^{(\dagger)}}{} U$.  For each $T \in
\mathcal{ L}(\Hsource, \Hrange)$, let
\begin{align*} 
  \widecheck{T} &:=\Sigma (1 \btensor U)T(1 \botensor
  U)\Sigma \colon \checkHsource \to \checkHrange, \\
  \widehat{T} &:= \Sigma (U \btensor 1)T(U \botensor
  1)\Sigma \colon \hatHsource \to \hatHrange.
\end{align*}
Switching from $(\frakb,H,\halpha,\hbeta,\alpha,\beta)$ to
$(\frakbo,H,\beta,\halpha,\hbeta,\alpha)$ or
$(\frakbo,H,\hbeta,\alpha,\beta,\halpha)$, respectively, we
can iterate the assignments $T \mapsto \widecheck{T}$ and $T
\mapsto \widehat{T}$, and obtain
\begin{align} \label{eq:balanced-iterate}
  \widecheck{\widecheck{\widecheck{T}}} &= \widehat{T}, &
  \widecheck{\widecheck{T}} &= (U \btensor U)T(U \botensor
  U) = \widehat{\widehat{T}}, & \widecheck{T} &=
  \widehat{\widehat{\widehat{T}}}.
\end{align}
\begin{definition} A {\em balanced
    $C^{*}$-pseudo-multiplicative unitary $(V,U)$} on a
  $C^{*}$-$(\frakb,\frakbo,\frakb,\frakbo)$-module
  $(H,\halpha,\hbeta,\alpha,\beta)$ consists of a symmetry
  $U \in {\cal
    L}({_{\halpha}H_{\hbeta}},{_{\alpha}H_{\beta}})$ and a
  $C^{*}$-pseudo-multiplicative unitary $V \colon \Hsource
  \to \Hrange$ such that $\checkV$ and $\hatV$ are
  $C^{*}$-pseudo-multiplicative unitaries.
\end{definition}
\begin{remark}
   In the definition above, $(\checkV,U)$ is a
    $C^{*}$-pseudo-multiplicative unitary if and only if
    $(\hatV,U)$ is one because $\checkV=(U \botensor U)\hatV(U
    \btensor U)$.
\end{remark}
Let $(V,U)$ be a balanced $C^{*}$-pseudo-multiplicative
unitary as above.
\begin{remarks}\label{remarks:balanced}
  \begin{enumerate}
  \item One easily verifies that  $(\checkV,U)$,
    $(\hatV,U)$, $(V^{op},U)$ are balanced
    $C^{*}$-pseudo-multiplica\-tive unitaries again. We call
    them the {\em predual}, {\em dual}, and {\em opposite}
    of $(V,U)$, respectively.
\item The relations \eqref{eq:pmu-intertwine} for the
  unitaries $\checkV$, $\hatV$ read as follows:
  \begin{align*}
    \hbeta \lt \hbeta &\xrightarrow{\checkV} \hbeta \rt
    \hbeta, & \halpha \rt \alpha &\xrightarrow{\checkV}
    \halpha \lt \alpha, & \halpha \rt \halpha
    &\xrightarrow{\checkV} \hbeta \rt \halpha, &
    \alpha \lt \hbeta &\xrightarrow{\checkV}\alpha \lt \alpha, \\
    \beta \lt \beta &\xrightarrow{\hatV} \beta \rt \beta, &
    \alpha \rt \halpha &\xrightarrow{\hatV} \alpha \lt
    \halpha, & \alpha \rt \alpha &\xrightarrow{\hatV} \beta
    \rt \alpha, & \halpha \lt \beta &\xrightarrow{\hatV}
    \halpha \lt \halpha,
\end{align*}
where $X \xrightarrow{W} Y$ means $WX=Y$.  They furthermore
imply
     \begin{align*}
       \hbeta \rt \halpha &\xrightarrow{V} \alpha \rt
       \halpha, & \halpha \rt \beta &\xrightarrow{\checkV}
       \hbeta \rt \beta, &
       \alpha \rt \hbeta &\xrightarrow{\hatV} \beta \rt \hbeta, \\
       \halpha \lt \alpha &\xrightarrow{V} \halpha \lt
       \beta, & \beta \lt \hbeta &\xrightarrow{\checkV}
       \beta \lt \alpha, & \hbeta \lt \beta
       &\xrightarrow{\hatV} \hbeta \lt \halpha.
   \end{align*}
 \item The spaces $\hA$ and $A$ are
   contained in $\mathcal{ L}(H_{\halpha})$ since $[
   \hA \halpha ] = [ \bbeta{2}V\kalpha{2}\halpha] = [
   \bbeta{2} \kbeta{2} \halpha] = [
   \rho_{\alpha}(\frakBo)\halpha] = \halpha$ and similarly
   $[ A\halpha]= [ \balpha{1}V\khbeta{1}\halpha] =
   \halpha$.
  \end{enumerate}
\end{remarks} 
\begin{lemma} \label{lemma:balanced}
  $V_{13}V_{23}\checkV_{12}=\checkV_{12}V_{13}$ and
  $\hatV_{23}V_{12}V_{13}=V_{13}\hatV_{23}$, that is,
  the following diagrams commute:
  \begin{gather} \label{eq:balanced-rel1} \smalldiagram
    \xymatrix@C=25pt@R=10pt{ {(\checkHsource)
        \rtensor{\hbeta \lt \hbeta}{\frakbo}{\alpha} H}
      \ar[r]^{V_{\leg{13}}} \ar[d]_{\checkV_{\leg{12}}} &
      {(\checkHsource) \rtensor{\alpha \lt
          \hbeta}{\frakb}{\beta} H}
      \ar[r]^{\checkV_{\leg{12}}} & {(\Hsource)
        \rtensor{\alpha \lt \alpha}{\frakb}{\beta} H,} \\
      {\Hone} \ar[rr]^{V_{\leg{23}}} && {H
        \rtensor{\hbeta}{\frakbo}{\alpha\rt \alpha}
        (\Hrange)} \ar[u]_{V_{\leg{13}}} } \\
 \label{eq:balanced-rel2}
    \xymatrix@C=25pt@R=15pt{ {H
        \rtensor{\hbeta}{\frakbo}{\alpha\rt \alpha}
        (\hatHsource)} \ar[r]^{\hatV_{\leg{23}}}
      \ar[d]_{V_{\leg{13}}} & {H
        \rtensor{\hbeta}{\frakbo}{\beta \rt \alpha}
        (\hatHrange)} \ar[r]^{V_{\leg{13}}} & {H
        \rtensor{\alpha}{\frakb}{\beta \rt \beta} (\hatHrange).} \\
      {(\Hsource) \rtensor{\alpha \lt \alpha}{\frakb}{\beta} H}
      \ar[rr]^{V_{\leg{12}}} && {\Hthree}
      \ar[u]_{\hatV_{\leg{23}}} }
  \end{gather}
\end{lemma}
\begin{proof} 
  Let $W:= \Sigma V \Sigma$. We insert the relation $\checkV
  = U_{1} W U_{1}$ into the pentagon equation
  $\checkV_{\leg{12}} \checkV_{\leg{13}} \checkV_{\leg{23}}
  = \checkV_{\leg{23}} \checkV_{\leg{12}}$ and obtain
  $U_{\leg{1}} W_{\leg{12}} U_{1} \cdot U_{1} W_{\leg{13}}
  U_{1} \cdot \checkV_{\leg{23}} = \checkV_{\leg{23}} \cdot
  U_{\leg{1}} W_{\leg{12}} U_{\leg{1}}$ and hence
  $W_{\leg{12}} W_{\leg{13}} \checkV_{\leg{23}} =
  \checkV_{\leg{23}} W_{\leg{12}}$.  We conjugate both sides
  of this equation by the automorphism
  $\Sigma_{\leg{23}}\Sigma_{\leg{12}}$, which amounts to
  renumbering the legs of the operators according to the
  permutation $(1,2,3) \mapsto (2,3,1)$, and obtain
  $V_{\leg{13}}V_{\leg{23}} \checkV_{\leg{12}} =
  \checkV_{\leg{12}} V_{\leg{13}}$.  A similar calculation
  shows that $\hatV_{23}V_{12}V_{13}=V_{13}\hatV_{23}$.
\end{proof}
\begin{proposition} \label{proposition:balanced-legs}
 $\hA_{\checkV} = U A_{V}U$,
  $\hDelta_{\checkV} = \Ad_{(U \btensor U)} \circ
    \Delta_{V} \circ \Ad_{U}$, $A_{\checkV} = \hA_{V}$, $
    \Delta_{\checkV} = \hDelta_{V}$ and
 $A_{\hatV} =
U \hA_{V}U$, $\Delta_{\hatV} = \Ad_{(U
      \botensor U)} \circ \hDelta_{V} \circ \Ad_{U}$,
    $\hA_{\hatV} = A_{V}$, $\hDelta_{\hatV} = \Delta_{V}$.
\end{proposition} 
\begin{proof} By definition, $A_{\checkV} = [ \bhbeta{1}
  \Sigma U_{\leg{2}} V U_{\leg{2}} \Sigma \khalpha{1} ] = [
  \langle U \hbeta|_{\leg{2}} V | U\halpha\rangle_{\leg{2}}
  ] = [ \bbeta{2} V \kalpha{2}] = \hA_{V}$.  Next, let
  $\ha=\langle \xi'|_{\leg{2}}V|\xi\rangle_{\leg{2}} \in
  \hA_{V}$, where $\xi' \in \beta$, $\xi \in \alpha$.   The
  following diagram
 \begin{gather*}\smalldiagram \hspace{-3pt}
   \xymatrix@R=10pt@C=10pt{ {\checkHrange}
     \ar[r]_{\checkV^{*}} \ar[dd]^{|\xi\rangle_{\leg{3}}} &
     {\checkHsource} \ar[r]_{\ha \btensor 1}
     \ar[d]^{|\xi\rangle_{\leg{3}}} &
     {\checkHsource} \ar[r]_{\checkV} & {\checkHrange} \\
 &
     {(\checkHsource) \rtensor{\hbeta \lt
         \hbeta}{\frakbo}{\alpha} H} \ar[r]_{V_{\leg{13}}} &
     {(\checkHsource) \rtensor{\alpha \lt
         \hbeta}{\frakb}{\beta} H}
     \ar[rd]_{\checkV_{\leg{12}}}
     \ar[u]^{\langle\xi'|_{\leg{3}}} & 
  \\
     {\Hone} \ar[rd]^{V_{\leg{12}}}  \ar[rrr]|{V_{13}V_{23}}
     \ar[ru]_(0.47){\checkV^{*}_{\leg{12}}} &
& &  {\Hfour}      \ar[dd]^{\langle\xi'|_{\leg{3}}}  
     \ar[uu]^{\langle\xi'|_{\leg{3}}} \\
 & {\Htwo}
     \ar[r]^{V_{\leg{23}}} & {\Hthree}
     \ar[ru]^(0.47){V_{\leg{12}}^{*}}
     \ar[d]^{\langle\xi'|_{\leg{3}}}&
\\
     {\Hsource} \ar[r]^{V} \ar[uu]^{|\xi\rangle_{\leg{3}}} &
     {\Hrange} \ar[r]^{1 \rtensorh \ha}
     \ar[u]^{|\xi\rangle_{\leg{3}}} & {\Hrange}
     \ar[r]^{V^{*}} & {\Hsource}}
  \end{gather*} 
  commutes  because diagrams \eqref{eq:balanced-rel1} and
  \eqref{eq:pmu-pentagon} commute. Therefore,
  $\Delta_{\checkV}(\ha) = \checkV(\ha \btensor
  1)\checkV^{*} = V^{*}(1\btensor \ha)V=\hDelta_{V}(\ha)$.
 Since elements of the form like $\ha$ are
dense in $\hA_{V}$, we can conclude
$\Delta_{\checkV}=\hDelta_{\checkV}$.  
The proof of the remaining assertions is similar.
\end{proof} 
\begin{corollary}
If $V$
is well-behaved, then also
 $\checkV$ and $\hatV$ are well-behaved.
\end{corollary}

\paragraph{Weak $C^{*}$-pseudo-Kac systems}
Let $(V,U)$ as above.
\begin{lemma} \label{lemma:weak-kac} 
 For each $\ha \in \hA$ and $a \in A$, we have equivalences
    \begin{gather*}
      \begin{aligned}
        (1 \botensor \ha)\hatV=\hatV(1 \btensor \ha) &&
        \Leftrightarrow && (U\ha U \btensor 1)V = V(U \ha U
        \botensor 1) && \Leftrightarrow && [U \ha U,\hA]=0, \\
        (a \botensor 1)\checkV=\checkV(a \btensor 1) 
        &&\Leftrightarrow && (1 \btensor Ua U)V = V(1
        \botensor U a U) && \Leftrightarrow && [Ua U,A]=0.
      \end{aligned}
    \end{gather*}
    These equivalent conditions hold for all $\ha \in \hA$ and $a \in A$ if and
    only if $V_{23}\hatV_{12}=\hatV_{12}V_{23}$ and
    $\checkV_{23}V_{12}=V_{12}\checkV_{23}$  in
    the sense that the following diagrams commute:
      \begin{align*}
        \smalldiagram\xymatrix@R=10pt{ {\hatHsource
            \rtensor{\hbeta}{\frakbo}{\alpha} H}
          \ar[r]_{\hatV_{\leg{12}}} \ar[d]^{V_{\leg{23}}} &
          {\hatHrange \rtensor{\hbeta}{\frakbo}{\alpha} H}
          \ar[d]^{V_{\leg{23}}} \\ {\hatHsource
            \rtensor{\alpha}{\frakb}{\beta} H}
          \ar[r]^{\hatV_{\leg{12}}} & {\hatHrange
            \rtensor{\alpha}{\frakb}{\beta} H,} }
 &&
        \smalldiagram\xymatrix@R=10pt{ {\Hsource
            \rtensor{\halpha}{\frakb}{\hbeta} H}
          \ar[r]_{V_{\leg{12}}} \ar[d]^{\checkV_{\leg{23}}}
          & {\Hrange \rtensor{\halpha}{\frakb}{\hbeta} H}
          \ar[d]^{\checkV_{\leg{23}}} \\
          {\Hsource \rtensor{\hbeta}{\frakbo}{\alpha} H}
          \ar[r]^{V_{\leg{12}}} & {\Hrange
            \rtensor{\hbeta}{\frakbo}{\alpha} H.}  }
      \end{align*}
  \end{lemma}
  \begin{proof} 
    This is straightforward, for example, $V_{23}\hatV_{12}
    = \hatV_{12}V_{23}$ holds if and only if $\langle
    \xi'|_{\leg{3}}
    V_{\leg{23}}\hatV_{\leg{12}}|\xi\rangle_{\leg{3}}=
    \langle \xi'|_{\leg{3}} \hatV_{\leg{12}}
    V_{\leg{23}}|\xi\rangle_{\leg{3}}$ for all $\xi \in
    \alpha,\xi' \in \beta$ and hence if and only if $(1
    \botensor \ha) \hatV = \hatV (1 \btensor \ha)$ for all
    $\ha \in \hA$.
\end{proof}
\begin{definition} \label{definition:kac}
  We call $(V,U)$ a {\em weak $C^{*}$-pseudo-Kac system} if
  $V$ is well-behaved and if the equivalent conditions in
  Lemma \ref{lemma:weak-kac} hold, and a {\em
    $C^{*}$-pseudo-Kac-system} if $V,\checkV,\hatV$ are
  regular and  $(\Sigma(1 \btensor
  U)V)^{3}=\Id$, where $\Sigma(1 \btensor U)V \colon \Hsource
  \to \Hsource $.
\end{definition} 
\begin{remark} \label{remark:kac} In leg notation, the
  equation $\big(\Sigma(1 \rtensorh U)V\big)^{3}=1$ reads $(\Sigma U_{\leg{2}} V)^{3} = 1$.  Conjugating by
  $\Sigma$ or $V$, we see that this condition is equivalent
  to the relation $( U_{\leg{2}}V\Sigma)^{3}=1$ and to the
  relation $(V\Sigma U_{\leg{2}})^{3}=1$.
\end{remark}
\begin{lemma} \label{lemma:kac-condition} 
  $(\Sigma U_{\leg{2}}V)^{3}=1$ if and only if $\hatV
  V\checkV = U_{\leg{1}}\Sigma$.
\end{lemma}
\begin{proof} 
Use the relation $U_{\leg{1}} U_{\leg{2}}
  (\Sigma U_{\leg{2}}V)^{3} U_{\leg{2}} \Sigma = \Sigma
  U_{\leg{1}} V U_{\leg{1}} \Sigma \cdot V \cdot \Sigma
  U_{\leg{2}} V U_{\leg{2}} \Sigma = \hatV \cdot V \cdot
  \checkV$.
\end{proof} 

\begin{proposition} \label{proposition:kac-weak} Every
  $C^{*}$-pseudo-Kac system is a weak $C^{*}$-pseudo-Kac
  system.
\end{proposition}
\begin{proof} 
  Let $(V,U)$ be a $C^{*}$-pseudo-Kac system. Then
  $V,\checkV,\hatV$ are regular and therefore
  well-behaved. Using diagrams \eqref{eq:pmu-pentagon} and
  \eqref{eq:balanced-rel1}, we find
  \begin{align*}
   V_{\leg{12}}\checkV_{\leg{12}}\Sigma_{\leg{12}}
  V_{\leg{23}} = V_{\leg{12}} \checkV_{\leg{12}}
  V_{\leg{13}} \Sigma_{\leg{12}} = V_{\leg{12}}V_{\leg{13}}
  V_{\leg{23}} \checkV_{\leg{12}}\Sigma_{\leg{12}} =
  V_{\leg{23}} V_{\leg{12}} \checkV_{\leg{12}}
  \Sigma_{\leg{12}}.
  \end{align*}
  By Lemma \ref{lemma:kac-condition},
  $V_{\leg{12}}\checkV_{\leg{12}}\Sigma_{\leg{12}}=\hatV_{\leg{12}}^{*}U_{\leg{1}}$
  and hence $\hatV_{\leg{12}}^{*}U_{\leg{1}}V_{\leg{23}} =
  V_{\leg{23}}\hatV_{\leg{12}}^{*}U_{\leg{1}}$.  Since
  $\hatV_{\leg{12}}$ is unitary and
  $U_{\leg{1}}V_{\leg{23}}=V_{\leg{23}}U_{\leg{1}}$, we can
  conclude
  $\hatV_{\leg{12}}V_{\leg{23}}=V_{\leg{23}}\hatV_{\leg{12}}$.
  A similar argument shows that
  $\checkV_{23}V_{12}=V_{12}\checkV_{23}$.
\end{proof}
The following result is crucial for the duality presented in
the next section.
\begin{proposition} \label{proposition:kac-compact}
Let $(V,U)$ be a $C^{*}$-pseudo-Kac system. Then $[
  A\hA] = [ \halpha\halpha^{*} ]$.
\end{proposition} 
\begin{proof} 
 The relation $[\halpha^{*}\hA]=\halpha^{*}$ (Remark
  \ref{remarks:balanced} iii)), regularity of $V$, and the
  relations $V^{*}=\Sigma U_{\leg{2}} V\Sigma U_{\leg{2}}
  V\Sigma U_{\leg{2}}$ and
  $[V\kalpha{2}\hA]=[\kbeta{2}\hA]$ 
  \cite[Proposition 5.8]{timmermann:cpmu-hopf} imply
  \begin{align*} 
    [ \halpha \halpha^{*}] 
    = [ U\alpha\alpha^{*}U\hA]  & = [ U \balpha{2}
V^{*}\kalpha{1} U\hA] \\ &= [ U \balpha{2}\Sigma
U_{\leg{2}} V\Sigma U_{\leg{2}} V\Sigma U_{\leg{2}}
\kalpha{1} U \hA] \\ &= [ \balpha{1} V\Sigma
U_{\leg{2}}V\kalpha{2}\hA] \\
&= [ \balpha{1}
V\Sigma U_{\leg{2}}\kbeta{2} \hA] = [ \balpha{1}
V\khbeta{1} \hA] = [ A
\hA]. \qedhere
 \end{align*}
\end{proof}
\begin{lemma}
  Let $(V,U)$ be a (weak) $C^{*}$-pseudo-Kac system.  Then
  also $(\checkV,U)$, $(\hatV,U)$, and $(V^{op},U)$ are (weak)
  $C^{*}$-pseudo-Kac systems.
\end{lemma}
\begin{proof}
  If $(V,U)$ is a weak $C^{*}$-pseudo-Kac system, then the
  tuples above are balanced $C^{*}$-pseudo-multiplica\-tive
  unitaries by Remark \ref{remarks:balanced} i), and the
  remaining necessary conditions follow easily from Proposition
  \ref{proposition:balanced-legs} and equation \eqref{eq:pmu-op}.

  If $(V,U)$ is a $C^{*}$-pseudo-Kac system, then equation
  \eqref{eq:balanced-iterate}, the relation
  $\widecheck{(V^{op})} = U_{1} V^{*} U_{1} = (\hatV)^{op}$,
  and the fact that $V^{op}$ is regular, imply that the
  tuples above satisfy the regularity condition in
  Definition \ref{definition:kac}.  To check that they also
  satisfy the second condition, we use Remark
  \ref{remark:kac} and calculate $(\Sigma U_{\leg{2}}
  \hatV)^{3} = (V \Sigma U_{\leg{2}})^{3} = 1$,
  $(\checkV\Sigma U_{\leg{2}} )^{3} =(\Sigma
  U_{\leg{2}}V)^{3} = 1$, $ (U_{2}V^{op}\Sigma)^{3} =
  (U_{\leg{2}}\Sigma V^{*})^{3} = ((V\Sigma
  U_{\leg{2}})^{3})^{*}= 1$.
\end{proof}

\paragraph{The $C^{*}$-pseudo-Kac system of a compact
  $C^{*}$-quantum groupoid}
In \cite{timmermann:leiden}, we introduced 
compact $C^{*}$-quantum groupoids and  associated to each
such object a regular $C^{*}$-pseudo-multiplicative unitary
$V$. We now recall this construction and define a symmetry
$U$ such that $(V,U)$ is a $C^{*}$-pseudo-Kac system.

A {\em compact $C^{*}$-quantum graph} consists of a unital
$C^{*}$-algebra $B$ with a faithful KMS-state $\mu$, a
unital $C^{*}$-algebra $A$ with unital embeddings $r \colon
B \to A$ and $s \colon B^{op} \to A$ such that
$[r(B),s(B^{op})]=0$, and faithful conditional expectations
$\phi \colon A \to r(B) \cong B$ and $\psi \colon A \to
s(B^{op}) \cong B^{op}$ such that the compositions $\nu:=\mu
\circ \phi$ and $\nu^{-1}:=\mu^{op} \circ \psi$ are
KMS-states related by some positive invertible element
$\delta \in A \cap r(B)' \cap s(B^{op})'$ via the formula
$\nu^{-1}(a)=\nu(\delta^{1/2}a\delta^{1/2})$, valid for all $a \in
A$.  An {\em involution} for such a compact $C^{*}$-quantum
graph is a $*$-antiisomorphism $R \colon A \to A$ such that
$R\circ R=\Id_{A}$, $R(r(b))=s(b^{op})$ and
$\phi(R(a))=\psi(a)^{op}$ for all $b \in B$, $a \in A$.

Let $(B,\mu,A,r,s,\phi,\psi)$ be a compact $C^{*}$-quantum
graph with involution $R$. We denote by
$(H_{\mu},\zeta_{\mu},J_{\mu})$ and
$(H_{\nu},\zeta_{\nu},J_{\nu})$ the GNS-spaces, canonical
cyclic vectors, and modular conjugations for the KMS-states
$\mu$ and $\nu$, respectively, and let
$\zeta_{\nu^{-1}}=\delta^{1/2}\zeta_{\nu}$.  As usual, we
have representations $B^{op} \to \mathcal{ L}(H_{\mu})$,
$b^{op} \mapsto J_{\mu}b^{*}J_{\mu}$, and $A^{op} \to
\mathcal{ L}(H_{\nu})$, $a^{op} \mapsto
J_{\nu}a^{*}J_{\nu}$. Using the isometries
 \begin{align*}
  \zeta_{\phi} &\colon H_{\mu}  \to H_{\nu}, \ b\zeta_{\mu}
  \mapsto r(b)\zeta_{\nu}, & \zeta_{\psi} &\colon
  H_{\mu^{op}} \to H_{\nu}, \ b^{op}\zeta_{\mu^{op}}
  \mapsto s(b^{op}) \zeta_{\nu^{-1}},
 \end{align*}
 we define subspaces $\halpha,\hbeta,\alpha,\beta \subseteq
 \mathcal{ L}(H_{\mu},H_{\nu})$ by
 $\halpha:=[A\zeta_{\phi}]$, $\hbeta:=[A\zeta_{\psi}]$,
 $\beta :=[A^{op}\zeta_{\phi}]$,
 $\beta:=[A^{op}\zeta_{\psi}]$. Let $H=H_{\nu}$ and
 $\frakb=\cbasesb$, where $\frakK=H_{\mu}$, $\frakB=B
 \subseteq \mathcal{ L}(H_{\mu})$, $\frakBo=B^{op} \subseteq
 \mathcal{ L}(H_{\mu})$.  Then
 $(H,\halpha,\hbeta,\alpha,\beta)$ is a
 $C^{*}$-$(\frakb,\frakbo,\frakb,\frakbo)$-module and
 $\cA:=A^{\beta,\alpha}_{H}$  a
 $C^{*}$-$(\frakbo,\frakb)$-algebra
 \cite{timmermann:leiden}.  

 A {\em compact $C^{*}$-quantum groupoid} consists of a
 compact $C^{*}$-quantum graph with involution as above and
 a morphism $\cA \to \cA \bfibre \cA$ of
 $C^{*}$-$(\frakbo,\frakb)$-algebras such that
 \begin{enumerate}
 \item $(\Delta \ast \Id) \circ \Delta = (\Id \ast \Delta)
   \circ \Delta$ as maps from $A$ to ${\cal L}(\Hrange
   \frange H)$;
 \item $\langle
   \zeta_{\phi}|_{2}\Delta(a)|\zeta_{\phi}\rangle_{2} =
   \rho_{\beta}(\phi(a)) $ and $\langle
   \zeta_{\psi}|_{1}\Delta(a)|\zeta_{\psi}\rangle_{1} =
   \rho_{\alpha}(\psi(a))$ for all $a \in A$;
    \item $[\Delta(A)\kalpha{1}]=[\kalpha{1}A] =
      [\Delta(A)|\zeta_{\psi}\rangle_{1}A]$,
      $[\Delta(A)\kbeta{2}]=[\kbeta{2}A] =
      [\Delta(A)|\zeta_{\phi}\rangle_{2}A]$;
    \item $R(\langle\zeta_{\psi}|_{1}\Delta(a)(d^{op}
      \btensor 1)|\zeta_{\psi}\rangle_{1}) =
      \langle\zeta_{\psi}|_{1}(a^{op} \btensor
      1)\Delta(d)|\zeta_{\psi}\rangle_{1}$ for all $a,d \in
      A$.
 \end{enumerate}
 By \cite[Theorem 5.4]{timmermann:leiden}, there exists a
 unique regular $C^{*}$-pseudo-multiplicative unitary $V
 \colon \Hsource \to \Hrange$ such that
 $V|a\zeta_{\psi}\rangle_{1}=\Delta(a)|\zeta_{\psi}\rangle_{1}$
 for all $a \in A$.  Denote by $J=J_{\nu}$ the modular
 conjugation for $\nu$ as above, by $I \colon H \to H$ the
 antiunitary given by $Ia\zeta_{\nu^{-1}} =
 R(a)^{*}\zeta_{\nu}$ for all $a \in A$, and let $U=IJ \in
 {\cal L}(H)$.
\begin{proposition}
  $(V,U)$ is a $C^{*}$-pseudo-Kac system.
\end{proposition}
\begin{proof}
  First, $U^{2}=IJIJ=IJJI=II=\Id_{H}$ because $IJ=JI$, and
  $U\zeta_{\phi}=\zeta_{\psi}$,
  $U\zeta_{\nu}=\zeta_{\nu^{-1}}$, $U\halpha=I\beta=\alpha$,
  $U\hbeta=I\alpha=\beta$ by \cite[Lemma 2.7, Proposition
  3.8]{timmermann:leiden}.  The relation $(J
  \rtensor{\alpha}{J_{\mu}}{\beta} I)V(J
  \rtensor{\alpha}{J_{\mu}}{\beta} I)=V^{*}$ \cite[Theorem
  5.6]{timmermann:leiden} implies
  \begin{align*}
    \checkV = \Sigma(1 \btensor JI)V(1 \botensor JI) \Sigma
    &=  (J \rtensor{\alpha}{J_{\mu}}{\hbeta} J) \Sigma
    (J
  \rtensor{\alpha}{J_{\mu}}{\beta} I)V(J
  \rtensor{\alpha}{J_{\mu}}{\beta} I) \Sigma(J
   \rtensor{\halpha}{J_{\mu}}{\hbeta} J)  \\ &=
  (J \rtensor{\alpha}{J_{\mu}}{\hbeta} J) \Sigma V^{*}\Sigma(J
   \rtensor{\halpha}{J_{\mu}}{\hbeta} J)  =
  (J \rtensor{\alpha}{J_{\mu}}{\hbeta} J) V^{op}(J
   \rtensor{\halpha}{J_{\mu}}{\hbeta} J).
  \end{align*}
  Since $V^{op}$ is a regular $C^{*}$-pseudo-multiplicative
  unitary, so is $\checkV$. In particular, $(V,U)$ is a
  balanced $C^{*}$-pseudo-multiplicative unitary. We show
  that $\hatV V=U_{1}\Sigma \checkV^{*}$, and then the claim
  follows from Lemma \ref{lemma:kac-condition}. Let $a,b \in
  A$ and $\omega= \hatV V (a\zeta_{\psi} \tr
  Ub\zeta_{\nu^{-1}})$. Since $\Delta(a)=\hatV^{*}(1
  \botensor a) \hatV$ by Proposition
  \ref{proposition:balanced-legs},
  \begin{align*}
\omega = \hatV
    \Delta(a)(\zeta_{\psi} \tr Ub\zeta_{\nu^{-1}})  &= (1
    \botensor a) \hatV(\zeta_{\psi} \tr Ub\zeta_{\nu^{-1}})
    \\ &= \Sigma ( U \btensor 1) (UaU \btensor 1)V (b
    \zeta_{\nu^{-1}} \tl \zeta_{\psi}) \\
    &= \Sigma(U \btensor 1)(UaU \btensor
    1)\Delta(b)(\zeta_{\psi} \tr \zeta_{\nu^{-1}}).
\end{align*}
Using the relations $UaU=JIaIJ=R(a)^{op}$ and $[UaU \btensor
1,\Delta(b)] \in [A^{op} \btensor 1, \AfibreA]=0$, we find
\begin{align*}
  \omega &= \Sigma(U
  \btensor 1)\Delta(b)(UaU\zeta_{\psi} \tr \zeta_{\nu^{-1}})
= (U \botensor 1) \Sigma (U \btensor U)\Delta(b)(U
  \btensor U)(aU\zeta_{\psi} \tr \zeta_{\nu}).
\end{align*}
Since $\checkV^{*}(1 \botensor UbU)\checkV = (U
  \btensor U)\Delta(b)(U \btensor U)$ by Proposition
  \ref{proposition:balanced-legs}, we obtain
\begin{align*}
  \omega &= (U \botensor 1) \Sigma \checkV^{*} (1 \botensor
  UbU)
  \checkV (aU\zeta_{\psi} \tr \zeta_{\nu}) \\
  &= (U \botensor 1) \Sigma \checkV^{*} (1 \botensor UbU)
  \Sigma (1 \btensor U)V (\zeta_{\nu} \tl UaU\zeta_{\psi}).
\end{align*}
By Proposition \cite[Proposition 5.5]{timmermann:leiden}, $V
(\zeta_{\nu} \tl UaU\zeta_{\psi}) = \zeta_{\nu }\tl
UaU\zeta_{\phi}$, whence
\begin{align*}
  \omega
    &=(U \botensor 1) \Sigma \checkV^{*} (1 \botensor UbU)
    (aU\zeta_{\phi} \tr \zeta_{\nu})  
    =(U \botensor 1) \Sigma \checkV^{*} (a\zeta_{\psi} \tr
    Ub\zeta_{\nu^{-1}}). \qedhere
\end{align*}
\end{proof}

\section{Reduced crossed products and duality}

\label{section:coactions}

Let $(V,U)$ be a weak $C^{*}$-pseudo-Kac system and let
$(\cA,\Delta)$, $(\hcA,\hDelta)$ be the Hopf
$C^{*}$-bimodules associated to $V$ as in the preceding
section.  Generalizing the corresponding constructions and
results for coactions of Hopf $C^{*}$-algebras
\cite{baaj:2}, we now associate to every coaction of one of
these Hopf $C^{*}$-bimodules a reduced crossed product that
carries a dual coaction of the other Hopf $C^{*}$-bimodule,
and prove a duality theorem concerning the iteration of
this construction.

\paragraph{Reduced crossed products for coactions of $(\cA,\Delta)$}
Let $\delta$ be a coaction of  the Hopf $C^{*}$-bimodule $(\cA,\Delta)$ on a
$C^{*}$-$\frakb$-algebra ${\cal C}=C^{\gamma}_{K}$ and
let\footnote{The notation $C \rtimes_{r} \hA$ is consistent
  with \cite{baaj:2} but not with \cite{enock:action}, where
  $C \rtimes_{r} A$ is used instead.}
\begin{align*}
  C \rtimes_{r} \hA &:= [\delta(C)(1 \btensor \hA)]
  \subseteq {\cal L}(\rHrange), & \cC \rtimes_{r} \hcA &:=
  (\rHrange_{\hbeta}, C \rtimes_{r} \hA).
\end{align*}
\begin{proposition} \label{proposition:rcp}
\begin{enumerate}
\item $[\delta(C)(\gamma \rt \hbeta)] \subseteq \gamma \rt
  \hbeta$ with equality if $\delta$ is left-full.
\item $C \rtimes_{r} \hA$ is a $C^{*}$-algebra and $\cC
  \rtimes_{r} \hcA$ is a $C^{*}$-$\frakbo$-algebra.
\item There exist nondegenerate $*$-homomorphisms $C \to M(C
  \rtimes_{r} \hA)$ and $\hA \to M(C \rtimes_{r} \hA)$,
  given by $c \mapsto \delta(c)$ and $\ha \mapsto 1 \btensor
  \ha$, respectively.
  \end{enumerate}
\end{proposition}
\begin{proof}
  i)  The relation $\hbeta=[A\hbeta]$ \cite[Proposition
  3.11]{timmermann:cpmu-hopf} implies that $[\delta
  (C)\kgamma{1}\hbeta] = [\delta(C)\kgamma{1}A\hbeta]
  \subseteq [\kgamma{1}A\hbeta] = [\kgamma{1}\hbeta]$. 

  ii) We first show that $[(1 \btensor \hA) \delta(C) ]
  \subseteq [ \delta(C)(1 \btensor \hA)]$.  Let
  $\delta^{(2)}:=(\Id \ast \Delta)\circ \delta = (\delta
  \ast \Id) \circ \delta \colon C \to {\cal L}(\rHrange
  \rfrange H)$. By definition of $\hA$ and $\Delta$,
  \begin{align*} [ (1 \btensor \hA) \delta(C) ] = [
    \bbeta{3}(1 \btensor V)\kalpha{3}\delta(C)] &= [
    \bbeta{3}(1 \btensor V)(\delta(C) \botensor
    1)\kalpha{3}] \\
    &=[ \bbeta{3}\delta^{(2)}(C)(1
    \btensor V)\kalpha{3}] \\
    &\subseteq [ \delta(C)\bbeta{3}(1 \btensor
    V)\kalpha{3}] = [ \delta(C)(1 \btensor
    \hA)]. 
  \end{align*}
  Consequently, $C \rtimes_{r} \hA$ is a $C^{*}$-algebra.
  By \cite[Proposition 3.11]{timmermann:cpmu-hopf},
  $[\hA\rho_{\hbeta}(\frakB)]=\hA$, and hence $[(C
  \rtimes_{r} \hA)\rho_{(\gamma \rt \hbeta)}(\frakB)] =
  [\delta(C)(1 \btensor \hA \rho_{\hbeta}(\frakB))] =
  [\delta(C)(1 \btensor \hA)] = C \rtimes_{r} \hA$.

  iii) Immediate.
\end{proof}
\begin{theorem} 
  There exists a unique coaction $\hdelta$ of
  $(\hcA,\hDelta)$ on  $\cC \rtimes_{r} \hcA$ such that
  $\hdelta(\delta(c)(1 \btensor \ha)) = (\delta(c) \botensor
  1) (1 \btensor \hDelta(\ha))$ for all $c\in C$, $ \ha \in
  \hA$. If $\hDelta$ is a fine coaction, then
  $\hdelta$ is a very fine coaction. If $\delta$ is
  left-full, then $\hdelta$ is left-full.
\end{theorem}
 \begin{proof}
   Define $\hdelta \colon C \rtimes_{r} \hA \to \mathcal{
     L}(K \rtensorcb \Hsource)$ by $ x \mapsto (1 \btensor
   \checkV)(x \btensor 1)(1 \btensor \checkV^{*})$. Then
   $\hdelta$ is injective and satisfies $\hdelta(\delta(c)(1
   \btensor \ha)) = (\delta(c) \botensor 1) (1 \btensor
   \hDelta(\ha))$ for all $c\in C$, $ \ha \in \hA$ because
   $\checkV(\ha \rtensorh 1)\checkV^{*} = \hDelta(\ha)$ by
   Proposition \ref{proposition:balanced-legs} and $(1
   \btensor \checkV)\delta(c)(1 \btensor \checkV^{*}) =
   \delta(c)$ as a consequence of the relation
   $\checkV(a\btensor 1)\checkV^{*}=a \botensor 1$. 
   We show that $\hdelta$ is a coaction of
   $(\hcA,\hDelta)$.  First, $[\hdelta(C \rtimes_{r}
   \hA)\kalpha{3}] \subseteq [\kalpha{3}(C \rtimes_{r}
   \hA)]$ because
   \begin{align}
     [ (\delta(C) \botensor 1)(1 \btensor
     \hDelta(\hA))\kalpha{3} ] &\subseteq [ (\delta(C)
     \botensor 1)\kalpha{3} (1 \btensor \hA) ] = [
     \kalpha{3}\delta(C)(1 \btensor
     \hA)]. \label{eq:rcp-fine-1}
   \end{align}
   Next, $[\hdelta(C \rtimes_{r} \hA) |\gamma \rt
   \hbeta\rangle_{\leg{1}}\hA] \subseteq [ |\gamma \rt
   \hbeta\rangle_{\leg{1}} \hA]$ because by Proposition
   \ref{proposition:rcp} i),
  \begin{align}
    [ (1 \btensor \hDelta(\hA))(\delta(C)
    \botensor 1) |\gamma \rt \hbeta\rangle_{\leg{1}}\hA]
    \nonumber &\subseteq [ (1 \btensor \hDelta(\hA))
    |\gamma \rt \hbeta\rangle_{\leg{1}}\hA] \nonumber \\ & =
    [ |\gamma\rangle_{\leg{1}}
    \hDelta(\hA)|\hbeta\rangle_{\leg{1}}\hA] 
    \subseteq [ |\gamma\rangle_{\leg{1}}
    |\hbeta\rangle_{\leg{1}} \hA]. \label{eq:rcp-fine-2}
  \end{align}
Furthermore, $\hdelta(x) (1 \btensor
  \checkV)|\xi\rangle_{\leg{3}} = (1 \btensor
  \checkV)|\xi\rangle_{\leg{3}}x$ for each $x \in C
  \rtimes_{r} \hA$, $\xi \in \hbeta$, and by Remark
  \ref{remarks:balanced} ii), $[(1 \btensor
  \checkV)\khbeta{3}(\gamma \rt \hbeta)]= \gamma \rt \hbeta
  \rt \hbeta$ and $[\bhbeta{3}(1 \btensor
  \checkV)^{*}(\gamma \rt \hbeta \rt \hbeta)]=\gamma \rt
  \hbeta$.  The maps $(\hdelta \ast \Id) \circ \hdelta$ and
  $(\Id \ast \hDelta) \circ \hdelta$ from $C \rtimes_{r}
  \hA$ to $\mathcal{ L}\big(K \rtensorcb \Hsource
  \rtensor{\hbeta}{\frakbo}{\alpha} H\big)$ are given by
  $\delta(c)(1 \btensor \ha) \mapsto \big(\delta(c)
  \botensor 1 \botensor 1\big)\big(1 \btensor
  \hDelta^{(2)}(\ha)\big)$ for all $ c \in C, \ha \in \hA$,
  where $\hDelta^{(2)}:=(\hDelta \ast \Id) \circ \hDelta =
  (\Id \ast \hDelta) \circ \hDelta$. Thus, $(\cC \rtimes_{r}
  \hcA, \hdelta)$ is a coaction of $(\hcA,\hDelta)$. If the
  coactions $\hDelta$ is fine, then the inclusion
  \eqref{eq:rcp-fine-1} is an equality and in any case
  $[\bhbeta{3}(1 \btensor \checkV)^{*}(\gamma \rt \hbeta \rt
  \hbeta)]=\gamma \rt \hbeta$, whence $\hdelta$ will be very
  fine. If $\delta$ is left-full, then the inclusion
  \eqref{eq:rcp-fine-2} is an equality by Proposition
  \ref{proposition:rcp} i) and hence $\hdelta$ is left-full.
\end{proof}
\begin{definition} \label{definition:rcp} We call $\cC \rtimes_{r}
  \hcA$  the {\em reduced crossed product} and $(\cC
  \rtimes_{r} \hcA,\hdelta)$ the {\em reduced dual coaction}
  of $(\cC,\delta)$.
\end{definition}
The construction of reduced  dual coactions is functorial in
the following sense:
\begin{proposition} 
  Let $\rho$ be a morphism between coactions $({\cal
    C},\delta_{{\cal C}})$ and $({\cal D},\delta_{{\cal
      D}})$ of $(\cA,\Delta)$. Then there exists a unique
  morphism $\rho \rtimes_{r} \Id$ from $(\cC \rtimes_{r}
  \hcA,\hdelta_{{\cal C}})$ to $({\cal D} \rtimes_{r}
  \hcA,\hdelta_{{\cal D}})$ such that $(\rho \rtimes_{r}
  \Id)\big((1 \btensor \ha)\delta_{{\cal C}}(c)\big) \cdot
  \delta_{{\cal D}}(d)(1 \btensor \ha') = (1 \btensor \ha)
  \delta_{{\cal D}}(\rho(c)d)(1 \btensor \ha')$ for all $c
  \in C$, $d \in D$, $\ha,\ha' \in \hA$.
\end{proposition}
\begin{proof} 
  The semi-morphism $\Ind_{\kbeta{2}}(\rho)$ of Lemma
  \ref{lemma:fp-c-morphism} restricts to a semi-morphism
  $\rho \rtimes_{r} \Id$ from ${\cal C} \rtimes_{r} \hcA$ to
  $M({\cal D} \rtimes_{r} \hcA)$ which satisfies the formula
  given above, and this formula implies that $\rho
  \rtimes_{r} \Id$ is a morphism of coactions as claimed.
\end{proof}
\begin{corollary}
  There exists a functor $- \rtimes_{r} \hcA \colon
  \coact_{(\cA,\Delta)} \to
  \coact_{(\hcA,\hDelta)}$ given by $(\cC,\delta)
  \mapsto (\cC \rtimes_{r} \hcA,\hdelta)$ and $\rho \mapsto
  \rho \rtimes_{r} \Id$. \qed
\end{corollary}

\paragraph{Reduced crossed products for coactions of $(\hcA,\hDelta)$}
The construction in the preceding paragraph carries over to
coactions of the Hopf $C^{*}$-bimodule $(\hcA,\hDelta)$ as
follows.  Let $\delta$ be a coaction of $(\hcA,\hDelta)$ on
a $C^{*}$-$\frakbo$-algebra $\cC = C^{\gamma}_{K}$ and let
\begin{align*}
  C \rtimes_{r} A &:= [\delta(C)(1 \botensor UAU)]
\subseteq {\cal L}(K   \rtensor{\gamma}{\frakbo}{\alpha} H), &
  \cC \rtimes_{r} \cA &= (K
  \rtensor{\gamma}{\frakbo}{\alpha} H_{\halpha},
  C \rtimes_{r} A).
\end{align*}
Using straightforward modifications of the preceding proofs,
one shows:
\begin{proposition} 
\begin{enumerate}
\item $[\delta(C)(\gamma \rt \halpha)] \subseteq \gamma \rt
  \halpha$ with equality if $\delta$ is fine.
\item $C \rtimes_{r} A$ is a $C^{*}$-algebra and $\cC
  \rtimes_{r} \cA$ is a $C^{*}$-$\frakb$-algebra.
\item There exist nondegenerate $*$-homomorphisms $C \to M(C
  \rtimes_{r} A)$ and $A \to M(C \rtimes_{r} A)$,
  given by $c \mapsto \delta(c)$ and $a \mapsto 1 \botensor
  a$, respectively. \qed
  \end{enumerate}
\end{proposition}
\begin{theorem} 
  There exists a unique  coaction $(\cC \rtimes_{r} \cA,
  \hdelta)$ of $(\cA,\Delta)$ such that $\hdelta(\delta(c)(1
  \botensor UaU)) = (\delta(c) \btensor 1) (1 \botensor
  \Ad_{(U \btensor 1)}\Delta(a)) $ for all $c\in C$, $ a \in
  A$. If $\Delta$ is a fine coaction, then
  $\hdelta$ is a very fine coaction. If $\delta$ is
  left-full, then $\hdelta$ is left-full. \qed
\end{theorem}
\begin{definition} Let $(\cC,\delta)$
  be a coaction of $(\hcA,\hDelta)$.  Then we call $\cC \rtimes_{r}
  \cA$ the {\em reduced crossed product} and $(\cC
  \rtimes_{r} \cA,\hdelta)$ the {\em reduced dual coaction}
  of $(\cC,\delta)$.
\end{definition}
\begin{proposition} 
  Let $\rho$ be a morphism between coactions $({\cal
    C},\delta_{{\cal C}})$ and $({\cal D},\delta_{{\cal
      D}})$ of $(\hcA,\hDelta)$. Then there exists a unique
  morphism $\rho \rtimes_{r} \Id$ from $(\cC \rtimes_{r}
  \cA,\hdelta_{{\cal C}})$ to $({\cal D} \rtimes_{r}
  \cA,\hdelta_{{\cal D}})$ such that $(\rho \rtimes_{r}
  \Id)\big((1 \botensor UaU)\delta_{{\cal C}}(c)\big) \cdot
  \delta_{{\cal D}}(d)(1 \botensor Ua'U)= (1 \botensor UaU)
  \delta_{{\cal D}}(\rho(c)d)(1 \botensor Ua'U)$ for all $c
  \in C$, $d \in D$, $a,a' \in A$. \qed
\end{proposition}
\begin{corollary} 
  There exists a functor $- \rtimes_{r} \cA \colon
  \coact_{(\hcA,\hDelta)} \to \coact_{(\cA,\Delta)}$ given
  by $(\cC,\delta) \mapsto (\cC \rtimes_{r} \cA,\hdelta)$
  and $\rho \mapsto \rho \rtimes_{r} \Id$. \qed
\end{corollary}

\paragraph{The duality theorem}
 The preceding constructions yield for each
coaction $(\cC,\delta_{{\cal C}})$ of $(\cA,\Delta)$ and
each coaction $(\mathcal{D},\delta_{{\cal D}})$ of
$(\hcA,\hDelta)$ a bidual $(\cC \rtimes_{r} \hcA \rtimes_{r}
\cA, \hhdelta_{{\cal C}})$ and $({\cal D} \rtimes_{r} \cA
\rtimes_{r} \hcA, \hhdelta_{{\cal D}})$, respectively. The following
generalization of the Baaj-Skandalis duality theorem
\cite{baaj:2} identifies these biduals in the case where
$(V,U)$ is a $C^{*}$-pseudo-Kac system and the initial
coactions are fine. Morally, it says that up to Morita
equivalence, the functors $-\rtimes_{r} \hcA$ and
$-\rtimes_{r} \cA$ implement an equivalence of the
 categories $\bfcoact^{f}_{(\cA,\Delta)}$ and
$\bfcoact^{f}_{(\hcA,\hDelta)}$.
\begin{theorem} \label{theorem:duality} Assume that $(V,U)$
  is a $C^{*}$-pseudo-Kac system.
  \begin{enumerate}
  \item Let $(\cC,\delta)$ be a (very) fine coaction of
    $(\cA,\Delta)$, where $\cC=C^{\gamma}_{K}$.  Then there
    exists an isomorphism $\Phi \colon C \rtimes_{r} \hA
    \rtimes_{r} A \to [\kbeta{2}C\bbeta{2}] \subseteq {\cal
      L}(K \rtensorcb H)$ such that $\Phi^{-1}$ is an
    (iso)morphism from $(K \rtensorcb H_{
      \halpha}, [ \kbeta{2} C \bbeta{2}])$ to $\cC
    \rtimes_{r} \hcA \rtimes_{r} \cA$ and $\hhdelta \circ
    \Phi^{-1} = (\Phi^{-1} \ast \Id) \circ \Ad_{(1 \btensor
      \Sigma\hatV)} \circ \Ind_{\kbeta{2}}(\delta)$.
  \item Let $({\cal D},\delta)$ be a (very) fine coaction of
    $(\hcA,\hDelta)$, where ${\cal D}=D^{\epsilon}_{L}$.
    Then there exists an isomorphism $\Phi \colon D
    \rtimes_{r} A \rtimes_{r} \hA \cong [ \kalpha{2} D
    \balpha{2}] \subseteq \mathcal{ L}(L
    \rtensor{\epsilon}{\frakbo}{\alpha} H)$ such that
    $\Phi^{-1}$ is an (iso)morphism from $(L
    \rtensor{\epsilon}{\frakbo}{\alpha} H_{
      \hbeta}, [\kalpha{2}D\balpha{2}])$ to ${\cal D}
    \rtimes_{r} \cA \rtimes_{r} \hcA$ and $\hhdelta \circ
    \Phi^{-1} = (\Phi^{-1} \ast \Id) \circ \Ad_{(1 \botensor
      \Sigma V)} \circ \Ind_{\kalpha{2}}(\delta)$.
  \end{enumerate}
\end{theorem}
\begin{proof} 
  We only prove  i); assertion ii) follows
  similarly after replacing $(V,U)$ by $(\checkV,U)$.  By
  Proposition \ref{proposition:balanced-legs} and
  Proposition \ref{proposition:kac-compact}, applied to the
  $C^{*}$-pseudo-Kac system $(\checkV,U)$, we have $[ \hA
  \Ad_{U}(A)] = [ A_{\checkV} \hA_{\checkV}] = [ \beta
  \beta^{*}]$, and since $\delta$ is fine,
  \begin{align*}
    [\kbeta{2}C\bbeta{2}] = [\delta(C)(1 \btensor
    \beta\beta^{*})] = [\delta(C)(1 \btensor \hA
    \Ad_{U}(A))].
  \end{align*}
  One easily verifies that the $*$-homomorphism $
  \Ind_{\kbeta{2}}(\delta)$ (see Lemma
  \ref{lemma:fp-c-morphism}) yields an (iso)morphism of
  $C^{*}$-$\frakb$-algebras 
  \begin{align*}
    \Ind_{\kbeta{2}}(\delta) \colon \big(\rHrange_{\halpha},
    [\kbeta{2}C\bbeta{2}]\big) \to \big(\rHrange \rtensorab
    H_{\halpha}, [\kbeta{2}\delta(C)\bbeta{2}]\big).
  \end{align*}
  Denote by $\Psi$ the composition of this (iso)morphism
  with the isomorphism $\Ad_{(1 \btensor V^{*})}$ and let
  $\delta^{(2)} = (\delta \ast \Id) \circ \delta = (\Id \ast
  \Delta) \circ \delta$. Let $x = \delta(c)(1 \btensor \ha
  UaU) \in [\kbeta{2}C\bbeta{2}]$, where $c\in C, \ha \in
  \hA, a \in A$. By Lemma \ref{lemma:weak-kac},
  \begin{align*}
    \Psi(x) &= \Ad_{(1 \btensor V^{*})}(\delta^{(2)}(c)(1 \btensor
    1 \btensor \ha UaU))
    = (\delta(c) \botensor 1)(1 \btensor
    \hDelta(\ha))(1 \btensor 1 \botensor UaU).
  \end{align*}
  Consequently, $\Psi([\kbeta{2}C\bbeta{2}])=C \rtimes_{r}
  \hA \rtimes_{r} A$.  The relations $\cC \rtimes_{r} \hcA
  \rtimes_{r} \cA = (\rHrange
  \rtensor{\hbeta}{\frakbo}{\alpha} H_{ \halpha}, C
  \rtimes_{r} \hA \rtimes_{r} A)$ and $(1 \btensor
  V^{*})(\gamma \rt \alpha \rt \halpha)= \gamma \rt \hbeta
  \rt \halpha$ imply that $\Psi$ is a morphism of
  $C^{*}$-$\frakb$-algebras as claimed.  Using the
  definition of $\hhdelta$, Proposition
  \ref{proposition:balanced-legs}, and Lemma
  \ref{lemma:weak-kac}, we find
\begin{align*}
  \hhdelta(\Psi(x)) &= (\delta(c) \botensor 1 \btensor 1)(1
  \btensor \widehat{\Delta}(\hat{a}) \btensor 1)(1 \btensor
  1
  \botensor \Ad_{(U \btensor 1)}(\Delta(a))) \\
  &= \Ad_{(1 \btensor V^{*} \btensor
    1)}\big((\delta^{(2)}(c) \btensor 1)(1 \btensor 1
  \btensor \ha \btensor 1)(1 \btensor 1 \btensor \Ad_{(U
    \btensor 1)}(\Delta(a)))\big)  \\
  &= (\Psi \ast \Id)\big((\delta(c) \btensor 1)(1 \btensor
  \ha \btensor 1)(1 \btensor \Ad_{(U \btensor
    1)}(\Delta(a)))\big)
  \\
  &= (\Psi \ast \Id)\big((1 \btensor
  \Sigma\hatV)\delta^{(2)}(c)(1 \btensor 1 \btensor \ha UaU)
  (1 \btensor \hatV^{*}\Sigma)\big) \\ &= (\Psi \ast
  \Id)\big((1 \btensor \Sigma \hatV) (\Ind_{\kbeta{2}}(\delta)
  (x))(1 \btensor \hatV^{*}\Sigma)\big). \qedhere
\end{align*}
 \end{proof}

\section{The $C^{*}$-pseudo-Kac system of a groupoid}
\label{section:kac-groupoid}

For the remainder of this article, we fix a locally compact,
Hausdorff, second countable groupoid $G$ with a left Haar
system $\lambda$. In \cite{timmermann:cpmu-hopf}, we
associated to such a groupoid a regular
$C^{*}$-pseudo-multiplicative unitary $V$ and identified the
underlying $C^{*}$-algebras of the Hopf $C^{*}$-bimodules
$(\hcA,\hDelta)$ and $(\cA,\Delta)$ of $V$ with
the function algebra $C_{0}(G)$ and the reduced groupoid
$C^{*}$-algebra $C^{*}_{r}(G)$, respectively.  We now recall
this construction and define a symmetry $U$ such that
$(V,U)$ becomes a $C^{*}$-pseudo-Kac system.  For background
on groupoids, see \cite{paterson,renault}.

Denote by $\lambda^{-1}$ the right Haar system associated to
$\lambda$ and let $\mu$ be a measure on the unit space
$G^{0}$ with full support.  We denote the range and the
source map of $G$ by $r$ and $s$, respectively, let
$G^{u}:=r^{-1}(u)$ and $G_{u}:=s^{-1}(u)$ for each $u \in
G^{0}$, and define measures $\nu,\nu^{-1}$ on $G$ such that
 \begin{align*} \int_{G} f \intd \nu &= \int_{G^{0}}
   \int_{G^{u}} f(x) \intd\lambda^{u}(x) \intd\mu(u), &
   \int_{G} f d\nu^{-1} &= \int_{G^{0}} \int_{G_{u}} f(x)
   \intd\lambda^{-1}_{u}(x) \intd\mu(u)
\end{align*}
for all $f \in C_{c}(G)$.  We assume that $\mu$ is
quasi-invariant in the sense that $\nu$ and $\nu^{-1}$ are
equivalent, and denote by $D:=\intd\nu/\intd\nu^{-1}$ the
Radon-Nikodym derivative. One can choose $D$ such that it is a
Borel homomorphism, see \cite[p.\ 89]{paterson}, and we do so.

We identify functions in $C_{b}(G^{0})$ and $C_{b}(G)$ with
multiplication operators on the Hilbert spaces
$L^{2}(G^{0},\mu)$ and $L^{2}(G,\nu)$, respectively, and let
$\frakK=L^{2}(G^{0},\mu)$, $\frakB= \frakBo= C_{0}(G^{0})
\subseteq \mathcal{ L}(\frakK)$,
$\frakb=(\frakK,\frakB,\frakBo)=\frakb$, $H = L^{2}(G,\nu)$.

Pulling functions on $G^{0}$ back to $G$ along $r$ or $s$,
we obtain representations $r^{*} \colon C_{0}(G^{0}) \to
C_{b}(G) \hookrightarrow \mathcal{ L}(H)$ and $ s^{*} \colon
C_{0}(G^{0}) \to C_{b}(G) \hookrightarrow \mathcal{ L}(H)$.
We define Hilbert $C^{*}$-$C_{0}(G^{0})$-modules
$L^{2}(G,\lambda)$ and $L^{2}(G,\lambda^{-1})$ as the
respective completions of the pre-$C^{*}$-module $C_{c}(G)$,
the structure maps being given by
 \begin{align*} 
   \langle \xi'|\xi\rangle(u)&= \int_{G^{u}}
   \overline{\xi'(x)}\xi(x) \intd\lambda^{u}(x), & \xi f &=
   r^{*}(f)\xi &  &\text{in the case of } L^{2}(G,\lambda), \\
   \langle \xi'|\xi\rangle(u)&= \int_{G_{u}}
   \overline{\xi'(x)}\xi(x) \intd\lambda^{-1}_{u}(x), & \xi
   f &= s^{*}(f)\xi & &\text{in the case of }
   L^{2}(G,\lambda^{-1})
\end{align*} 
respectively, for all $\xi,\xi' \in C_{c}(G)$, $u \in
G^{0}$, $f \in C_{0}(G^{0})$. Then there exist isometric
embeddings $j\colon L^{2}(G,\lambda) \to \mathcal{ L}(\frakK,H)$
and $\hat j \colon L^{2}(G,\lambda^{-1}) \to \mathcal{
  L}\big(\frakK,H\big)$ such that  
  \begin{align*}
    \big(j(\xi) \zeta\big)(x) &= \xi(x)\zeta(r(x)), &
    \big(\hat j(\xi)\zeta\big)(x) &= \xi(x) D^{-1/2}(x)
    \zeta(s(x))
  \end{align*}
  for all $\xi \in C_{c}(G)$, $\zeta \in C_{c}(G^{0})$.  Let
  $\alpha=\beta:=j(L^{2}(G,\lambda))$ and $\halpha=\hbeta :=
  \hat j(L^{2}(G,\lambda^{-1}))$.  Then
  $(H,\halpha,\hbeta,\alpha,\beta)$ is a
  $C^{*}$-$(\frakb,\frakbo,\frakb,\frakbo)$-module,
  $\rho_{\alpha}=\rho_{\beta}=r^{*}$ and
  $\rho_{\halpha}=\rho_{\hbeta}=s^{*}$, and
  $j(\xi)^{*}j(\xi')=\langle \xi|\xi'\rangle$ and $\hat
  j(\eta)^{*}\hat j(\eta')=\langle\eta|\eta'\rangle$ for all $\xi,\xi' \in
  L^{2}(G,\lambda)$, $\eta,\eta' \in L^{2}(G,\lambda^{-1})$
  \cite{timmermann:cpmu-hopf}.  

The Hilbert spaces
  $\Hsource$ and $\Hrange$ can be described as follows.
  Define measures $\nu^{2}_{s,r}$ on $\GsrG$ and
  $\nu^{2}_{r,r}$ on $\GrrG$ such that
  \begin{gather} \label{eq:measures}
    \begin{aligned}
      \int_{\GsrG}  f\intd\nu^{2}_{s,r} &= \int_{G^{0}}
      \int_{G^{u}} \int_{G^{s(x)}} f(x,y)
      \intd\lambda^{s(x)}(y) \intd\lambda^{u}(x) \intd\mu(u), \\
      \int_{\GrrG}  g\intd\nu^{2}_{r,r} &= \int_{G^{0}}
      \int_{G^{u}}\int_{G^{u}}
      g(x,y)\intd\lambda^{u}(y)\intd\lambda^{u}(x)\intd\mu(u)
    \end{aligned}
  \end{gather}
 for all $f \in C_{c}(\GsrG)$, $g\in C_{c}(\GrrG)$. Then
 there exist unitaries 
  \begin{align*} 
    \Phi \colon \Hsource &\to
    L^{2}(\GsrG,\nu^{2}_{s,r})&&\text{and} & \Psi
        \colon \Hrange &\to L^{2} (\GrrG,\nu^{2}_{r,r})
  \end{align*}
  such that for all $\eta,\xi \in C_{c}(G)$, $\zeta \in
  C_{c}(G^{0})$, 
\begin{align*}
  \Phi\big(\hat j(\eta) \tr \zeta \tl
  j(\xi)\big)(x,y) &= \eta(x)D^{-1/2}(x) \zeta(s(x))
  \xi(y), \\ \Psi\big(j(\eta) \tr \zeta \tl
  j(\xi)\big)(x,y) &= \eta(x) \zeta(r(x)) \xi(y). 
\end{align*}
From now on, we use these isomorphisms without further
notice.
\begin{theorem} \label{theorem:groupoid-kac} There exists a
  $C^{*}$-pseudo-Kac system $(V,U)$ on
  $(H,\halpha,\hbeta,\alpha,\beta)$ such that for all
  $\omega \in C_{c}(\GsrG)$, $(x,y) \in \GrrG$, $\xi \in
  C_{c}(G)$, $z \in G$,
      \begin{align} \label{eq:groupoid-kac}
        \begin{aligned}
          (V\omega)(x,y) &= \omega(x,x^{-1}y) & &\text{and}
          & (U \xi)(x) &= \xi(x^{-1}) D(x)^{-1/2}.
        \end{aligned}
  \end{align}
\end{theorem}
\begin{proof} 
  By \cite[Theorem 2.7, Example 5.3
  ii)]{timmermann:cpmu-hopf}, there exists a regular
  $C^{*}$-pseudo-multiplicative unitary $V$ as claimed.  The
  second formula in \eqref{eq:groupoid-kac} defines a
  unitary $U \in \mathcal{ L}(H)$ by definition of the
  Radon-Nikodym derivative $D=d\nu/d\nu^{-1}$, and $U^{2}=\Id$
  because $(U^{2}\xi) (x) = (U\xi)(x^{-1}) D(x)^{-1/2} =
  \xi(x) D(x)^{1/2} D(x)^{-1/2} = \xi(x)$ for all $\xi \in
  C_{c}(G)$ and $x \in G$.  The unitary $\hatV=\Sigma
  U_{\leg{1}}VU_{\leg{1}}\Sigma$ is equal to $V^{op}=\Sigma
  V^{*} \Sigma$ because
  \begin{align*} 
    (U_{\leg{1}}VU_{\leg{1}} \omega)(x,y) &= (VU_{\leg{1}}
    \omega)(x^{-1},y)D(x)^{-1/2} \\ &= (U_{\leg{1}}
    \omega)(x^{-1},xy)D(x)^{-1/2} \\ &= \omega(x,xy)
    D(x^{-1})^{-1/2} D(x)^{-1/2} = \omega(x,xy)
  \end{align*}
  for all $\omega \in C_{c}(\GrrG)$, $(x,y) \in \GsrG$.  In
  particular, $\hatV$ is a regular
  $C^{*}$-pseudo-multiplicative unitary.  It remains to show
  that the map $Z:=\Sigma U_{\leg{2}}V \colon \Hsource \to
  \Hsource$ satisfies $Z^{3}=1$.  But for all $\omega \in
  C_{c}(\GsrG)$ and $(x,y) \in \GsrG$,
  \begin{align*}
    (Z\omega)(x,y) &= (V\omega)(y,x^{-1})D(x)^{-1/2} =
    \omega(y,y^{-1}x^{-1})
    D(x)^{-1/2}, \\
    (Z^{3}\omega)(x,y) &= (Z^{2}\omega) (y,y^{-1}x^{-1})
    D(x)^{-1/2} \\ &= (Z\omega)(y^{-1}x^{-1},x y y^{-1})
    \big(D(x)D(y)\big)^{-1/2} \\ &= \omega(x,x^{-1} xy)
    \big(D(x)D(y)D(y^{-1}x^{-1})\big)^{-1/2} =
    \omega(x,y). \qedhere
  \end{align*}
\end{proof}

The Hopf $C^{*}$-bimodules $(\hcA,\hDelta)$ and
$(\cA,\Delta)$ associated to $V$ can be described as follows
\cite[Theorem 3.22]{timmermann:cpmu-hopf}. Given $g\in
C_{c}(G)$, define $L(g) \in C^{*}_{r}(G) \subseteq {\cal
  L}(H)$  by
\begin{align*}
  (L(g)f)(x) = \int_{G^{r(x)}} g(z)f(z^{-1}x)
D^{-1/2}(z)\intd\lambda^{r(x)}(z)
\end{align*}
for all $x \in G$, $f\in C_{c}(G) \subseteq
L^{2}(G,\nu)=H$. Then
\begin{gather} \label{eq:groupoid-leg-ha}
    \begin{aligned}
      \hA &= C_{0}(G) \subseteq {\cal L}(H), &
      \big(\hDelta(f)\omega\big)(x,y) &= f(xy)
      \omega(x,y), \end{aligned} \\ 
    \begin{aligned}
      A &= C^{*}_{r}(G), &
      \big(\Delta(L(g))\omega'\big)(x',y') &=
      \int_{G^{u'}} g(z)D^{-1/2}(z)
      \omega'(z^{-1}x',z^{-1}y') \intd\lambda^{u'}(z)
    \end{aligned}
  \nonumber
\end{gather}
for all $f \in C_{0}(G)$, $\omega \in C_{c}(\GsrG)$, $(x,y)
\in \GsrG$ and $g \in C_{c}(G)$, $\omega' \in
C_{c}(\GrrG)$, $(x',y') \in \GrrG$, where $u'=r(x')=r(y')$.
We shall loosely refer to $C_{0}(G)$ and $C^{*}_{r}(G)$ as
Hopf $C^{*}$-bimodules, having in mind $(\hcA,\hDelta)$ and
$(\cA,\Delta)$, respectively.

\section{Actions of $G$ and coactions of $C_{0}(G)$}

\label{section:actions}

Let $G$ be a groupoid and consider $C_{0}(G)$ as a Hopf
$C^{*}$-bimodule as in the preceding section.  Then
coactions of $C_{0}(G)$ can be related to actions of $G$ as
follows.  Let us say that a tuple
$(\bfF,\bfG,\eta,\epsilon)$ is {\em an embedding of a
  category $\bfC$ into a category $\bfD$ as a full and
  coreflective subcategory} if $\bfF \colon \bfC \to \bfD$
is a full and faithful functor and $\bfG \colon \bfD \to
\bfC$ is a faithful right adjoint to $\bfF$, where $\eta
\colon \Id_{\bfC} \to \bfG\bfF$ is the unit and $\epsilon
\colon \bfF\bfG \to \Id_{\bfD}$ is the counit of the
adjunction; see also \cite[\S IV.3]{maclane}.  In this
section, we construct such an embedding of the category of
actions of $G$ on continuous $C_{0}(G^{0})$-algebras into
the category of certain admissible coactions of $C_{0}(G)$.
We keep the notation introduced in the preceding section.

\paragraph{$C_{0}(G^{0})$-algebras and
  $C^{*}$-$\frakb$-algebras}
We shall embed the category of admissible
$C_{0}(G^{0})$-algebras into the category of admissible
$C^{*}$-$\frakb$-algebras as a full and coreflective
subcategory. 

Recall that a {\em $C_{0}(X)$-algebra}, where $X$ is some
locally compact Hausdorff space, is a $C^{*}$-algebra $C$
with a fixed nondegenerate $*$-homomorphism of $C_{0}(X)$
into the center of the multiplier algebra $M(C)$
\cite{blanchard,kasparov:invent}. We denote the fiber of a
$C_{0}(X)$-algebra $C$ at a point $x \in X$ by $C_{x}$ and
write the quotient map $p_{x} \colon C \to C_{x}$ as $c
\mapsto c_{x}$. Recall that $C$ is a {\em continuous
  $C_{0}(X)$-algebra} if the map $ X \to \reals$ given by $x
\mapsto \|c_{x}\|$ is continuous for each $c \in C$. A {\em
  morphism} of $C_{0}(X)$-algebras $C, D$ is a nondegenerate
$*$-homomorphism $\pi\colon C \to M(D)$ such that
$\pi(fc)=f\pi(c)$ for all $f \in C_{0}(X)$, $c \in C$.
\begin{definition}
  We call a $C_{0}(G^{0})$-algebra $C$ {\em admissible} if
  it is continuous and if $C_{u} \neq 0$ for each $u \in
  G^{0}$, and we call a $C^{*}$-$\frakb$-algebra
  $C^{\gamma}_{K}$ {\em admissible} if
  $[\rho_{\gamma}(C_{0}(G^{0}))C]=C$ and $[C\gamma]=\gamma$.
  A {\em morphism} between admissible
  $C^{*}$-$\frakb$-algebras $C^{\gamma}_{K}$,
  $D^{\epsilon}_{L}$ is a semi-morphism $\pi$ from
  $C^{\gamma}_{K}$ to $M(D)^{\epsilon}_{L}$ that is
  nondegenerate in the sense that $[\pi(C)D]=D$.
Denote by $\bfCG^{a}$ the category of all admissible
$C_{0}(G^{0})$-algebras, and by $\bfb^{a}$ the category of all
admissible $C^{*}$-$\frakb$-algebras.
\end{definition}

\begin{lemma} \label{lemma:groupoid-cx-g}
  \begin{enumerate}
  \item Let $C^{\gamma}_{K}$ be an admissible
    $C^{*}$-$\frakb$-algebra. Then $C$ is an admissible
    $C_{0}(G^{0})$-algebra with respect to $\rho_{\gamma}$.
  \item Let $\pi$ be a morphism between admissible
    $C^{*}$-$\frakb$-algebras $C^{\gamma}_{K}$ and
    $D^{\epsilon}_{L}$. Then $\pi$ is a morphism of
    $C_{0}(G^{0})$-algebras from $(C,\rho_{\gamma})$ to
    $(D,\rho_{\epsilon})$.
  \end{enumerate}
\end{lemma}
\begin{proof}
  i) The subalgebra $\rho_{\gamma}(C_{0}(G^{0})) \subseteq
  M(C)$ is central because $C \subseteq \mathcal{
    L}(K_{\gamma}) \subseteq
  \rho_{\gamma}(C_{0}(G^{0}))'$. The map $C \hookrightarrow
  \mathcal{ L}(K_{\gamma}) \cong \mathcal{ L}(\gamma)$ is a
  faithful field of representations in the sense of
  \cite[Theorem 3.3]{blanchard}, and therefore $C$ is a
  continous $C_{0}(G^{0})$-algebra. We have $C_{u} \neq 0$
  for each $u \in G^{0}$ because otherwise $C=[C I_{u}]$,
  where $I_{u} =C_{0}(G^{0} \setminus \{u\})$, and then
  $[\gamma^{*}\gamma]=[\gamma^{*}C\gamma] =
  [\gamma^{*}I_{u}C\gamma]=[\gamma^{*} \gamma I_{u}] = I_{u}
  \neq C_{0}(G^{0})$, contradicting the fact that
  $K_{\gamma}$ is a $C^{*}$-$\frakb$-module.

  ii) This follows from \cite[Proposition 3.5]{timmermann:fiber}.
\end{proof}
  We embed $\bfCG^{a}$ into
$\bfb^{a}$ using a  KSGNS-construction for the following kind of
weights.
\begin{definition}
  A {\em $C_{0}(G^{0})$-weight} on a $C_{0}(G^{0})$-algebra
  $C$ is a positive $C_{0}(G^{0})$-linear map $\phi \colon
  C \to C_{0}(G^{0})$. We denote the set of all such
  $C_{0}(G^{0})$-weights by ${\cal W}(C)$.
\end{definition}
Let $C$ be an admissible $C_{0}(G^{0})$-algebra.   The results
in \cite{blanchard:champs} imply:
\begin{lemma} \label{lemma:cx-weights} $\bigcap_{\phi \in
    {\cal W}(C)} \ker \phi = \{0\}$ and $ [\bigcup_{\phi \in
    {\cal W}(C)} \phi(C)]=C_{0}(G^{0})$. \qed
\end{lemma}
Let $\phi \in {\cal W}(C)$.  Then $\phi$ is completely
positive \cite[Theorem 3.9]{paulsen} and bounded \cite[Lemma
5.1]{lance}.  Let $E_{\phi}=C \tr_{\phi} \frakK$ (see
Section \ref{section:introduction}) and define $\eta_{\phi}
\colon C \to {\cal L}(E_{\phi})$ and $l_{\phi} \colon C \to
{\cal L}(\frakK,E_{\phi})$ by $\eta_{\phi}(c)(d \tr_{\phi}
\zeta) = cd \tr_{\phi} \zeta$ and $l_{\phi}(c)\zeta = c
\tr_{\phi} \zeta$ for all $c,d \in C$, $\zeta \in
\frakK$. One easily verifies that for all $c,d \in C$, $f
\in C_{0}(G^{0})$, $\zeta \in \frakK$,
\begin{gather} \label{eq:groupoid-cx-1}
  \begin{gathered}
    \begin{aligned}
      l_{\phi}(c)^{*}l_{\phi}(d)&=\phi(c^{*}d), &
      l_{\phi}(c)f &= l_{\phi}(cf),
    \end{aligned}
    \\
    \eta_{\phi}(c)(d \tr_{\phi} f\zeta) = cdf \tr_{\phi}
    \zeta = \eta_{\phi}(cf)(d \tr_{\phi} \zeta).
  \end{gathered}
\end{gather}
The {\em universal $C_{0}(G^{0})$-representation} $\eta_{C}
\colon C \to \mathcal{ L}(E_{C})$ of $C$ is the direct sum
of the representations $\eta_{\phi} \colon C \to \mathcal{
  L}(E_{\phi})$, where $\phi \in {\cal W}(C)$.  Denote by
$l_{C} \subseteq {\cal L}(\frakK,E_{C})$ the closed linear
span of all maps $l_{\phi}(c) \colon \frakK \to E_{\phi}
\hookrightarrow E_{C}$, where $c \in C$, $\phi \in {\cal
  W}(C)$.
\begin{lemma} \label{lemma:groupoid-cx-f}
  $\eta_{C}(C)^{l_{C}}_{E_{C}}$ is an admissible
  $C^{*}$-$\frakb$-algebra and $\eta_{C}$ is an isomorphism
  of $C_{0}(G^{0})$-algebras from $C$ to
  $(\eta_{C}(C),\rho_{l_{C}})$.
\end{lemma}
\begin{proof}
  The definition of $l_{C}$, the equations
  \eqref{eq:groupoid-cx-1} and Lemma \ref{lemma:cx-weights}
  imply that $[l_{C}\frakK] = \bigoplus_{\phi} E_{\phi} =
  E_{C}$, $[l_{C}^{*}l_{C}] = [\bigcup_{\phi} \phi(C)] =
  C_{0}(G^{0})$ and $[l_{C}C_{0}(G^{0})]=l_{C}$, whence
  $(E_{C},l_{C})$ is a $C^{*}$-$\frakb$-module, and that
  $[\eta_{C}(C)\rho_{l_{C}}(C_{0}(G^{0}))]=[\eta_{C}(CC_{0}(G^{0}))]=\eta_{C}(C)$
  and $[\eta_{C}(C)l_{C}]=l_{C}$, whence
  $\eta_{C}(C)^{l_{C}}_{E_{C}}$ is an admissible
  $C^{*}$-$\frakb$-algebra. Lemma \ref{lemma:cx-weights}
  implies that $\eta_{C}$ is injective and hence an
  isomorphism of $C$ onto $\eta_{C}(C)$, and the last
  equation in \eqref{eq:groupoid-cx-1} implies that
  $\eta_{C}(c)\rho_{\gamma}(f)=\eta_{C}(cf)$ for all $c\in
  C$, $f \in C_{0}(G^{0})$.
\end{proof}
\begin{theorem} \label{theorem:groupoid-cx} There exists an
  embedding $(\bfF,\bfG,\eta,\epsilon)$ of $\bfCG^{a}$ into
  $\bfb^{a}$ as a full and coreflective subcategory such
  that
  \begin{enumerate}
  \item $\bfF$ is given by $C \mapsto
    \eta_{C}(C)^{l_{C}}_{E_{C}}$ on objects and by $\bfF \pi
    \colon \eta_{C}(c) \mapsto \eta_{D}(\pi(c))$ for each
    morphism $\pi$ between objects $C$, $D$ in $\bfCG^{a}$;
  \item $\bfG$ is given by $C^{\gamma}_{K} \mapsto
    (C,\rho_{\gamma})$ on objects and $\pi \mapsto \pi$ on
    morphisms;
  \item $\eta_{C}$ is defined as above for each object $C$
    in $\bfCG^{a}$;
  \item $\epsilon_{{\cal C}} = \eta_{\bfG {\cal C}}^{-1}$
      for each object ${\cal C}$ in $\bfb^{a}$.
  \end{enumerate}
\end{theorem}
\begin{proof}[Proof of Theorem \ref{theorem:groupoid-cx}]
  The functor $\bfG \colon \bfb^{a} \to \bfCG^{a}$ is well defined
  by Lemma \ref{lemma:groupoid-cx-g} and evidently faithful.

  Let $C$ be an admissible $C_{0}(G^{0})$-algebra,
  ${\cal D}=D^{\gamma}_{K}$  an admissible
  $C^{*}$-$\frakb$-algebra, and $\pi \colon C \to \bfG
  {\cal D}$  a morphism in $\bfCG^{a}$. We claim that $\pi
  \circ \eta_{C}^{-1}$ is a
  morphism from $\bfF C$ to ${\cal D}$ in $\bfb^{a}$.  Let $\xi
  \in \gamma$. Then the map $\phi \colon C \to C_{0}(G^{0})
  \subseteq {\cal L}(\frakK)$ given by $c \mapsto
  \xi^{*}\pi(c)\xi$ is a $C_{0}(G^{0})$-weight, and there
  exists an isometry $S \colon E_{\phi} \to K$ such that
  $S(c \tr_{\phi} \zeta)=\pi(c)\xi \zeta$ for all $c \in C$,
  $\zeta \in \frakK$.  Denote by $P \colon E_{C} \to
  E_{\phi}$ the natural projection. Then $[SPl_{C}] =
  [Sl_{\phi}(C)]=[\pi(C)\xi]$ lies in $\gamma$ and contains
  $\xi$, and $SP\eta_{C}(c) = S\eta_{\phi}(c) = \pi(c)$ for
  each $c \in C$.  Since $\xi \in \gamma$ was arbitrary, the
  claim follows.

  Using Lemma \ref{lemma:groupoid-cx-f}, we conclude that
  $\bfF$ is well defined and that $\eta$ is a natural
  isomorphism from $\Id$ to $\bfG\bfF$.  Indeed,
  if $\pi \colon C \to D$ is a morphism in $\bfCG^{a}$, then
  $\bfF \pi = \eta_{D} \circ \pi \circ \eta_{C}^{-1}$ is a
  morphism from $\bfF C$ to $\bfF D$ by the argument above.

  Finally, let ${\cal D}$ be an admissible
  $C^{*}$-$\frakb$-algebra. The argument above, applied to
  the identity on $\bfG {\cal D}$, yields a
  morphism $\epsilon_{{\cal D}}$ from $\bfF \bfG {\cal D}$
  to ${\cal D}$ in $\bfb^{a}$ such that the composition 
  $\bfG {\cal D}
  \xrightarrow{\eta_{\bfG {\cal D}}} \bfG \bfF \bfG {\cal D}
  \xrightarrow{\bfG \epsilon_{{\cal D}}} \bfG {\cal D}$ is
  the identity. Since $\eta$ is a natural transformation,
  also $\epsilon \colon \bfF\bfG \to \Id$ is one.
  For each admissible $C_{0}(G^{0})$-algebra $C$, the
  composition $\bfF C \xrightarrow{\bfF \eta_{C}} \bfF \bfG
  \bfF C \xrightarrow{\epsilon_{\bfF C}} \bfF C$ is the
  identity by construction. From \cite[\S IV.2 Theorem
  2]{maclane}, we can conclude that $\bfF$ is a left adjoint
  to $\bfG$ such that $\eta$ and $\epsilon$ form the unit
  and counit, respectively, of the adjunction.  Since $\eta$
  is a natural isomorphism, $\bfF$ is full and faithful
  \cite[\S IV.3 Theorem 1]{maclane}.
\end{proof}

\paragraph{Actions of $G$ and coactions of $C_{0}(G)$}
We next embed the category of admissible actions of $G$ as a full and
coreflective subcategory into the category of all admissible
coactions of $C_{0}(G)$.

The definition of an action of $G$ requires the following
preliminaries.  Given $C_{0}(G^{0})$-algebras $(C,\rho)$ and
$(D,\sigma)$, where $D$ is commutative, we denote by $C
\cgtensor{\rho}{\sigma} D$ the $C_{0}(G^{0})$-tensor product
\cite{blanchard:tensor}, and drop the subscript $\rho$ or
$\sigma$ if this map is understood.  Given a
$C_{0}(G^{0})$-algebra $C$ and a continuous surjection $t
\colon G \to G^{0}$, we consider $C_{0}(G)$ as a
$C_{0}(G^{0})$-algebra via $t^{*} \colon C_{0}(G^{0}) \to
M(C_{0}(G))$ and let $t^{*}C:=C \cgtensor{}{t^{*}}
C_{0}(G)$, which is a $C_{0}(G)$-algebra in a natural
way. Each morphism $\pi$ of $C_{0}(G^{0})$-algebras $C,D$
induces a morphism of  $t^{*}\pi$ of $C_{0}(G)$-algebras
from $t^{*}C$ to $t^{*}D$ via $c
\cgtensor{}{} f \mapsto \pi(c) \cgtensor{}{} f$.  An {\em
  action} of $G$ on a $C_{0}(G^{0})$-algebra $C$ is an
isomorphism $\sigma \colon s^{*}C \to r^{*}C$ of
$C_{0}(G)$-algebras such that the restrictions of $\sigma$
to the fibers satisfy $\sigma_{x}\circ \sigma_{y} =
\sigma_{xy}$ for all $(x,y) \in \GsrG$ \cite{legall}.  A
{\em morphism} between actions $(C,\sigma^{C})$ and
$(D,\sigma^{D})$ of $G$ is a morphism of
$C_{0}(G^{0})$-algebras $\pi$ from $C$ to $D$ satisfying
$\sigma^{D} \circ s^{*}\pi = r^{*}\pi \circ \sigma^{C}$.

\begin{definition}
  We call an action $(C,\sigma)$ of $G$ {\em admissible} if
  the $C_{0}(G^{0})$-algebra $C$ is admissible, and we call
  a coaction $(C^{\gamma}_{K},\delta)$ of $C_{0}(G)$ {\em
    admissible} if $C^{\gamma}_{K}$ is an admissible
  $C^{*}$-$\frakb$-algebra and $[\delta(C)(1 \btensor
  C_{0}(G))]= C \btensor C_{0}(G)$ in $\mathcal{ L}(K
  \rtensor{\gamma}{\frakb}{\alpha} H)$.
\end{definition}
\begin{remark}
 If $\sigma$ is an action of $G$ on a continuous
    $C_{0}(G^{0})$-algebra, then the subset $Y:=\{u \in
    G^{0} \mid C_{u} \neq 0\} \subseteq G^{0}$ is open, $C$
    is an admissible $C_{0}(Y)$-algebra, and $\sigma$
    restricts to an action of the subgroupoid $G|_{Y}:=\{ x
    \in G \mid r(x),s(x) \in Y\} \subseteq G$.
  \end{remark}

  \begin{lemma} \label{lemma:groupoid-tensors} Let
    $C^{\gamma}_{K}$ and $D^{\epsilon}_{L}$ be admissible
    $C^{*}$-$\frakb$-algebras, where $D$ is
    commutative. Then there exists an isomorphism $C
    \cgtensor{\rho_{\gamma}}{\rho_{\epsilon}} D
    \to C \rtensor{\gamma}{\frakb}{\epsilon} D$, $ c
    \cgtensor{}{} d \mapsto c \btensor d$.
\end{lemma}
\begin{proof}
  Use \cite[Lemma 2.7]{blanchard:tensor} and apply
  \cite[Proposition 4.1]{blanchard:tensor} to the field of
  representations $C \hookrightarrow {\cal
    L}(K_{\gamma})\cong {\cal L}(\gamma)$, noting that $
  \gamma \tr_{\rho_{\epsilon}} D \cong [\kgamma{1}D]$ as a
  Hilbert $C^{*}$-$D$-module via $\xi \tr d \mapsto
  |\xi\rangle_{1}d$ and that $(C
  \rtensor{\gamma}{\frakb}{\epsilon} D)[\kgamma{1}D]
  \subseteq [\kgamma{1}D]$.
\end{proof}
We use the isomorphism above without further notice.
  \begin{proposition} \label{proposition:groupoid-actions}
    \begin{enumerate}
    \item Let $(C^{\gamma}_{K},\delta)$ be an admissible
      coaction of $C_{0}(G)$.  There exists a unique
      action $\sigma_{\delta}$ of $G$ on
      $(C,\rho_{\gamma})$  given by $c
      \cgtensor{}{} f \mapsto \delta(c)(1
      \btensor f)$.
    \item Let $(C,\sigma)$ be an admissible action of $G$.
      There exists a unique admissible, injective coaction
      $\delta_{\sigma}$ of $C_{0}(G)$ on $\bfF C$
      given by $ \eta_{C}(c) \mapsto
      (r^{*} \eta_{C} )(\sigma(c \cgtensor{}{}
      1))$.
    \end{enumerate}
  \end{proposition}
  \begin{proof}
    i) Since $\delta(C)$ and $1 \btensor C_{0}(G)$ commute,
    there exists a unique $*$-homomorphism $\tilde \sigma$
    from the algebraic tensor product $C \odot C_{0}(G)$ to
    $r^{*}C$ such that $\tilde \sigma(c \odot f)=\delta(c)(1
    \btensor f)$ for all $c\in C$, $f \in C_{0}(G)$. Since
    $\delta$ is a coaction, $\delta(c \rho_{\gamma}(g)) =
    \delta(c) \rho_{(\gamma \rt \hbeta)}(g) = \delta(c)(1
    \btensor s^{*}(g))$ for all $g \in C_{0}(G^{0})$.  From
    \cite[Lemma 2.7]{blanchard:tensor}, we can conclude that
    $\tilde \sigma$ factorizes to a $*$-homomorphism
    $\sigma=\sigma_{\delta} \colon s^{*}C \to r^{*}C$
    satisfying the formula in i). This $\sigma$ is
    surjective because $[\delta(C)(1 \btensor C_{0}(G))]=C
    \btensor C_{0}(G)$. In particular, $\sigma_{x}$ is
    surjective for each $x \in G$.  We claim that
    $\sigma_{x} \circ \sigma_{y}=\sigma_{xy}$ for all $(x,y)
    \in \GsrG$. Define $r_{1} \colon \GsrG \to G^{0}$ by
    $(x,y) \mapsto r(x)$.  By Lemma
    \ref{lemma:groupoid-tensors}, we have isomorphisms $C
    \rtensor{\gamma}{\frakb}{\alpha} C_{0}(G)
    \rtensor{\hbeta}{\frakb}{\alpha} C_{0}(G) \cong C
    \cgtensor{}{r^{*}} C_{0}(G) \cgtensor{s^{*}}{r^{*}}
    C_{0}(G) \cong C \cgtensor{}{r_{1}^{*}} C_{0}(\GsrG)
    \cong r_{1}^{*} C$.  Using formula
    \eqref{eq:groupoid-leg-ha}, we find
    \begin{gather} \label{eq:groupoid-rels}
      \begin{aligned}
        \sigma_{x} \circ \sigma_{y} \circ p_{s(y)} &=
        \sigma_{x} \circ p_{y} \circ \delta = p_{(x,y)}
        \circ (\delta \ast \Id) \circ \delta, \\
        \sigma_{xy} \circ p_{s(y)} &= p_{xy} \circ \delta =
        p_{(x,y)} \circ (\Id \ast \hDelta) \circ \delta,
      \end{aligned}
    \end{gather}
    and the claim follows.    
    Finally, $\sigma_{u}=\Id_{C_{u}}$ for each $u \in G^{0}$
    because $\sigma_{u}$ is surjective and idempotent, and
    $\sigma_{x}$ is injective for each $x \in G$ because
    $\sigma_{s(x)} = \sigma_{x^{-1}} \circ \sigma_{x}$ is
    injective. Therefore, $\sigma$ is injective.

    ii) Let $D:=\eta_{C}(C) \rtensor{l_{C}}{\frakb}{\alpha}
    C_{0}(G)$. Then ${\cal D}:=((E_{C})
    \rtensor{l_{C}}{\frakb}{\alpha} H_{\hbeta}, D)$ is an
    admissible $C^{*}$-$\frakb$-algebra.  Define $\delta
    \colon C \to D$ by $c \mapsto (r^{*}\eta_{C})(\sigma(c
    \cgtensor{}{} 1))$.  Let $c\in C$, $g \in
    C_{0}(G^{0})$. Then $cg \cgtensor{}{} 1= c \cgtensor{}{}
    s^{*}(g)$ in $M(C \cgtensor{}{s^{*}} C_{0}(G))$ and
    therefore $\delta(cg) = \delta(c)(1 \btensor s^{*}(g)) =
    \delta(c)\rho_{(l_{C} \rt \hbeta)}(g)$. Consequently,
    $\delta$ is a morphism of $C_{0}(G^{0})$-algebras from
    $C$ to $(D, \rho_{l_{C} \rt \hbeta})=\bfG {\cal D}$.  By
    definition of $\bfF$ and $\epsilon$, the morphism
    $\delta_{\sigma}:=\epsilon_{{\cal D}} \circ \bfF \delta
    \colon \bfF C \to \bfF \bfG {\cal D}\to {\cal D}$
    satisfies $\delta_{\sigma} \circ \eta_{ C} =
    \delta$, and a similar calculation as in
    \eqref{eq:groupoid-rels} shows that $(\delta_{\sigma}
    \ast \Id) \circ \delta_{\sigma} = (\Id \ast \hDelta)
    \circ \delta_{\sigma}$. Consequently, $\delta_{\sigma}$
    is a coaction of $(\hcA,\hDelta)$. Since $\sigma$ is
    injective, so are $\delta$ and
    $\delta_{\sigma}$. Finally, $\delta_{\sigma}$ is
    admissible because $[\delta_{\sigma}(\eta_{C}(C))(1
    \btensor C_{0}(G))] = (r^{*}\eta_{C})(\sigma(s^{*}C)) =
    r^{*}\eta_{C}(C) = [\eta_{C}(C) \btensor C_{0}(G)]$.
  \end{proof}
  \begin{corollary} \label{corollary:groupoid-action} Every
    admissible coaction of $C_{0}(G)$ is injective,
    left-full, and right-full.
\end{corollary}
\begin{proof}
  If $(C^{\gamma}_{K},\delta)$ is an admissible coaction,
  then $[\delta(C)\kalpha{2}]=[\delta(C)(1 \btensor
  C_{0}(G))\kalpha{2}] = [(C \btensor
  C_{0}(G))\kalpha{2}]=[\kalpha{2}C]$ and
  $[\delta(C)\kgamma{1}C_{0}(G)]=[\delta(C)(1 \btensor
  C_{0}(G))\kgamma{1}] = [(C \btensor
  C_{0}(G))\kgamma{1}]=[\kgamma{1} C_{0}(G)]$ because
  $[C_{0}(G) \alpha]=\alpha$ and
  $[C\gamma]=\gamma$. Finally, $\delta$ is injective because
  $\sigma_{\delta}$ is injective and $\delta(c) =
  \sigma_{\delta}(c \cgtensor{}{} 1)$ for all $c \in C$.
\end{proof}
\begin{proposition} \label{proposition:groupoid-actions-morphisms}
  Let $({\cal C},\delta^{{\cal C}})$, $({\cal
    D},\delta^{{\cal D}})$ be admissible coactions with
  associated actions
  $\sigma^{{\cal C}}=\sigma_{\delta^{{\cal C}}}$,
  $\sigma^{{\cal D}}=\sigma_{\delta^{{\cal D}}}$, and let
  $\pi \in \bfb^{a}({\cal C},{\cal D})= \bfCG^{a}(\bfG {\cal
    C}, \bfG {\cal D})$. Then $(\pi \ast \Id)\circ
  \delta^{{\cal C}} = \delta^{{\cal D}} \circ \pi$ if and
  only if $r^{*}\pi\ \circ \sigma^{{\cal C}} = \sigma^{{\cal
      D}} \circ s^{*}\pi$.
\end{proposition}
\begin{proof}
  Write ${\cal C}=C^{\gamma}_{K}$. The assertion holds
  because for all $c \in C$ and $f \in C_{0}(G)$,
  \begin{align*}
    ((\pi \ast \Id)(\delta^{{\cal C}}(c)))(1 \btensor f) &=
  (\pi \ast \Id)(\delta^{{\cal C}}(c)(1 \btensor f)) =
  (r^{*}\pi \circ \sigma^{{\cal C}})(c \cgtensor{}{} f), \\
  \delta^{{\cal D}}(\pi(c))(1 \btensor f) &=
  \sigma^{{\cal D}}(\pi(c) \cgtensor{}{} f) = (\sigma^{{\cal D}}
  \circ s^{*}\pi)(c \cgtensor{}{} f). \qedhere
\end{align*}
\end{proof}
We denote by $\bfact^{a}$ and $\bfcoact^{a}_{C_{0}(G)}$ the
categories of all admissible actions of $G$ and all
admissible coactions of $C_{0}(G)$, respectively.
\begin{theorem} \label{theorem:groupoid-actions} There
  exists an embedding
  $(\bfFhA,\bfGhA,\hat\eta,\hat\epsilon)$ of $\bfact^{a}$
  into $\bfcoact^{a}_{C_{0}(G)}$ as a full and coreflective
  subcategory such that
  \begin{enumerate}
  \item $\bfFhA$ is given by $(C,\sigma) \mapsto (\bfF
    C,\delta_{\sigma})$ on objects and $\pi \mapsto \bfF
    \pi$ on morphisms;
  \item  $\bfGhA$ is given by $({\cal C},\delta) \mapsto
    (\bfG {\cal C},\sigma_{\delta})$ on objects and $\pi
    \mapsto \bfG \pi=\pi$ on morphisms;
  \item $\hat\eta_{(C,\sigma)} = \eta_{C}$ and
    $\hat\epsilon_{({\cal C},\delta)} = \epsilon_{{\cal C}}$
    for all objects $(C,\sigma)$ and $({\cal C},\delta)$.
  \end{enumerate}
\end{theorem}
\begin{proof}
  The assignments $\bfGhA$ and $\bfFhA$ are well defined on
  objects and morphisms by Proposition
  \ref{proposition:groupoid-actions} and
  \ref{proposition:groupoid-actions-morphisms}.  For each
  admissible action $(C,\sigma)$, we have that $\eta_{C} \in
  \bfact^{a}((C,\sigma),\bfGhA\bfFhA(C,\sigma))$ because
  $\sigma_{\delta_{\sigma}}(\eta_{ C}(c) \cgtensor{}{} f) =
  \delta_{\sigma}(\eta_{C}(c))(1 \btensor f) =
  r^{*}\eta_{C}(\sigma(c \cgtensor{}{} f))$ for all $c \in
  C$, $f \in C_{0}(G)$, and Proposition
  \ref{proposition:groupoid-actions-morphisms} implies that
  $\epsilon_{{\cal C}}=\eta_{\bfG {\cal C}}^{-1} \in
  \bfcoact^{a}_{C_{0}(G)}(\bfFhA\bfGhA({\cal
    C},\delta),({\cal C},\delta))$ for each admissible
  coaction $({\cal C},\delta)$.  Now, the assertion follows
  from Theorem \ref{theorem:groupoid-cx}.
\end{proof}

\paragraph{Comparison of the associated reduced crossed products}
The reduced crossed product for an action $(C,\sigma)$ of
$G$ is defined as follows \cite{legall}.  The subspace
$C_{c}(G;C,\sigma):=C_{c}(G)r^{*}C \subseteq r^{*}C$ carries
the structure of a $*$-algebra and the
structure of a pre-Hilbert $C^{*}$-module over $C$ such that
\begin{align*}
  (ab)_{x} &= \int_{G^{r(x)}} a_{y} \sigma_{y}(b_{y^{-1}x})
  \intd \lambda^{r(x)}(y),  & (a^{*})_{x} &=
  \sigma_{x}(a^{*}_{x^{-1}}), \\
  \langle a|b\rangle_{u} &= \int_{G^{u}}
  \sigma_{y}((a_{y^{-1}})^{*}b_{y^{-1}}) \intd
  \lambda^{u}(y) = (a^{*}b)_{u}, & (ac)_{x} &=
  a_{x}\sigma_{x}(c_{s(x)})
\end{align*}
for all $a,b \in C_{c}(G;C,\sigma)$, $u \in G^{0}$ and $c
\in C$, $x \in G$. Denote the completion of this pre-Hilbert
$C^{*}$-module by $L^{2}(G,\lambda^{-1};C,\sigma)$.  Using
the relation $\langle a|bd\rangle_{u}=(abd)_{u}=\langle
b^{*}a|d\rangle_{u}$, which holds for all $a,b,d \in
C_{c}(G;C,\sigma)$, $u \in G^{0}$, and a routine norm
estimate, one verifies the existence of a $*$-homomorphism
$\pi \colon C_{c}(G;C,\sigma) \to \mathcal{
  L}(L^{2}(G,\lambda^{-1};C,\sigma))$ such that $\pi(b)d =
bd$ for all $b,d \in C_{c}(G;C,\sigma)$.  Then the {\em reduced
  crossed product} of $(C,\sigma)$ is the $C^{*}$-algebra $C
\rtimes_{\sigma,r} G:= [\pi(C_{c}(G;C,\sigma))] \subseteq
\mathcal{ L}(L^{2}(G,\lambda^{-1};C,\sigma))$.
\begin{proposition} \label{proposition:groupoid-crossed} Let
  $(C^{\gamma}_{K},\delta)$ be an admissible coaction of
  $C_{0}(G)$, consider $C$ as a $C_{0}(G^{0})$-algebra via
  $\rho_{\gamma}$, and let $\sigma=\sigma_{\delta}$. Then
  there exists an isomorphism $ C
  \rtimes_{\sigma_{\delta},r} G \to C \rtimes_{r}
  C^{*}_{r}(G)$ given by $\pi(c \cgtensor{}{} f) \mapsto
  \delta(c) (\Id \btensor UL(f)U)$ for all $c \in C$, $f\in
  C_{c}(G)$.
\end{proposition}
\begin{proof}
  Let $\delta_{U}:=\Ad_{(\Id \btensor U)} \circ \delta
  \colon C \to \mathcal{ L}(K
  \rtensor{\gamma}{\frakb}{\halpha} H)$.  We equip
  $C_{c}(G;C,\sigma)$ with the structure of a pre-Hilbert
  $C^{*}$-module over $C$ such that $\langle a|b\rangle_{u}
  = \int_{G_{u}} (a_{x})^{*}b_{x} \intd \lambda_{u}^{-1}(x)$
  and $(ac)_{x} = a_{x}c_{s(x)}$ for all $a,b \in
  C_{c}(G;C,\sigma)$, $c \in C$, $u \in G^{0}$, and denote
  by $L^{2}(G,\lambda^{-1};C)$ the completion.  One easily
  checks that there exists a unique unitary $\Phi \colon
  L^{2}(G,\lambda^{-1};C) \to [|\halpha\rangle_{2}C] =
  [\delta_{U}(C)|\halpha\rangle_{2}\rangle]$ given by $c
  \cgtensor{}{} f \mapsto |\hat j(f)\rangle_{2}c$, and that
  for all $ c\in C, f \in C_{c}(G), y \in G$,
  \begin{align*} 
    \Phi^{-1}(\delta_{U}(c)|\hat j(f)\rangle_{2})_{y} &=
    \sigma_{y^{-1}}(c_{r(y)}) f(y).
  \end{align*}
  Hence, there exists a unitary $\Psi \colon
  L^{2}(G,\lambda^{-1};C,\sigma)\to
  [\delta_{U}(C)|\halpha\rangle_{2}\rangle]$ given by $c
  \cgtensor{}{} f \mapsto \delta_{U}(c)|\hat
j(f)\rangle_{2}$. Let $c,d\in C$, $f,g \in C_{c}(G)$ and
$\omega= \Phi^{-1}(\Psi(c \cgtensor{}{} f))$.  Then
$\delta_{U}(d)(\Id \btensor L(g))\Psi=\Psi\pi(d
\cgtensor{}{} g)$ because for all $x \in G$,
  \begin{align*}
    \Phi^{-1}(\delta_{U}(d)(\Id \btensor L(g))\Phi(\omega))_{x}
    &= \int_{G_{u}} \sigma_{x^{-1}}(d_{r(x)}) g(xy^{-1})
    \omega_{y}
    \intd \lambda^{-1}_{u}(y)   \\
    &= \int_{G_{u}} \sigma_{x^{-1}}
    \big(d_{r(xy^{-1})}g(xy^{-1})\sigma_{xy^{-1}}(c_{r(y)})f(y)
    \big)
    \intd\lambda_{u}^{-1}(y)\\
    &= \Phi^{-1}(\Psi(\pi(d \cgtensor{}{} g)(c \cgtensor{}{} f)))_{x}.
  \end{align*}
  Since $d \in C$ and $g \in C_{c}(G)$ were arbitrary, the
  assertion follows.
\end{proof}

\section{Fell bundles on groupoids}
 \label{section:prelim}

 We now gather preliminaries on Fell bundles that are needed
 in Sections \ref{section:fell} and \ref{section:etale}.  We
 use the notion of a Banach bundle and standard notation; a
 reference is \cite{dupre}.

\paragraph{Fell bundles on groupoids and their
  $C^{*}$-algebras}
We first recall the notion of a Fell bundle on $G$ and the
definition of the associated reduced $C^{*}$-algebra
\cite{kumjian}.  Given an upper semicontinuous Banach
bundle $p \colon \mathcal{ F} \to G$, denote by $\mathcal{
  F}^{0}$ the restriction of $\mathcal{ F}$ to $G^{0}$, by
$\FsrF$ the restriction of $\mathcal{ F} \times \mathcal{
  F}$ to $\GsrG$,  by ${\cal F}_{x}$ for each $x \in G$
the fiber at $x$, by $\Gamma_{c}({\cal F})$ the space of
continuous sections of ${\cal F}$ with compact support, and
by $\Gamma_{0}({\cal F}^{0})$ the space of continuous
sections of ${\cal F}^{0}$ that vanish at infinity in
norm.  
\begin{definition} \label{definition:fell-bundle} A {\em
    Fell bundle} on $G$ is an upper semicontinuous Banach
  bundle $p \colon \mathcal{ F} \to G$ with a continuous
  multiplication $\FsrF \to \mathcal{ F}$ and a
  continuous involution $*\colon \mathcal{ F} \to \mathcal{
    F}$ such that for all $e \in \mathcal{ F}$,
  $(e_{1},e_{2}) \in \FsrF$, $(x,y) \in \GsrG$,
  \begin{enumerate}
  \item $p(e_{1}e_{2}) = p(e_{1})p(e_{2})$ and $p(e^{*}) =
    p(e)^{-1}$;
  \item the  map $\mathcal{ F}_{x} \times \mathcal{
      F}_{y} \to \mathcal{ F}_{xy}, (e'_{1},e'_{2}) \mapsto
    e'_{1}e'_{2}$, is bilinear and the map $\mathcal{ F}_{x}\to 
    \mathcal{ F}_{x^{-1}}$, $e' \mapsto e'{}^{*}$, is
    conjugate linear;
  \item $(e_{1}e_{2})e_{3}=e_{1}(e_{2}e_{3}), \,
    (e_{1}e_{2})^{*}=e^{*}_{2}e^{*}_{1}$, and $(e^{*})^{*} =
    e$;
  \item $\|e_{1}e_{2}\| \leq \|e_{1}\|\|e_{2}\|$,
    $\|e^{*}e\| = \|e\|^{2}$, and $e^{*}e \geq 0$ in the
    $C^{*}$-algebra $\mathcal{ F}_{s(p(e))}$.
  \end{enumerate}
  We call $\mathcal{ F}$ {\em saturated} if $[\mathcal{
    F}_{x}\mathcal{F}_{y}]=\mathcal{F}_{xy}$ for all $(x,y) \in
  \GsrG$, and {\em admissible} if $\Gamma_{0}(\mathcal{
    F}^{0})$ is an admissible $C_{0}(G^{0})$-algebra
  with respect to the pointwise operations.
\end{definition}
Let ${\cal F}$ be a Fell bundle on $G$.  The associated
reduced $C^{*}$-algebra is defined as follows.  The space
$\Gamma_{c}(\mathcal{ F})$ is a $*$-algebra with respect to
the multiplication and involution given by
\begin{align} \label{eq:fell-algebra} (cd)(x)
  &=\int_{G^{r(x)}} c(y)d(y^{-1}x) \intd\lambda^{r(x)}(y) =
  \int_{G_{s(x)}} c(xz^{-1})d(z)\intd\lambda^{-1}_{s(x)}(z)
\end{align}
and $c^{*}(x) = c(x^{-1})^{*}$, respectively, and a
pre-Hilbert $C^{*}$-module over $\Gamma_{0}(\mathcal{
  F}^{0})$ with respect to the structure maps
\begin{align*}
  \langle c|d\rangle(u) &=  \int_{G_{u}} c(x)^{*}d(x)
  \intd \lambda^{-1}_{u}(x) = (c^{*}d)(u), &
  (c e)(x) &= c(x)e(s(x)),
\end{align*}
where $c,d \in \Gamma_{c}(\mathcal{ F})$, $e\in
\Gamma_{0}(\mathcal{ F}^{0})$, $x\in G$.  Denote by
$\Gamma^{2}(\mathcal{ F},\lambda^{-1})$ the completion of
this pre-Hilbert $C^{*}$-module. Then there exists a
$*$-homomorphism
\begin{align*}
  L_{\mathcal{ F}} \colon \Gamma_{c}(\mathcal{ F}) \to \mathcal{
    L}(\Gamma^{2}(\mathcal{
    F},\lambda^{-1})), \quad L_{\mathcal{ F}}(a)b = ab \text{
    for all } a,b \in \Gamma_{c}(\mathcal{ F}),
\end{align*}
and $C^{*}_{r}(\mathcal{ F}):=[L_{\mathcal{
    F}}(\Gamma_{c}(\mathcal{ F}))] \subseteq \mathcal{
  L}(\Gamma^{2}(\mathcal{ F},\lambda^{-1}))$ is the {\em
  reduced $C^{*}$-algebra of $\mathcal{ F}$}. We 
identify $\Gamma_{c}({\cal F})$ with $L_{{\cal
    F}}(\Gamma_{c}({\cal F})) \subseteq C^{*}_{r}({\cal F})$
via $L_{{\cal F}}$.
 
We equip $\Gamma_{c}(\mathcal{ F})$ with the {\em inductive
  limit topology}; thus, a net converges if it converges
uniformly and if the supports of its members are contained
in some compact set.  We shall frequently use the following
general result; see \cite[Proposition 2.3]{dupre}.
\begin{lemma} \label{lemma:submodule} 
 Let  $ {\cal E}$ be an upper semicontinuous
Banach bundle on a locally compact, second countable, Hausdorff
space $X$ and let $\Gamma' \subseteq
  \Gamma_{c}({\cal E})$ be a subspace such that
  \begin{enumerate}
  \item $\Gamma'$ is closed under pointwise multiplication
    with elements of $C_{c}(X)$;
  \item $\{ f(x) \mid f \in \Gamma'\} \subseteq {\cal
      E}_{x}$ is dense for each $x \in X$.
  \end{enumerate}
  Then $\Gamma'$ is dense in $\Gamma_{c}({\cal E})$. \qed
\end{lemma}
Given $f \in \Gamma_{c}({\cal F})$ and $g \in
\Gamma_{0}({\cal F}^{0})$, define $fg, gf \in
\Gamma_{c}({\cal F})$ by $(fg)(x)=f(x)g(s(x))$,
$(gf)(x)=g(r(x))f(x)$ for all $x \in G$. Using the relation
$[{\cal F}_{x}]=[{\cal F}_{x}{\cal F}_{x}^{*}{\cal F}_{x}]$,
where $x \in G$, and Lemma \ref{lemma:submodule}, we find:
\begin{lemma} \label{lemma:gamma0}
  $\Gamma_{c}({\cal F})\Gamma_{0}({\cal F}^{0})$ and
  $\Gamma_{0}({\cal F}^{0})\Gamma_{c}({\cal F})$ are
  linearly dense
  in $\Gamma_{c}({\cal F})$. \qed
\end{lemma}
\paragraph{The multiplier bundle  of a Fell bundle}
Given a Fell bundle ${\cal F}$ on $G$, we define a
multiplier bundle ${\cal M(F)}$ on $G$, extending  the
definition in \cite[\S
VIII.2.14]{fell:rep2}.  Given a subspace $C \subseteq G$, we denote
by ${\cal F}|_{C}$ the restriction of ${\cal F}$ to $C$.
\begin{definition}
  Let $x \in G$.  A {\em multiplier of $\mathcal{ F}$ of
    order $x$} is a map $T \colon \mathcal{ F}|_{G^{s(x)}}
  \to \mathcal{ F}|_{G^{r(x)}}$ such that $T\mathcal{ F}_{y}
  \subseteq \mathcal{ F}_{xy}$ for all $y \in G^{s(x)}$ and
  such that there exists a map $T^{*} \colon \mathcal{
    F}|_{G^{r(x)}} \to \mathcal{ F}|_{G^{s(x)}}$ such that
  $e^{*}Tf = (T^{*}e)^{*}f$ for all $e \in \mathcal{
    F}|_{G^{r(x)}}$, $f \in \mathcal{ F}|_{G^{s(x)}}$. We
  denote by $\mathcal{ M}(\mathcal{ F})_{x}$ the set of all
  multipliers of $\mathcal{ F}$ of order $x$.
\end{definition}
As for adjointable operators of Hilbert $C^{*}$-modules, one
deduces from the definition the following simple properties.
Let $x \in G$. Then for each $T \in \mathcal{ M}(\mathcal{
  F})_{x}$, the map $T^{*}$ is uniquely determined, $T^{*}
\in \mathcal{ M}(\mathcal{ F})_{x^{-1}}$, and $T^{**}=T$.
Moreover, each $T \in \mathcal{ M}(\mathcal{ F}_{x})$ is
fiberwise linear in the sense that $T(\kappa e+f)=\kappa
Te+Tf$ for all $\kappa \in \complex$, $e,f \in \mathcal{
  F}_{y}$, $y \in G^{s(x)}$.  The restrictions $T_{s(x)}
\colon \mathcal{ F}_{s(x)} \to \mathcal{ F}_{x}$ and
$(T^{*})_{x} \colon \mathcal{ F}_{x} \to \mathcal{
  F}_{s(x)}$ are adjoint operators of Hilbert
$C^{*}$-modules over $\mathcal{ F}_{s(x)}$, and since
$\mathcal{ F}_{y} = [\mathcal{ F}_{r(y)}\mathcal{ F}_{y}]$
for each $y \in G^{s(x)}$, the map $\mathcal{ M}(\mathcal{
  F})_{x} \to \mathcal{ L}(\mathcal{ F}_{s(x)},\mathcal{
  F}_{x})$, $T \mapsto T_{s(x)}$, is a bijection. Clearly,
we have a natural embedding ${\cal F}_{x} \hookrightarrow
{\cal M(F)}_{x}$, where each $f \in {\cal F}$ acts as a
multiplier via left multiplication. For each $y \in
G^{s(x)}$, we have $\mathcal{ M}(\mathcal{ F})_{x}\mathcal{
  M}(\mathcal{ F})_{y} \subseteq \mathcal{ M}(\mathcal{
  F})_{xy}$, and for each $f \in {\cal F}_{z}$, $z \in
G_{r(x)}$, we let $fT:=(T^{*}f^{*})^{*}$.
\begin{definition}
  For each $x \in G$, consider $\mathcal{ M}(\mathcal{
    F})_{x}$ as a Banach space via the identification with
  $\mathcal{ L}(\mathcal{ F}_{s(x)},\mathcal{ F}_{x})$.  Let
  $\mathcal{ M}(\mathcal{ F})= \coprod_{x \in G} \mathcal{
    M}(\mathcal{ F})_{x}$ and denote by $\tilde p \colon
  \mathcal{ M}(\mathcal{ F}) \to G$ the natural map. The
  {\em strict topology} on $\mathcal{ M}(\mathcal{ F})$ is
  the weakest topology that makes $\tilde p$ and the maps
  $\mathcal{ M}(\mathcal{ F}) \to \mathcal{ F}$ of the form
  $c \mapsto c\cdot d(s(\tilde p(c)))$ and $c \mapsto
  d(r(\tilde p(c))) \cdot c$ continuous for each $d \in
  \Gamma_{c}(\mathcal{ F}^{0})$.  Denote by
  $\Gamma_{c}(\mathcal{ M}(\mathcal{ F}))$ the space of all
  sections that are strictly continuous, norm-bounded, and
  compactly supported. 
\end{definition}
\begin{remark}
  The bundle ${\cal M(F)}$ satisfies all axioms of a Fell
  bundle except for the fact that it is no Banach bundle
  with respect to the strict topology unless ${\cal
    M(F)}={\cal F}$. Indeed, for each $u\in G^{0}$, the
  subspace topology on $\mathcal{ M}(\mathcal{
    F})_{u}\cong\mathcal{ L}(\mathcal{ F}_{u})\cong
  M(\mathcal{ F}_{u})$ is the strict topology and coincides
  with the norm topology only if $M(\mathcal{
    F}_{u})=\mathcal{ F}_{u}$.
\end{remark}
Given $f \in \Gamma_{c}({\cal M}({\cal F}))$ and $g \in
\Gamma_{0}({\cal F}^{0})$, define $fg, gf \in
\Gamma_{c}({\cal F})$ by $(fg)(x)=f(x)g(s(x))$,
$(gf)(x)=g(r(x))f(x)$ for all $x \in G$ again.
\begin{lemma} \label{lemma:fell-multiplier-sections}
  \begin{enumerate}
  \item Let $c \in \Gamma_{c}(\mathcal{ M}(\mathcal{ F}))$
    and $d \in \Gamma_{c}(\mathcal{ F})$. Then there exists
    a section $cd \in \Gamma_{c}(\mathcal{ F})$ such that
    $(cd)(x)=\int_{G^{r(x)}}
    c(y)d(y^{-1}x)\intd\lambda^{r(x)}(y)$ for all $x \in G$.
  \item $\Gamma_{c}(\mathcal{ M}(\mathcal{ F}))$ carries a
    structure of a $*$-algebra such that $c^{*}(x)
    =c(x^{-1})^{*}$ and $(cd)(x)e = \int_{G^{r(x)}}
    c(y)d(y^{-1}x)e \intd\lambda^{r(x)}$ for all $c,d \in
    \Gamma_{c}(\mathcal{ M}(\mathcal{ F}))$, $x \in G$, $e
    \in \mathcal{ F}_{s(x)}$.
  \item There exists a $*$-homomorphism $L_{\mathcal{
        M}(\mathcal{ F})} \colon \Gamma_{c}(\mathcal{
      M}(\mathcal{ F})) \to M(C^{*}_{r}(\mathcal{ F}))$ such
    that $L_{\mathcal{ M}(\mathcal{ F})}(c)L_{\mathcal{
        F}}(d) = L_{\mathcal{ F}}(cd)$ for
    all $c \in \Gamma_{c}(\mathcal{ M}(\mathcal{ F}))$, $d
    \in \Gamma_{c}(\mathcal{ F})$.
  \item $\Gamma_{c}({\cal M}({\cal F}))$ is closed under
    pointwise multiplication with elements of $C_{c}(G)$.
  \end{enumerate}
\end{lemma}
\begin{proof}
  i) Define $cd\colon G \to {\cal F}$ as above, and let
  $\epsilon > 0$. Using Lemma \ref{lemma:gamma0}, we find a
  sequence $(g_{n})_{n}$ in the span of $\Gamma_{0}({\cal
    F}^{0})\Gamma_{c}({\cal F})$ that converges to $d$ in
  the inductive limit topology.  Since $\Gamma_{c}({\cal
    M}({\cal F}))\Gamma_{0}(\mathcal{ F}^{0}) \subseteq
  \Gamma_{c}({\cal F})$, the map $h_{n} \colon x \mapsto
  \int_{G^{r(x)}} c(y)g_{n}(y^{-1}x) \intd\lambda^{r(x)}(y)$
  lies in $\Gamma_{c}(\mathcal{ F})$ for each $n$. Using the
  fact that $c$ has compact support and bounded norm, one
  easily concludes that $(h_{n})_{n}$ converges in the
  inductive limit topology to $cd$ which therefore is in
  $\Gamma_{c}({\cal F})$.

  ii) Note that $(cd)(x)$ is well defined
  because the map $y \mapsto d(y^{-1}x)e$ is in
  $\Gamma_{c}(\mathcal{ F})$ and thus i) applies. Now, the assertion
  follows from standard arguments.

  iii) One easily verifies that there exists a
  representation $L_{\mathcal{ M}(\mathcal{ F})} \colon
  \Gamma_{c}(\mathcal{ M}(\mathcal{ F})) \to \mathcal{
    L}(\Gamma^{2}(\mathcal{ F}))$ such that $L_{\mathcal{
      M}(\mathcal{ F})}(c)d = cd$ for all $c \in
  \Gamma_{c}(\mathcal{ M}(\mathcal{ F}))$, $d \in
  \Gamma_{c}(\mathcal{ F})$, and that $L_{\mathcal{
      M}(\mathcal{ F})}(c)L_{\mathcal{ F}}(d)e =
  cde=L_\mathcal{ F}(cd)e$ for all $c \in
  \Gamma_{c}(\mathcal{ M}(\mathcal{ F}))$, $d,e \in
  \Gamma_{c}(\mathcal{ F})$.

  iv) This follows immediately from the fact that
  $\Gamma_{c}({\cal F})$ is closed under pointwise
  multiplication by elements of $C_{c}(G)$.
\end{proof}

\paragraph{Morphisms between Fell bundles}
Let $\mathcal{ F}$ and $\mathcal{ G}$ be Fell bundles on
$G$.
\begin{definition} \label{definition:fell-morphism} A {\em
    (fibrewise nondegenerate) morphism} from ${\cal F}$ to
  ${\cal G}$ is a continuous map $T \colon \mathcal{ F} \to
  \mathcal{M}(\mathcal{ G})$ that satisfies the following
  conditions:
  \begin{enumerate}
  \item for each $x \in G$, the map $T$ restricts to a
    linear map $T_{x} \colon \mathcal{ F}_{x} \to
    \mathcal{M( G)}_{x}$;
  \item $T(e_{1})T(e_{2})=T(e_{1}e_{2})$ and
    $T(e)^{*}=T(e^{*})$ for all $(e_{1},e_{2}) \in \FsrF$, $e \in \mathcal{ F}$;
  \item ${\cal G}_{x} = [T(\mathcal{ F}_{x})\mathcal{
      G}_{s(x)}]$ for each $x \in G^{0}$.
  \end{enumerate}
\end{definition}
Let $T$ be a morphism from ${\cal F}$ to ${\cal G}$. Then
$T_{u} \colon \mathcal{ F}_{u} \to \mathcal{M(G)}_{u}$ is a
nondegenerate $*$-homomorphism for each $u \in G^{0}$; in
particular, $\|T_{u}\|\leq 1$. One easily concludes that
$\|T_{x}\| \leq 1$ for each $x \in G$. Hence, the formula $f
\mapsto T\circ f$ defines $*$-homomorphisms $T_{*} \colon
\Gamma_{c}(\mathcal{ F}) \to \Gamma_{c}(\mathcal{M(G)})$ and
$T^{0}_{*} \colon \Gamma_{0}({\cal F}^{0}) \to
M(\Gamma_{0}({\cal G}^{0}))$.
\begin{proposition} \label{proposition:fell-morphism} 
  \begin{enumerate}
  \item $T^{0}_{*} \colon \Gamma_{0}({\cal F}^{0}) \to
    M(\Gamma_{0}({\cal G}^{0}))$ is  nondegenerate.
  \item $T_{*}(\Gamma_{c}(\mathcal{ F}))\Gamma_{c}(\mathcal{
      G}^{0})$ is dense in $\Gamma_{c}(\mathcal{ F})$.
  \item $T_{*}$ extends to a nondegenerate $*$-homomorphism
    $T_{*} \colon C^{*}_{r}({\cal F}) \to M(C^{*}_{r}({\cal
      G}))$.
  \end{enumerate}
\end{proposition}
\begin{proof}
  i), ii) This follows immediately from Lemma
  \ref{lemma:submodule} and
  \ref{lemma:fell-multiplier-sections}. 

  iii) Part ii) and a straightforward calculation show that
  there exists a unitary $\Psi \colon \Gamma^{2}({\cal
    F},\lambda^{-1}) \tr_{T^{0}_{*}} \Gamma_{0}({\cal
    G}^{0}) \to \Gamma^{2}({\cal G},\lambda^{-1})$ such that
  $(\Psi(f \tr g))(x) = T_{*}(f)g$ for all $f \in
  \Gamma_{c}({\cal F})$, $g \in \Gamma_{0}({\cal G}^{0})$.
  The map $C^{*}_{r}({\cal F}) \to {\cal L}(\Gamma^{2}({\cal
    G},\lambda^{-1}))$ given by $f \mapsto \Psi(f \tr
  \Id)\Psi^{*}$ is the desired extension. Lemma
  \ref{lemma:gamma0} and part ii)
  imply that $[T_{*}(\Gamma_{c}({\cal F}))\Gamma_{c}({\cal
    G})] = [T_{*}(\Gamma_{c}({\cal F}))\Gamma_{0}({\cal
    G}^{0})\Gamma_{c}({\cal G})]=[\Gamma_{c}({\cal
    G})\Gamma_{c}({\cal G})] = C^{*}_{r}({\cal G})$.
\end{proof}

\section{From Fell bundles on $G$ to  coactions of
  $C^{*}_{r}(G)$}

\label{section:fell}

Let $G$ be a groupoid, $V$ the associated
$C^{*}$-pseudo-multiplicative unitary, and $C_{r}^{*}(G)$
or, more precisely, $(\cA,\Delta)$ the associated Hopf
$C^{*}$-bimodule as in Section
\ref{section:kac-groupoid}. We relate Fell bundles on $G$ to
coactions of $C^{*}_{r}(G)$ as follows.  Let ${\cal F}$ be
an admissible Fell bundle ${\cal F}$ on $G$.  We shall
construct a coaction of $C^{*}_{r}(G)$ on $C^{*}_{r}({\cal
  F})$ which is unitarily implemented by a representation of
$V$, and identify the reduced crossed product of this
coaction with the reduced $C^{*}$-algebra of another Fell
bundle. Finally, we show that this construction is
functorial. 

Recall that a {\em representation of unitary $V$} is a
$C^{*}$-$(\frakb,\frakbo)$-module $\cKhd$ together with a
unitary $X \colon \rHsource \to \rHrange$ that satisfies
$X(\gamma \lt \alpha) = \gamma \rt \alpha$, $X(\hdelta \rt
\beta) = \hdelta \lt \beta$, $ X(\hdelta \rt \hbeta) =
\gamma \rt \hbeta$, and $X_{12}X_{13}V_{23}=V_{23}X_{12}$
\cite[Definition 4.1]{timmermann:cpmu-hopf}.  We construct a
coaction out of such a representation as follows.
\begin{lemma} \label{lemma:rep-coaction} Let $(\cKhd,X)$ be
  a representation of $V$, let $ C^{\gamma}_{K}$ be a
  $C^{*}$-$\frakb$-algebra such that
  $[C,\rho_{\hbeta}(\frakB)] = 0$, define $\delta \colon C
  \to {\cal L}(\rHrange)$ by $c \mapsto X(c \botensor
  \Id)X^{*}$, and assume that
  $[\delta(C)\kgamma{1}A]\subseteq [\kgamma{1}A]$ and
  $[\delta(C)\kbeta{2}]\subseteq [\kbeta{2}C]$. Then
  $\delta$ is injective, a morphism from $(K_{\gamma},C)$ to
  $(\rHrange_{\alpha}, C \fib{\gamma}{\frakb}{\beta} A)$,
  and a coaction of $(\cA,\Delta)$ on $C^{\gamma}_{K}$. If
  the inclusions above are equalities, then $\delta$ is
  left- or right-full, respectively.
\end{lemma}
\begin{proof}
  Evidently, $\delta$ is injective.  It is a morphism of
  $C^{*}$-$\frakb$-algebras  because
  $X|\xi\rangle_{2}c = \delta(c)X|\xi\rangle_{2}$ for each
  $\xi \in \alpha$, $c \in C$ and because $[X\kalpha{2}\gamma] =
  \gamma \rt \alpha$ and $[(X \kalpha{2})^{*}(\gamma \rt
  \alpha)] = [\balpha{2}(\gamma \lt \alpha)]=\gamma$.  A
  similar diagram as in \cite[Proof of Lemma
  3.13]{timmermann:cpmu-hopf} shows that $(\delta \ast
  \Id)(\delta(c)) = (\Id \ast \Delta)(\delta(c))$ for
  each $c \in C$.  
\end{proof}
\paragraph{The  representation of $V$ associated to ${\cal F}$}
Denote by ${\cal W}={\cal W}(\Gamma_{0}({\cal F}^{0}))$
 the set of all  $C_{0}(G^{0})$-weights on
$\Gamma_{0}(\mathcal{ F}^{0})$ and let $\phi \in {\cal W}$.
\begin{lemma} \label{lemma:groupoid-bundle-continuous} Let
  $c,d \in \Gamma_{c}(\mathcal{ F})$. Then the map $x
  \mapsto \phi_{s(x)}(c(x)^{*}d(x))$ lies in $C_{c}(G)$.
\end{lemma}
\begin{proof}
  The function $G \to s^{*}\mathcal{ F}^{0}$ given by $x
  \mapsto c(x)^{*}d(x)$ is continuous and has compact
  support, and the composition $h\colon x \mapsto
  \phi_{s(x)}(c(x)^{*}d(x))$ is continuous because the map
  $\mathcal{ F}^{0} \to \complex$ given by $f \mapsto
  \phi_{p(f)}(f)$ is continuous. 
\end{proof}
Define Hilbert $C^{*}$-$C_{0}(G^{0})$-modules
$\Gamma^{2}(\mathcal{ F},\lambda;\phi)$,
$\Gamma^{2}(\mathcal{ F},\lambda^{-1};\phi)$ and a Hilbert
space $K_{\phi}=\Gamma^{2}(\mathcal{ F},\nu;\phi)$ as
the respective completions of $\Gamma_{c}(\mathcal{ F})$,
where for all $c,d \in \Gamma_{c}(\mathcal{ F})$, $f \in
C_{0}(G^{0})$, the inner product $\langle c|d\rangle$ and
the product $cf$ are given by
\begin{align*}
u &\mapsto \int_{G^{u}} \phi_{s(x)}
 (c(x)^{*}d(x)) \intd \lambda^{u}(x), & y &\mapsto
 c(y)f(r(y)) && \text{in case of } \Gamma^{2}(\mathcal{
   F},\lambda;\phi), \\
u &\mapsto \int_{G_{u}}
 \phi_{s(x)}(c(x)^{*}d(x)) \intd \lambda_{u}^{-1}(x), &
y &\mapsto c(y)f(s(y)) && \text{in case of }
 \Gamma^{2}(\mathcal{ F},\lambda^{-1};\phi), \\
 &\text{and }  \int_{G} \phi_{s(x)} (c(x)^{*}d(x))
 \intd \nu(x) &&&&\text{in case of } \Gamma^{2}(\mathcal{
   F},\nu;\phi).
\end{align*}
\begin{lemma} \label{lemma:fell-full} $[\langle
  E|E\rangle]=[\phi(\Gamma_{0}({\cal F}^{0}))]$ for $E \in
  \{\Gamma^{2}({\cal F},\lambda,\phi), \Gamma^{2}({\cal
    F},\lambda^{-1};\phi)\}$.
\end{lemma}
\begin{proof}
  Assume that $(\phi(c^{*}c))(u) \neq 0$ for some $c \in
  \Gamma_{c}({\cal F}^{0})$, $u \in G^{0}$. Choose $d \in
  \Gamma_{c}({\cal F})$ such that $d|_{G^{0}}=c$. Then the
  function on $G$ given by $x \mapsto
  \phi_{s(x)}(d(x)^{*}d(x))$ is non-negative and nonzero at
  $u$, whence $\langle d|d\rangle_{E}(u) \neq 0$. Now, the
  assertion follows because $[\langle E|E\rangle]$ and
  $[\phi(\Gamma_{0}({\cal F}^{0}))]$ are closed ideals in
  $C_{0}(G^{0})$.
\end{proof}
Let $K=\bigoplus_{\phi \in {\cal W}} K_{\phi}$ and identify
each $K_{\phi}$ with a subspace of $K$.  Given $c \in
\Gamma_{c}({\cal F})$ and $f \in C_{0}(G^{0})$, define $fc,
cf, cD^{-1/2} \in \Gamma_{c}({\cal F})$ by
\begin{align*}
  fc &\colon x \mapsto f(r(x))c(x), & cf &\colon x \mapsto
  c(x)f(s(x)), & cD^{-1/2} &\colon x \mapsto
  c(x)D^{-1/2}(x).
\end{align*}
Let $\phi \in {\cal W}$. Straightforward calculations show
that there exist maps
\begin{align*}
  j_{\phi} &\colon \Gamma^{2}(\mathcal{
  F},\lambda;\phi) \to \mathcal{ L}(\frakK,K_{\phi}) &&
\text{and} &
\hat j_{\phi} &\colon \Gamma^{2}(\mathcal{
  F},\lambda^{-1};\phi) \to \mathcal{ L}(\frakK,K_{\phi})
\end{align*}
such that
$j_{\phi}(c)f = fc$ and $\hat j_{\phi}(c)f =
(cD^{-1/2})f$ for all $c \in \Gamma_{c}(\mathcal{F})$, $f
\in C_{c}(G^{0})$, and 
\begin{align*}
  j_{\phi}(c)^{*}j_{\phi}(d)&=\langle
c|d\rangle_{\Gamma^{2}(\mathcal{F},\lambda;\phi)}, & \hat
j_{\phi}(c)^{*} \hat j_{\phi}(d)&=\langle
c|d\rangle_{\Gamma^{2}(\mathcal{F},\lambda^{-1};\phi)}
\text{for
all } c,d \in \Gamma_{c}(\mathcal{F}).
\end{align*}
Denote by $\gamma \subseteq {\cal L}(\frakK,K)$ and $\hdelta
\subseteq {\cal L}(\frakK,K)$ the closed linear span of all
subspaces $j_{\phi}(\Gamma^{2}(\mathcal{ F},\lambda;\phi))$
and $\hat j_{\phi}(\Gamma^{2}(\mathcal{
  F},\lambda^{-1};\phi))$, respectively, where $\phi \in
{\cal W}$.  Lemmas \ref{lemma:cx-weights} and
\ref{lemma:fell-full}  imply:
  \begin{lemma} \label{lemma:groupoid-bundle-module} $\cKhd$
    is a $C^{*}$-$(\frakb,\frakbo)$-module and
    $\rho_{\gamma}(f)(c_{\phi})_{\phi} = (fc_{\phi})_{\phi}$
    and $\rho_{\hdelta}(f)(c_{\phi})_{\phi} =
    (c_{\phi}f)_{\phi}$ for all $f\in C_{0}(G^{0})$,
    $(c_{\phi})_{\phi} \in \bigoplus_{\phi} \Gamma_{c}({\cal
      F}) \subseteq K$. \qed
\end{lemma}
For $t=s,r$, denote by $p_{1}^{t,r} \colon \GtrG \to G$ the
projection onto the first component, by $\mathcal{
  F}^{2}_{t,r}=(p_{1}^{t,r})^{*}\mathcal{ F}$ the
corresponding pull-back of $\mathcal{ F}$, and by
$\Gamma^{2}(\mathcal{ F}^{2}_{t,r},\nu^{2}_{t,r};\phi)$
the Hilbert space that is the completion of
$\Gamma_{c}(\mathcal{ F}^{2}_{t,r})$ with respect to the
inner product
 \begin{align*}
   \langle c|d\rangle &= \int_{\GtrG}
   \phi_{s(x)}(c(x,y)^{*}d(x,y)) \intd  \nu^{2}_{t,r}(x,y).
 \end{align*}
Straightforward calculations show that there
   exist unitaries
   \begin{align*}
     \Phi\colon \rHsource &\to \bigoplus_{\phi \in {\cal W}}
     \Gamma^{2}(\mathcal{ F}^{2}_{s,r},
     \nu^{2}_{s,r};\phi), & \Psi \colon \rHrange &\to
     \bigoplus_{\phi \in {\cal W}} \Gamma^{2} (\mathcal{
       F}^{2}_{r,r},\nu^{2}_{r,r};\phi),
   \end{align*}
   such that for all $\phi \in {\cal W}$, $c \in
   \Gamma_{c}(\mathcal{ F})$, $f \in C_{c}(G^{0})$, $g
   \in C_{c}(G)$,
\begin{itemize}
\item $\Phi\big(\hat j_{\phi}(c) \tr f \tl
  j(g)\big)$ is in $\Gamma^{2}({\cal
    F}^{2}_{s,r}; \nu^{2}_{s,r};\phi)$ and  given by
  $(x,y)\mapsto ((cD^{-1/2})f)(x) g(y)$,
\item $\Psi\big(j_{\phi}(c) \tr f \tl j(g)\big)$ is
  in $\Gamma^{2}({\cal F}^{2}_{r,r};
  \nu^{2}_{r,r};\phi)$ and given by $(x,y) \mapsto (fc)(x)
 g(y)$.
  \end{itemize}
  We use the isomorphisms above without further notice. If
  $(T_{\phi})_{\phi}$ is a norm-bounded family of operators
  between Hilbert spaces $(H^{1}_{\phi})_{\phi}$ and
  $(H^{2}_{\phi})_{\phi}$, we denote by $\bigoplus_{\phi}
  T_{\phi} \in {\cal L}(\bigoplus_{\phi} H_{\phi}^{1},
  \bigoplus_{\phi} H_{\phi}^{2})$ the operator given by
  $(\xi_{\phi})_{\phi} \mapsto (T_{\phi}\xi_{\phi})_{\phi}$.
  Similar arguments as those used for the construction of
  $V$ in \cite[Theorem 2.7]{timmermann:cpmu-hopf} show:
  \begin{proposition} \label{proposition:groupoid-bundle-rep}
    For each $\phi \in {\cal W}$, there exists a unitary
    $X_{\phi} \colon \Gamma^{2}({\cal
      F}^{2}_{s,r},\nu^{2}_{s,r};\phi) \to \Gamma^{2}({\cal
      F}^{2}_{r,r},\nu^{2}_{r,r};\phi)$ such that
    $(X_{\phi}f)(x,y) =f(x,x^{-1}y)$ for all $f \in
    \Gamma_{c}({\cal F}^{2}_{s,r})$, $(x,y) \in \GrrG$. The
    pair $(\cKhd,\bigoplus_{\phi} X_{\phi})$ is a
    representation of $V$. \qed
  \end{proposition}

\paragraph{The  coaction of $C^{*}_{r}(G)$ on
  $C^{*}_{r}({\cal F})$}  
We apply Lemma \ref{lemma:rep-coaction} to the
representation $(\cKhd,X)$ and obtain a coaction of
$C^{*}_{r}(G)$ on $C^{*}_{r}({\cal F})$ as follows. 
\begin{lemma} \label{lemma:groupoid-algebra-phi} Let $\phi
  \in {\cal W}$.  There exists a representation $\pi_{\phi}
  \colon C^{*}_{r}(\mathcal{ F}) \to \mathcal{ L}(K_{\phi})$
  such that for all $c,d \in \Gamma_{c}({\cal F})$, $x \in
  G$,
  \begin{align*}
    (\pi_{\phi}(c)d)(x) &=\int_{G^{r(x)}} c(z)d(z^{-1}x)
    D^{-1/2}(z) \intd\lambda^{r(x)}(z)
  \end{align*}
  and $\pi_{\phi}(c)\hat j_{\phi}(d)=\hat j_{\phi}(cd)$ and
  $\pi_{\phi}(c) \rho_{\gamma}(f) = \pi_{\phi}(c f)$ for all
  $c,d \in \Gamma_{c}(\mathcal{ F})$, $f \in C_{0}(G^{0})$.
\end{lemma}
\begin{proof}
  Identify $ \Gamma^{2}(\mathcal{ F},\lambda^{-1})
  \tr_{\phi} L^{2}(G^{0},\mu) $ with $K_{\phi}$ via $c \tr f
  \mapsto \hat j_{\phi}(c)f$ for all $c \in \Gamma_{c}({\cal
    F})$, $f \in C_{c}(G^{0})$, and define $\pi_{\phi}$ by
  $c \mapsto c \tr_{\phi} \Id$.  
\end{proof}
Define $\pi \colon C^{*}_{r}({\cal F}) \to {\cal L}(K)$ by
$c \mapsto \bigoplus_{\phi} \pi_{\phi}(c)$.  Lemmas
\ref{lemma:cx-weights} and \ref{lemma:groupoid-algebra-phi}
imply:
\begin{lemma} \label{lemma:groupoid-algebra} The
  representation $\pi$ is faithful, $\pi(C^{*}_{r}(\mathcal{
    F}))^{\gamma}_{K}$ is a $C^{*}$-$\frakb$-algebra, and
  $[\pi(C^{*}_{r}(\mathcal{ F}))\hdelta]=\hdelta$. \qed
\end{lemma}
Define $\delta \colon \pi(C^{*}_{r}({\cal F})) \to {\cal
  L}(\rHrange)$ by $\pi(c) \mapsto X(\pi(c) \botensor
\Id)X^{*}$. Then $\delta(\pi(c))=\bigoplus_{\phi}
\delta(\pi(c))_{\phi}$ for each $c \in C^{*}_{r}({\cal F})$,
where $\delta(\pi(c))_{\phi} \in {\cal L}(\Gamma^{2}({\cal
  F}^{2}_{r,r},\nu^{2}_{r,r};\phi))$ acts as follows.
\begin{lemma} \label{lemma:fell-delta-formula}
  For all $c
  \in \Gamma_{c}(\mathcal{ F})$,  $\phi \in
  {\cal W}$, $d \in \Gamma_{c}(\mathcal{ F}^2_{r,r})$,
  $(x,y) \in \GrrG$,
  \begin{align*} 
    \big(\delta(\pi(c))_{\phi}d\big)(x,y) &= \int_{G^{r(x)}}
    c(z)d(z^{-1}x,z^{-1}y) D^{-1/2}(z) \intd
    \lambda^{r(x)}(z).
  \end{align*}
\end{lemma}
\begin{proof}
  The verification is straightforward and similar to the calculation of the
  comultiplication $\Delta$ on $C^{*}_{r}(G)$; see
  \cite[Theorem 3.22]{timmermann:cpmu-hopf}.
\end{proof}
\begin{theorem} \label{theorem:groupoid-bundle-coaction}
  $(\pi(C^{*}_{r}(\mathcal{ F}))^{\gamma}_{K},\delta)$ is a
  very fine and left-full coaction of $C^{*}_{r}(G)$.
\end{theorem}
The proof involves the following two lemmas.
\begin{lemma} \label{lemma:groupoid-bundle-coaction} Let
  $\phi\in {\cal W}$.  There exist maps $T_{\phi}\colon
  \Gamma_{c}(\mathcal{ F}^{2}_{r,r}) \to {\cal
    L}(K_{\phi},\Gamma^{2}({\cal
    F}^{2}_{r,r},\nu^{2}_{r,r};\phi))$ and $ S_{\phi}\colon
  \Gamma_{c}(\mathcal{ F}^{2}_{r,r}) \to {\cal
    L}(H,\Gamma^{2}({\cal F}^{2}_{r,r},\nu^{2}_{r,r};\phi))$
  that are continuous with respect to the inductive topology
  on $\Gamma_{c}(\mathcal{ F}^{2}_{r,r})$ and the operator
  norm, respectively, such that for all $c \in
  \Gamma_{c}({\cal F}^{2}_{r,r})$, $d \in
  \Gamma_{c}(\mathcal{ F})$, $f \in C_{c}(G)$, $(x,y) \in
  \GrrG$,
    \begin{align*}
      (T_{\phi}(c)d)(x,y) &= \int_{G^{r(x)}}
      c(z,y) d(z^{-1}x) D^{-1/2}(z) \intd \lambda^{r(x)}(z), \\
      (S_{\phi}(c)f)(x,y) &= \int_{G^{r(y)}} c(x,z)f(z^{-1}y)
        D^{-1/2}(z) \intd\lambda^{r(y)}(z).
    \end{align*}
  \end{lemma}
\begin{proof}
  Let $c,d, T_{\phi}(c)d$ as above. Then
  \begin{align*}
\|T_{\phi}(c)d\|^{2} &=
    \int_{G} \int_{G^{r(x)}}
    \int_{G^{r(x)}} \int_{G^{r(x)}} \phi_{s(x)}
    \left(d(z_{1}^{-1}x)^{*} c(z_{1},y)^{*} c(z_{2},y)
      d(z_{2}^{-1}x)\right) \cdot \\ &\hspace{2cm} \cdot
    D^{-1/2}(z_{1})D^{-1/2}(z_{2}) \intd\lambda^{r(x)}(y)
    \intd\lambda^{r(x)}(z_{1}) \intd\lambda^{r(x)}(z_{2})
    \intd\nu(x).
  \end{align*}
  We substitute $x'=z_{1}^{-1}x$, $z=z_{1}^{-1}z_{2}$, use
  the relations $D(z_{2})=D(z_{1})D(z)$ and
\begin{align*}
  D^{-1}(z_{1})\intd\lambda^{r(x)}(z_{1})
  \intd\nu(x) &= D^{-1}(z_{1})
  \intd\lambda^{r(z_{1})}(x)\intd\nu(z_{1}) \\ &=
  \intd\lambda^{s(z_{1})}(x') \intd\nu^{-1}(z_{1}) =
  \intd\lambda^{-1}_{r(x')}(z_{1})\intd\nu(x'),
\end{align*}
and find
\begin{align*}
  \|T_{\phi}(c)d\|^{2} &= \int_{G}   \int_{G_{r(x')}} \int_{G^{r(x')}}
 \int_{G^{s(z_{1})}} 
  \phi_{s(x)}(d(x')^{*}c(z_{1},y)^{*}c(z_{1}z,y)d(z^{-1}x'))
  \cdot \\ &\hspace{4cm} \cdot D^{-1/2}(z)
  \intd\lambda^{s(z_{1})}(y) \intd\lambda^{r(x')}(z)
  \intd\lambda^{-1}_{r(x')}(z_{1}) \intd\nu(x') \\
  &= \int_{G}\int_{G^{r(x')}}
\phi_{s(x')} (d(x')R_{c}(z)d(z^{-1}x'))
\intd\lambda^{r(x')}(z)\intd\nu(x') =
\langle d|\pi_{\phi}(R_{c})d\rangle_{K_{\phi}},
\end{align*}
where $R_{c} \in \Gamma_{c}({\cal F})$ is given by
\begin{align*}
  R_{c}(z) &= \int_{G_{r(z)}}\int_{G^{s(z_{1})}}
  c(z_{1},y)^{*}c(z_{1}z,y) \intd\lambda^{s(z_{1})}(y)
  \intd\lambda^{-1}_{r(z)}(z_{1}) \quad \text{for all } z
  \in G.
\end{align*}
  Hence, $T_{\phi}(c)$ extends to a bounded linear
  operator of norm $\|T_{\phi}(c)\|^{2}\leq
  \|\pi_{\phi}(R_{c})\|$.  If $(c_{n})_{n}$ is a
  sequence in $\Gamma_{c}(\mathcal{ F}^{2}_{r,r})$
  converging to $c$ in the inductive limit topology,
  then the functions $R_{(c-c_{n})}$ defined
  similarly as $R_{c}$  converge to $0$ in the inductive limit
  topology and hence $\|T_{\phi}(c-c_{n})\|^{2} \leq
  \|\pi_{\phi}(R_{(c-c_{n})})\|$ converges to $0$.

  The proof of the assertion concerning $S_{\phi}$ is very
  similar.
\end{proof}
Given $c,d \in \Gamma_{c}({\cal F})$ and $f \in
C_{c}(G)$, define $\omega_{c,d,f} \in
\Gamma_{c}({\cal F}^{2}_{r,r})$ by
  \begin{align*}
    (x,y) \mapsto \int_{G^{r(x)}} c(z) d(z^{-1}x) f(z^{-1}y)
    \intd\lambda^{r(x)}(z).
  \end{align*}
  \begin{lemma} \label{lemma:fell-omega-dense} The linear
    span of all elements $\omega_{c,d,f}$ as above is dense
    in $\Gamma_{c}({\cal F}^{2}_{r,r})$ with respect to the
    inductive limit topology.
\end{lemma}
\begin{proof}
  Let $(x,y) \in \GrrG$, $e \in {\cal F}_{x}$, let $C
  \subseteq \GrrG$ be a compact neighbourhood of $(x,y)$,
  and let $\epsilon > 0$. Since $[{\cal F}_{r(x)}{\cal
    F}_{x}] = {\cal F}_{x}$, we can choose $c',d' \in
  \Gamma_{c}({\cal F})$ such that $\|c'(z)d'(z^{-1}x)
  -e\|<\epsilon$ for all $z$ in some neighbourhood of $r(x)$
  in $G^{r(x)}$. Next, we can choose $h_{c},h_{d},f \in
  C_{c}(G)$ such that the elements $c,d \in \Gamma_{c}({\cal
    F})$ given by $c(z)=c'(z)h_{c}(z)$ and
  $d(z)=d'(z)h_{d}(z)$ for all $z \in G$ satisfy
  $\|\omega_{c,d,f}(x,y) - e\| < \epsilon$ and $\supp
  \omega_{c,d,f} \subseteq C$.  A standard partition of
  unity argument concludes the proof.
\end{proof}
\begin{proof}[Proof of Theorem
  \ref{theorem:groupoid-bundle-coaction}] 
  We show that Lemma \ref{lemma:rep-coaction} applies.  Let
  $\phi \in {\cal W}$, $c,d \in \Gamma_{c}(\mathcal{ F})$,
  $f ,g\in C_{c}(G)$. Define 
  $e_{1},e_{2},e_{3},e_{4} \in \Gamma^{2}({\cal
    F}^{2}_{r,r},\nu^{2}_{r,r};\phi)$ and
  $\omega_{1},\omega_{2},\omega_{3},\omega_{4} \in
  \Gamma_{c}({\cal F}^{2}_{r,r})$ by
  \begin{align*}
    e_{1} &= \delta(\pi(c))_{\phi}|j(f)\rangle_{2} d, &
    \omega_{1}(z,y) &= c(z)f(z^{-1}y) \text{ for all }
    (z,y) \in \GrrG, \\
    e_{2}&= |j(f)\rangle_{2}\pi_{\phi}(c)d, &
    \omega_{2}(z,y) &= c(z)f(y) \text{ for all }
    (z,y) \in \GrrG, \\
    e_{3}&=|j_{\phi}(c)\rangle_{1}L(f)g, & \omega_{3}(x,z)
    &= c(x)f(z) \text{ for all } (x,z) \in \GrrG,
    \\
    e_{4}&=\delta(\pi(c))_{\phi}|j_{\phi}(d)\rangle_{1}L(f)g,
    & \omega_{4}&=\omega_{c,d,f}.
  \end{align*}
  Using Lemma \ref{lemma:fell-delta-formula}, we find that
  for all $(x,y) \in \GrrG$,
  \begin{align*} 
    e_{1} (x,y) &= \int_{G^{r(x)}} c(z)D^{-1/2}(z)
    d(z^{-1}x) f(z^{-1}y) \intd\lambda^{r(x)}(z) =
    (T_{\phi}(\omega_{1})d)(x,y), \\
    e_{2} (x,y) &= \int_{G^{r(x)}} c(z) d(z^{-1}x)
    D^{-1/2}(z) \intd\lambda^{r(x)}(z) f(y) =
    (T_{\phi}(\omega_{2})d)(x,y), \\
    e_{3} (x,y) &= c(x) \int_{G^{r(y)}}
    f(z)D^{-1/2}(z)g(z^{-1}y) \intd\lambda^{r(y)}(z) =
    (S_{\phi}(\omega_{3})g)(x,y), \\
    e_{4} (x,y)& = \int_{G^{r(x)}} c(z_{1}) D^{-1/2}(z_{1})
    d(z_{1}^{-1}x) (L(f)g)(z_{1}^{-1}y)
    \intd\lambda^{r(x)}(z_{1}) \\
    &=\int_{G^{r(x)}} \int_{G^{s(z_{1})}}
    c(z_{1})D^{-1/2}(z_{1})
    d(z_{1}^{-1}x)f(z_{2})  \cdot \\
    & \hspace{4cm} \cdot D^{-1/2}(z_{2})
    g(z_{2}^{-1}z_{1}^{-1}y) \intd\lambda^{s(z_{1})}(z_{2})
    \intd\lambda^{r(x)}(z_{1}) \\
    &=\int_{G^{r(x)}} \int_{G^{r(x)}} c(z_{1})
    d(z_{1}^{-1}x)f(z_{1}^{-1}z'_{2})D^{-1/2}(z_{2}')
    g(z'_{2}{}^{-1}y) \intd\lambda^{r(x)}(z_{2}')
    \intd\lambda^{r(x)}(z_{1}) \\
    &= (S_{\phi}(\omega_{c,d,f})g)(x,y).
  \end{align*}
  By Lemmas \ref{lemma:submodule} and
  \ref{lemma:fell-omega-dense}, sections of the form like
  $\omega_{1},\omega_{2},\omega_{3}$ or $\omega_{4}$,
  respectively, are linearly dense in $\Gamma_{c}(\mathcal{
    F}^{2}_{r,r})$. Therefore, $[\delta(\pi(C^{*}_{r}({\cal
    F})))_{\phi}\kalpha{2}]= [T_{\phi}(\Gamma_{c}(\mathcal{
    F}^{2}_{r,r}))] =
  [\kalpha{2}\pi_{\phi}(C^{*}_{r}(\mathcal{ F}))]$ and similarly
  $[\delta(\pi(C^{*}_{r}(\mathcal{ F}))) |\gamma\rangle_{1}
  C^{*}_{r}(G)] = [\bigcup_{\phi \in {\cal W}}
  S_{\phi}(\Gamma_{c}(\mathcal{ F}^{2}_{r,r}))] =
  [\kgamma{1}C^{*}_{r}(G)]$.
\end{proof}
Given $g,g' \in C_{c}(G)$, define $h_{g,g'} \in C_{c}(G)$ by
\begin{align} \label{eq:hxixi}
  h_{g,g'}(z) &=\int_{G^{r(z)}}
  \overline{g(y)}g'(z^{-1}y) \intd\lambda^{r(z)}(y)
  \quad \text{for all } z \in G.
\end{align}
\begin{lemma} \label{lemma:fell-eq-1}
Let $c \in \Gamma_{c}(\mathcal{
    F})$, $g,g' \in C_{c}(G)$. Then $\langle j(g)|_{2}
  \delta(\pi(c))_{\phi}|j(g')\rangle_{2} = \pi_{\phi}(c')$,
  where $c'(x)=c(x)h_{g,g'}(x)$ for all $x \in G$.
\end{lemma}
\begin{proof}
The operators on
  both sides map each $d \in \Gamma_{c}({\cal F})$ to
  the section
  \begin{gather*}
    x \mapsto \int_{G^{r(x)}}\int_{G^{r(x)}}
    \overline{g(y)}c(z)d(z^{-1}x)g'(z^{-1}y)D^{-1/2}(z)
    \intd\lambda^{r(x)}(z) \intd\lambda^{r(x)}(y). \qedhere
 \end{gather*}
\end{proof}

\paragraph{The reduced crossed product of the coaction}
The bundle ${\cal F}^{2}_{s,r}$ carries the structure of a
Fell bundle, and the reduced crossed product
$\pi(C^{*}_{r}(\mathcal{ F})) \rtimes_{r} C_{0}(G)$ for the
coaction $\delta$ constructed above can be identified with
$C^{*}_{r}(\mathcal{F}^{2}_{s,r})$ as follows.  

Denote by $G \ltimes G$ the transformation groupoid for the
action of $G$ on itself given by right multiplication. Thus,
$G \ltimes G= \GsrG$ as a set, $(G\ltimes G)^{0} =
\bigcup_{u \in G^{0}} \{u\} \times G^{u}$ can be identified
with $G$ via $(r(y),y) \equiv y$, the range map $\tilde r$,
the source map $\tilde s$, and the multiplication are given
by $ (x,y) \stackrel{\tilde r}{\mapsto} xy$, $(x,y)
\stackrel{\tilde s}{\mapsto} y$, and $((x,y),(x',y'))\mapsto
(xx',y')$, respectively, and the topology on $G \ltimes G$
is the weakest topology that makes $\tilde r, \tilde s$ and the map
$(x,y) \mapsto x$ continuous. We equip $G \ltimes G$ with
the right Haar system $\tilde \lambda^{-1}$ given by $\tilde
\lambda^{-1}_{y}(C \times \{y\}) =\lambda^{-1}_{r(y)}(C)$
for all $C \subseteq G_{r(y)}$, $y \in G$.

The bundle $\mathcal{F}^{2}_{s,r}$ is a Fell bundle on $G
\ltimes G$ with respect to the multiplication and involution
given by $((f,y),(f',y')) \mapsto (ff',y')$ and $(f,y)
\mapsto (f^{*},p(f)y)$.  The convolution product in
$\Gamma_{c}({\cal F}^{2}_{s,r})$ is given by
\begin{align} \label{eq:convolution-qf} (cd)(x,y) &=
  \int_{G_{r(y)}} c(xz^{-1},zy)d(z,y)
  \intd\lambda^{-1}_{r(y)}(z)
  \end{align}
  for all $c,d \in \Gamma_{c}(\mathcal{F}^{2}_{s,r})$,
  $(x,y) \in \GsrG$, because $(G \ltimes G)_{\tilde s(x,y)} =
  G_{r(y)} \times \{y\}$ and $(x,y) (z,y)^{-1} =
  (xz^{-1},zy)$ for all $z \in G_{r(y)}$.
  \begin{proposition} \label{proposition:groupoid-bundle-crossed}
    There exists an isomorphism $ \pi(C^{*}_{r}(\mathcal{
      F})) \rtimes_{r} C_{0}(G) \to
    C^{*}_{r}(\mathcal{F}^{2}_{s,r})$ such that
    $\delta(\pi(c))(1 \btensor f) \mapsto
    L_{\mathcal{F}^{2}_{s,r}}(d)$  whenever $c \in
    \Gamma_{c}(\mathcal{ F})$, $f \in C_{c}(G)$, and
    $d(x,y)=c(x)f(y)$ for all $(x,y) \in \GsrG$.
\end{proposition}
Let $\phi \in {\cal W}$. Then the map $r^{*}\phi \colon
\Gamma_{0}(({\cal F}^{2}_{s,r})^{0}) \to C_{0}(G)$ given by
$(r^{*}\phi(c))(y)=\phi_{r(y)}(c(r(y),y))$ for all $c \in
\Gamma_{0}((\mathcal{F}^{2}_{s,r})^{0})$ and $y \in G$ is a
$C_{0}(G)$-weight. One easily verifies that there exists a
representation $L_{r^{*}\phi}\colon C^{*}_{r}({\cal
  F}^{2}_{s,r}) \to {\cal L}(\Gamma^{2}({\cal
  F}^{2}_{s,r},\tilde \lambda^{-1};r^{*}\phi))$ such that
$L_{r^{*}\phi}(c)d=cd$ for all $c,d \in \Gamma_{c}({\cal
  F}^{2}_{s,r})$.
\begin{lemma}
\begin{enumerate}
\item There exists a unique unitary $U_{\phi}\colon
  \Gamma^{2}({\cal F}^{2}_{s,r},\tilde
  \lambda^{-1};r^{*}\phi) \tr H \to \Gamma^{2}({\cal
    F}^{2}_{s,r},\nu^{2}_{s,r};\phi) \subseteq \rHsource$
  such that $(U_{\phi}(e \tr g))(x,y) = e(x,y) g(y)D^{-1/2}(x)$ for all
  $e \in \Gamma_{c}(\mathcal{F}^{2}_{s,r})$, $g \in
  C_{c}(G)$, $(x,y) \in \GsrG$.
\item $\delta(\pi(c))(1 \btensor f)X_{\phi}U_{\phi} =
  X_{\phi} U_{\phi}(L_{r^{*}\phi}(d) \tr \Id)$ for all
  $c,d,f$ as in Proposition
  \ref{proposition:groupoid-bundle-crossed}.
  \end{enumerate}
\end{lemma}
\begin{proof}
 i)  For all $e,g$ as in above,
 \begin{align*}
    \| U_{\phi}(e \tr g)\|^{2} &= \int_{G} \int_{G_{r(y)}}
    \phi_{s(x)}(e(x,y)^{*}e(x,y) )|g(y)|^{2}
    \intd\lambda^{-1}_{r(y)}(x)
    \intd\nu(y) = \|
    e\tr g\|^{2}.
  \end{align*}

  ii) Let $c,d,e,f,g,(x,y)$ as above and $\hDelta(f)_{\phi}
  = X_{\phi}^{*}(1 \btensor f)X_{\phi}$. A short calculation
  shows that $(\hDelta(f)_{\phi}U_{\phi} (e \tr g)\big)(x,y)
  = f(xy) e(x,y)g(y)D^{-1/2}(x)$. Using
  \eqref{eq:convolution-qf}, we find that $    \big( (\pi_{\phi}(c) \botensor
    \Id)\hDelta(f)_{\phi}U_{\phi} (e \tr g)\big)(x,y)$ is
    equal to
  \begin{align*}
    \int_{G^{r(x)}}& c(z) f(z^{-1}xy) e(z^{-1}x,y)g(y) 
 D^{-1/2}(z)D^{-1/2}(z^{-1}xy)
    \intd\lambda^{r(x)}(z) \\
    &= \int_{G_{s(x)}}c(xz^{-1})f(zy)e(z,y)g(y)
    D^{-1/2}(xy) \intd\lambda^{-1}_{s(x)}(z)  \\
    &=\int_{G_{s(x)}}
    d(xz^{-1},zy) e(z,y) g(y)  D^{-1/2}(xy)\intd\lambda^{-1}_{s(x)}(z) 
= (U_{\phi}(de \tr g))(x,y).
  \end{align*}
Thus, $\delta(\pi(c))(1 \btensor f)X_{\phi}U_{\phi}= X_{\phi}
  (\pi_{\phi}(c) \botensor \Id)\hDelta(f)_{\phi}X_{\phi}U_{\phi} =
  X_{\phi}U_{\phi}(L_{r^{*}\phi}(d) \tr \Id)$.
\end{proof}
\begin{proof}[Proof of Proposition
  \ref{proposition:groupoid-bundle-crossed}]
  Consider the $*$-homomorphism
  \begin{align*}
    \Phi \colon C^{*}_{r}(\mathcal{F}^{2}_{s,r}) \to {\cal
      L}(\rHrange), \quad L_{\mathcal{F}^{2}_{s,r}}(d)
    \mapsto \bigoplus_{\phi \in {\cal W}}
    X_{\phi}U_{\phi}(L_{r^{*}\phi}(d)\tr
    \Id)U_{\phi}^{*}X_{\phi}^{*}.
  \end{align*}
  By part ii) of the lemma above,
  $\Phi(C^{*}_{r}(\mathcal{F}^{2}_{s,r}))$ contains $
  [\delta(\pi(C^{*}_{r}(\mathcal{ F})))(1 \btensor
  C_{0}(G))] = \pi(C^{*}_{r}(\mathcal{ F})) \rtimes_{r}
  C_{0}(G)$.  The same lemma implies that this inclusion is
  an equality because the map $a \to \bigoplus_{\phi} a \tr
  \Id$ is continuous with respect to the inductive limit
  topology on $\Gamma_{c}(\mathcal{ F}^{2}_{s,r})$ and
  sections of the form $(x,y) \mapsto c(x)f(y)$, where $c\in
  \Gamma_{c}(\mathcal{ F})$, $f \in C_{c}(G)$, are dense in
  $\Gamma_{c}(\mathcal{ F}^{2}_{s,r})$ by Lemma
  \ref{lemma:submodule}.  Lemma \ref{lemma:cx-weights}
  implies that $[\bigcap_{\phi} \ker r^{*}\phi] = 0$, and
  therefore $\Phi$ is injective.
\end{proof}
\begin{proposition} \label{proposition:bundle-crossed-compact}
  If $\mathcal{ F}$ is saturated, then
  $C^{*}_{r}(\mathcal{F}^{2}_{s,r}) \cong \mathcal{ K}
  (\Gamma^{2}(\mathcal{ F},\lambda^{-1}))$.
\end{proposition}
\begin{proof}
  To simplify notation, let $\Gamma^{2} =
  \Gamma^{2}(\mathcal{ F},\lambda^{-1})$, $ \tilde
  \Gamma^{2}=\Gamma^{2}(\mathcal{F}^{2}_{s,r},\tilde
  \lambda^{-1})$, $\Gamma_{0} = \Gamma_{0}(\mathcal{
    F}^{0})$, $\tilde \Gamma_{0} =
  \Gamma_{0}((\mathcal{F}^{2}_{s,r})^{0})$.  There exists a
  unitary $\Psi \colon \Gamma^{2} \tr_{s^{*}} C_{0}(G) \to
  \tilde \Gamma^{2}$ such that $(\Psi(c \tr f))(x,y)=c(xy)f(y)$ for
  all $c\in \Gamma_{c}(\mathcal{ F})$, $f \in C_{c}(G)$,
  $(x,y) \in \GsrG$, because
\begin{align*}
  \langle \Psi(c \tr f)|\Psi(c' \tr f')\rangle((r(y),y)) &=
  \int_{G_{r(y)}}c(xy)^{*}c'(xy)\intd\lambda^{-1}_{r(y)}(x)
  \overline{f(y)}f'(y) \\ &= \overline{f(y)}\langle
  c|c'\rangle_{\Gamma^{2}}(s(y)) f(y) = \langle c \tr f|c'
  \tr f'\rangle(y)
\end{align*}
for all $c,c' \in \Gamma_{c}(\mathcal{ F}),f,f' \in
C_{c}(G)$, $y \in G$ by right-invariance of
$\lambda^{-1}$. The $*$-homo\-morphism $\Phi \colon
\mathcal{ K}(\Gamma^{2}) \to \mathcal{ L}(\tilde
\Gamma^{2})$ given by $k \mapsto \Psi(k \tr_{s^{*}}
\Id)\Psi^{*}$ is injective because $s^{*} \colon
C_{0}(G^{0}) \to \mathcal{ L}(C_{0}(G))$ is injective, and
the claim follows once we have shown that $\Phi(\mathcal{
  K}(\Gamma^{2}))=C^{*}_{r}({\cal F}^{2}_{s,r})$.  Let $d,d'
\in \Gamma_{c}(\mathcal{ F})$ and denote by
$|d\rangle\langle d'| \in \mathcal{ K}(\Gamma^{2})$ the
operator given by $e \mapsto d\langle d'|e\rangle$. Then for
all $c,f,(x,y)$ as above,
\begin{align*}
  \big(\Psi(|d\rangle\langle d'|c \tr f)\big)(x,y) &=
  \int_{G_{s(y)}} d(xy)d'(z)^{*}c(z)f(y)
  \intd\lambda^{-1}_{s(y)}(z) \\
  &= \int_{G_{s(y)}} d(xy)d'(z)^{*} \big(\Psi(c \tr
  f)\big)(zy^{-1},y)
  \intd\lambda^{-1}_{s(y)}(z) \\
  &= \int_{G_{r(y)}} d(xy)d'(z'y)^{*} \big(\Psi(c \tr
  f)\big)(z',y)   \intd\lambda^{-1}_{r(y)}(z').
\end{align*}
Comparing with equation \eqref{eq:convolution-qf}, we find
that $\Psi(|d\rangle\langle d'| \tr
\Id)\Psi^{*}=L_{\mathcal{F}^{2}_{s,r}}(e)$, where $e \in
\Gamma_{c}(\mathcal{F}^{2}_{s,r})$ is given by
$e(xz^{-1},zy) = d(xy)d'(zy)^{*}$, or equivalently, by
$e(x',y')=d(x'y')d'(y')^{*}$ for all $(x',y') \in
\GsrG$. Since $\mathcal{ F}$ is saturated, Lemma
\ref{lemma:submodule} implies that sections of this
form are dense in $\Gamma_{c}(\mathcal{F}^{2}_{s,r})$ with
respect to the inductive limit topology, and since the map
$e \mapsto L_{\mathcal{F}^{2}_{s,r}}(e) $ is continuous with
respect to this topology, we can conclude that
$\Phi(\mathcal{ K}(\Gamma^{2})) = \Psi(\mathcal{
  K}(\Gamma^{2}) \tr \Id)\Psi^{*} = C^{*}_{r}({\cal
  F}^{2}_{s,r})$.
\end{proof}

\begin{corollary}
  If $\mathcal{ F}$ is saturated, then
  $\pi(C^{*}_{r}(\mathcal{ F})) \rtimes_{r} C_{0}(G)$ and
  $\Gamma_{0}(\mathcal{ F}^{0})$ are Morita equivalent.
\end{corollary}
\begin{proof}
  One easily verifies that $\Gamma^{2}({\cal
    F},\lambda^{-1})$ is full.
\end{proof}
\begin{example}
  Let $\sigma$ be an action of $G$ on an admissible
  $C_{0}(G^{0})$-algebra $C$ and let $\delta_{\sigma}$ be
  the corresponding coaction of $C_{0}(G)$ on $\bfF C$
  (Proposition \ref{proposition:groupoid-actions}). Then
  there exists an admissible Fell bundle $\mathcal{ C}$ on
  $G$ with fibre $\mathcal{ C}_{x} = C_{r(x)}$ for each $x
  \in G$, continuous sections $\Gamma_{0}(\mathcal{ C}) =
  r^{*}C$, and multiplication and involution given by $cd =
  c\sigma_{x}(d)$, $c^{*} = \sigma_{x^{-1}}(c^{*})$ for all
  $c \in \mathcal{ C}_{x}$, $d \in \mathcal{ C}_{y}$, $(x,y)
  \in \GsrG$ \cite{kumjian}, and the identity on
  $\Gamma_{c}(\mathcal{ C}) = C_{c}(G)r^{*}C$ extends to an
  isomorphism $C^{*}_{r}(\mathcal{ C}) \to C \rtimes_{r}
  G$. One easily verifies that with respect to the
  isomorphism $\pi(C^{*}_{r}(\mathcal{ C})) \cong
  C^{*}_{r}(\mathcal{ C}) \cong C \rtimes_{r} G \cong \bfF C
  \rtimes_{r} C^{*}_{r}(G)$ of Proposition
  \ref{proposition:groupoid-crossed}, the coaction of
  Theorem \ref{theorem:groupoid-bundle-coaction} coincides
  with the dual coaction on $ \bfF C \rtimes_{r}
  C^{*}_{r}(G)$.  Moreover, the Fell bundle $\mathcal{ C}$
  is saturated and $C^{*}_{r}(\mathcal{ C}) \rtimes_{r}
  C_{0}(G) \cong \bfF C \rtimes_{r} C^{*}_{r}(G) \rtimes_{r}
  C_{0}(G)$ is Morita equivalent to $\Gamma_{0}(\mathcal{
    C}^{0})\cong C$, as we already know by Theorem
  \ref{theorem:duality}.
\end{example}
\begin{remark}
  The Fell bundle $\mathcal{ F}$ can be equipped with the
  structure of an $\mathcal{F}^{2}_{s,r}$-$\mathcal{
    F}^{0}$-equivalence in the sense of
  \cite{williams:fell} in a straightforward way. 
\end{remark}

\paragraph{Functoriality of the construction}
Let ${\cal G}$, ${\cal F}$ be admissible Fell bundles on $G$
with associated representations $((K_{{\cal G}},
{\gamma_{{\cal G}},\hdelta_{{\cal G}}}), X_{{\cal G}})$,
$((K_{{\cal F}}, {\gamma_{{\cal F}},\hdelta_{{\cal
      F}}}),X_{{\cal F}})$ and coactions $(\pi_{{\cal
    G}}(C^{*}_{r}({\cal G}))^{\gamma_{{\cal G}}}_{K_{{\cal
      G}}},\delta_{{\cal G}})$, $(\pi_{{\cal
    F}}(C^{*}_{r}({\cal F}))^{\gamma_{{\cal F}}}_{K_{{\cal
      F}}},\delta_{{\cal F}})$, and let $T$ be a morphism
from ${\cal G}$ to ${\cal F}$.
\begin{proposition} \label{proposition:fell-coaction-morphism}
  There exists a  unique morphism $\tilde T_{*}$  from
  $\big(\pi_{\mathcal{ G}}(C^{*}_{r}(\mathcal{
    G}))^{\gamma_{\mathcal{ G}}}_{K_{\mathcal{
        G}}},\delta_{\mathcal{ G}})$ to $\big(\pi_{\mathcal{
      F}}(C^{*}_{r}(\mathcal{ F}))^{\gamma_{\mathcal{
        F}}}_{K_{\mathcal{ F}}},\delta_{\mathcal{ F}})$ that
  satisfies $\tilde T_{*}(\pi_{\mathcal{
      G}}(a))=\pi_{\mathcal{ F}}(T_{*}(a))$ for all $a \in
  \Gamma_{c}(\mathcal{ G})$.
\end{proposition}
The proof involves the following construction. 
\begin{lemma} \label{lemma:bundle-morphism-intertwiner} Let
  $\phi \in {\cal W}(\Gamma_{0}({\cal F}^{0}))$, $f \in
  \Gamma_{0}(\mathcal{ F}^{0})$ and define $\psi \in {\cal
    W}(\Gamma_{0}({\cal G}^{0}))$ by $g \mapsto
  \phi(f^{*}T_{*}^{0}(g)f)$.
 \begin{enumerate}
 \item There exists a unique isometry    $T^{f}_{\phi}
   \colon K_{\psi} \to K_{\phi}$  such that
   $T^{\phi}_{f}g=T_{*}(g)f$ for all $g \in \Gamma_{c}({\cal G})$.
 \item  $T_{\phi}^{f}j_{\psi}(g) =
   j_{\phi}(T_{*}(g)f)$, $T_{\phi}^{f} \hat j_{\psi}(g) =
   \hat j_{\psi}(T_{*}(g)f)$, $T_{\phi}^{f}
   \pi_{\psi}(g)=\pi_{\phi}(T_{*}(g))T_{\phi}^{f}$ for all $g
   \in \Gamma_{c}(\mathcal{ G})$.
\end{enumerate}
Denote also the map $K_{{\cal G}} \to K_{\psi}
\hookrightarrow K_{{\cal F}}$ given by
$(\xi_{\psi'})_{\psi'} \mapsto T_{\phi}^{f}\xi_{\psi}$ by
$T^{f}_{\phi}$.
\begin{enumerate} \setcounter{enumi}{2}
\item $T^{f}_{\phi}$ is a semi-morphism from
  $(K_{{\cal G}},\hdelta_{{\cal G}},\gamma_{{\cal G}})$ to
  $(K_{{\cal F}},\hdelta_{{\cal F}},\gamma_{{\cal F}})$ and
 $(T^{f}_{\phi} \btensor \Id)X_{{\cal G}}
  =X_{{\cal F}} (T^{f}_{\phi} \botensor \Id)$.
\item $\delta_{{\cal F}}(\pi_{{\cal F}}(h))(T_{\phi}^{f}
  \btensor \Id)\delta_{{\cal G}}(\pi_{{\cal G}}(g)) =
  \delta_{{\cal F}}(\pi_{{\cal F}}(hT_{*}(g)))(T_{\phi}^{f} \btensor
  \Id)$ for all $h \in \Gamma_{c}(\mathcal{ F})$, $g \in
  \Gamma_{c}(\mathcal{ G})$.
  \end{enumerate}
\end{lemma}
\begin{proof}
  i) Uniqueness is clear. Existence follows from the
  fact that  for all $g,g' \in \Gamma_{c}(\mathcal{ G})$,
  \begin{align*}
    \langle
    T_{*}(g)f|T_{*}(g')f\rangle_{K_{\phi}} &= \int_{G}
    \phi_{s(x)}\big(f(s(x))^{*}T(g(x)^{*}g'(x))f(s(x))\big)
    \intd\nu(x) \\
    &= \int_{G} \psi_{s(x)}(g(x)^{*}g'(x)) \intd\nu(x)
    = \langle g|g\rangle_{K_{\psi}}.
  \end{align*}

  ii)  Straightforward.

  iii) By ii), $T_{\phi}^{f}\gamma_{{\cal G}}\subseteq
  \gamma_{{\cal F}}$ and $T_{\phi}^{f}\hdelta_{{\cal G}}
  \subseteq \hdelta_{{\cal F}}$. For all $\omega \in
  \Gamma_{c}(\mathcal{ G}^{2}_{s,r})$ and $(x,y) \in \GrrG$,
  \begin{align*}
    ((T_{\phi}^{f} \btensor \Id)X_{{\cal G}}\omega)(x,y)
    = \omega(x,x^{-1}y)f(s(x)) 
    = (X_{{\cal F}}(T_{\phi}^{f} \botensor \Id)\omega)(x,y).
  \end{align*}

  iv) By ii) and iii), we have $X_{{\cal F}}(\pi_{{\cal
      F}}(h) \botensor \Id)X_{{\cal F}}^{*}(T_{\phi}^{f}
  \btensor \Id)X_{{\cal G}}(\pi_{{\cal G}}(g) \botensor
  \Id)X_{{\cal G}}^{*} =X_{{\cal F}}(\pi_{{\cal
      F}}(hT_{*}(g)) \botensor \Id)X_{{\cal
      F}}^{*}(T_{\phi}^{f}\btensor \Id)$ for all $g \in
  \Gamma_{c}(\mathcal{ G})$ and $h \in \Gamma_{c}(\mathcal{
    F})$.
\end{proof}
\begin{proof}[Proof of Proposition
  \ref{proposition:fell-coaction-morphism}] 
Denote by $\mathcal{ T} \subseteq \mathcal{ L}(K_{\mathcal{
      G}},K_{\mathcal{ F}})$  the closed linear span of
  all operators $T^{f}_{\phi}$, where $\phi \in {\cal
    W}(\Gamma_{0}({\cal F}^{0}))$ and $f \in
  \Gamma_{0}({\cal F}^{0})$.  Then Lemma
  \ref{lemma:bundle-morphism-intertwiner} and Proposition
  \ref{proposition:fell-morphism} imply that $S \pi_{\mathcal{
      G}}(g) = \pi_{\mathcal{ F}}(T_{*}(g))S$ for all $S \in
  \mathcal{ T}, g\in \Gamma_{c}(\mathcal{ G})$ and that
  \begin{align*}
    [\mathcal{ T}\gamma_{\mathcal{ G}}] &=
    \Big[\bigcup_{\phi} j_{\phi}
    \big(T_{*}(\Gamma_{c}(\mathcal{ G}))\Gamma_{0}(\mathcal{
      F}^{0})\big)\Big] = \Big[\bigcup_{\phi}
    j_{\phi}(\Gamma_{c}({\cal F}))\Big]=\gamma_{\mathcal{
        F}}.
  \end{align*}
  By Proposition \ref{proposition:fell-morphism}, $T_{*}$
  extends to a nondegenerate $*$-homomorphism
  $C^{*}_{r}({\cal G}) \to M(C^{*}_{r}(\mathcal{ F}))$.
  Thus, there exists a semi-morphism $\tilde T_{*}$ from
  $\pi_{\mathcal{ G}}(C^{*}_{r}(\mathcal{
    G}))^{\gamma_{\mathcal{ G}}}_{K_{\mathcal{ G}}}$ to
  $\pi_{\mathcal{ F}}(C^{*}_{r}(\mathcal{
    F}))^{\gamma_{\mathcal{ F}}}_{K_{\mathcal{ F}}}$ such
  that $\tilde T_{*}(\pi_{\mathcal{ G}}(g))=\pi_{\mathcal{
      F}}(T_{*}(g))$ for all $g \in \Gamma_{c}(\mathcal{
    G})$.  For all $h \in \pi_{\mathcal{
      F}}(\Gamma_{c}(\mathcal{ F}))$, $g \in \pi_{\mathcal{
      G}}(\Gamma_{c}(\mathcal{ G}))$, $S \in \mathcal{ T}$,
  \begin{align*}
    \delta_{\mathcal{ F}}(h) \cdot (\tilde T_{*} \ast
    \Id)(\delta_{\mathcal{ G}}(g)) \cdot (S \btensor
    \Id)&= \delta_{\mathcal{ F}}(h) (S \btensor
    \Id)\delta_{\mathcal{ G}}(g) =
    \delta_{\mathcal{ F}}(h \tilde T_{*}(g))(S \btensor
    \Id)
  \end{align*}
  by Lemma \ref{lemma:bundle-morphism-intertwiner}, and
  therefore $ \delta_{\mathcal{ F}}(h) \cdot (\tilde T_{*} \ast
  \Id)(\delta_{\mathcal{ G}}(g)) = \delta_{\mathcal{
      F}}(h\tilde T_{*}(g))$.
\end{proof} 
Denote by $\bffell^{a}_{G}$ the category of all admissible
Fell bundles on $G$, and by $\bfcoact_{C^{*}_{r}(G)}^{a}$
the category of very fine left-full coactions of $C^{*}_{r}(G)$.
\begin{theorem}
  The assignments $\mathcal{ F} \mapsto (\pi_{\mathcal{
      F}}(C^{*}_{r}(\mathcal{ F}))^{\gamma_{\mathcal{
        F}}}_{K_{\mathcal{ F}}},\delta_{\mathcal{ F}})$ and
  $T \mapsto \tilde T_{*}$ form a faithful functor $\bfFA \colon
  \bffell^{a}_{G} \to \bfcoact_{C^{*}_{r}(G)}^{a}$.
\end{theorem}
\begin{proof}
  Functoriality of the constructions is evident. Assume that
  $\bfFA S=\bfFA T$ for some morphisms $S,T$ from $\mathcal{
    F}$ to $\mathcal{ G}$ in $\bffell^{a}_{G}$. Then the maps
  $S_{*},T_{*} \colon \Gamma_{c}(\mathcal{ F}) \to
  \Gamma_{c}(\mathcal{ M}(\mathcal{ G}))$ coincide because
  $\pi_{\mathcal{ G}}$ is injective. Since $\{a(x)\mid a \in
  \Gamma_{c}(\mathcal{ F})\} = \mathcal{ F}_{x}$ for each $x
  \in G$ and $S(a(x))=(S_{*}a)(x)=(T_{*}a)(x)=T(a(x))$ for
  each $a \in \Gamma_{c}(\mathcal{ F})$, $x\in G$, we can
  conclude that $S=T$.
\end{proof}

\section{From coactions of $C^{*}_{r}(G)$ to Fell bundles
  for  \'etale $G$}

\label{section:etale}

We now assume that the groupoid $G$ is \'etale
\cite{renault} in the sense that the set $\mathfrak{ G}$ of
all open subsets $U \subseteq G$ for which the restrictions
$r_{U}=r|_{U} \colon U \to r(U)$ and $s_{U}=s|_{U} \colon U
\to s(U)$ are homeomorphisms is a cover of $G$. Moreover, we
assume that the Haar systems $\lambda$ and $\lambda^{-1}$ are the
families of counting measures. Then the functor $\bfFA$ has
a right adjoint $\bfGA$ and embeds the category of admissible Fell
bundles into a category of very fine coactions of
$C^{*}_{r}(G)$ as a full and coreflective subcategory.  The
construction of the functor $\bfGA$ uses the correspondence
between Banach bundles and convex Banach modules developed
in \cite{dupre}.

\paragraph{The Fell bundle of a coaction of $C^{*}_{r}(G)$}
Let $\delta$ be an injective coaction of $C^{*}_{r}(G)$ on a
$C^{*}$-$\frakb$-algebra ${\cal C}=C^{\gamma}_{K}$.  Since
$G$ is \'etale,  $\rho_{\beta}(\frakB)\subseteq
C^{*}_{r}(G)$ and $\delta(C)\kgamma{1} \subseteq
[\kgamma{1}C^{*}_{r}(G)]$. For each $U \in \mathfrak{G}$, we
define a closed subspace
\begin{align*}
  C_{U} &:= \left\{ c \in [C\rho_{\gamma}(C_{0}(s(U)))] \,\middle|\,
  \delta(c)\kgamma{1} \subseteq [\kgamma{1} L(C_{0}(U))]
\right\} \subseteq C,
\end{align*}
denote by $s_{U*} \colon C_{0}(U) \to C_{0}(s(U))$ and
$r_{U*} \colon C_{0}(U) \to C_{0}(r(U))$ the push-forward of
functions along $s_{U}$ and $r_{U}$, respectively, 
and consider $C_{U}$ as a right Banach $C_{0}(U)$-module via the
formula $c \cdot f :=c\rho_{\gamma}(s_{U*}(f))$. Denote by
$\Gamma_{f}({\cal F})$ the space of all sections of ${\cal
  F}$ that can be written as finite sums of sections in
$\Gamma_{0}({\cal F}|_{U})$, where $U \in
\mathfrak{G}$. Then $\Gamma_{f}({\cal F})$ is a $*$-algebra
with respect to the operations defined in
\eqref{eq:fell-algebra}, and one has natural inclusions
$\Gamma_{c}({\cal F}) \subseteq \Gamma_{f}({\cal F})
\subseteq C^{*}_{r}({\cal F})$ of $*$-algebras.
\begin{proposition} \label{proposition:fell-construction}
  There exist a continuous Fell bundle $\mathcal{ F}$ on $G$ and
  a $*$-homomor\-phism $\iota \colon \Gamma_{f}(\mathcal{ F}) \to
  C$ such that for each $U \in \mathfrak{G}$, the map $\iota$
  restricts to an isometric isomorphism $\iota_{U} \colon
  \Gamma_{0}(\mathcal{ F}|_{U}) \to C_{U}$ of 
Banach  $C_{0}(U)$-modules. If $({\cal F}',\iota')$ is another
  such  pair, then there exists an isomorphism $T \colon
  {\cal F} \to {\cal F}'$ such that $\iota' \circ T_{*} = \iota$.
\end{proposition}
The proof requires some preliminaries. First, note that for
all $c \in C$, $f \in C_{0}(G^{0})$,
\begin{align*}
\delta(c\rho_{\gamma}(f)) = \delta(c)\rho_{(\gamma \rt
  \alpha)}(f) = \delta(c)(1 \btensor \rho_{\alpha}(f)) =
\delta(c)(1 \btensor r^{*}(f)).
\end{align*}
\begin{lemma} \label{lemma:fell-sections} Let
  $U,V \in \mathfrak{G}$.
  \begin{enumerate}
  \item $c \cdot f=\rho_{\gamma}(r_{U*}(f))c$
    for each $c \in C_{U}$ and $f \in C_{0}(U)$. 
  \item $C_{V}C_{U} \subseteq C_{VU}$,
    $(C_{U})^{*}=C_{U^{-1}}$, and $C_{U} = [C_{V} C_{0}(U)]
    \subseteq C_{V}$ if $U \subseteq V$.
  \item $C_{s(U)}$ is a continuous $C_{0}(s(U))$-algebra.
  \item $C_{U}$ is a convex and continuous Banach
    $C_{0}(U)$-module.
  \end{enumerate}
\end{lemma}
\begin{proof}
  i) Let $c,f$ as above.  Since $L(g)r^{*}(s_{U*}(f)) =
  r^{*}(r_{U*}(f))L(g)$ for all $g \in C_{0}(U)$, we have
  $\delta(c \cdot f) = \delta(c)(1 \btensor
  r^{*}(s_{U*}(f))) = (1 \btensor r^{*}(r_{U*}(f)))
  \delta(c) = \delta(\rho_{\gamma}(r_{U*}(f))c)$ and by
  injectivity of $\delta$ also $c\cdot
  f=\rho_{\gamma}(r_{U*}(f))c$.

  ii)  Clearly, $\delta(C_{V}C_{U})\kgamma{1}\subseteq
    \kgamma{1}L(C_{0}(VU))$. Using i) twice, we find
  \begin{align*}
    C_{V}C_{U} &\subseteq
    [C_{V}\rho_{\gamma}(C_{0}(s(V))C_{0}(r(U)))C_{U}] \\ &=
    [C_{V}\rho_{\gamma}(C_{0}(s(V) \cap r(U))C_{U}]
    \subseteq [C\rho_{\gamma}(C_{0}(s(VU)))].
  \end{align*}
  Consequently, $C_{V}C_{U} \subseteq C_{VU}$. By i) again,
  we have $(C_{U})^{*} =
  [\rho_{\gamma}(C_{0}(r(U)))C_{U}]^{*} \subseteq
  [C\rho_{\gamma}(C_{0}(s(U^{-1})))]$, and using the
  relation $\delta(C_{U}^{*}) \kgamma{1} \subseteq
  [\kgamma{1}C^{*}_{r}(G)]$, we obtain
  \begin{align*}
    \delta(C_{U}^{*})\kgamma{1} \subseteq
    [\kgamma{1}\bgamma{1}\delta(C_{U})^{*}\kgamma{1}]
    \subseteq [\kgamma{1}L(C_{0}(U))^{*}
    \bgamma{1}\kgamma{1}] = [\kgamma{1}L(C_{0}(U^{-1}))].
  \end{align*}
  If $U \subseteq V$, then $C_{U}
  \subseteq [C_{V} C_{0}(U)] \subseteq
  C_{V}$, and  $C_{V} C_{0}(U) \subseteq
  C_{U}$ because
  \begin{align*}
    \delta(C_{V} C_{0}(U))\kgamma{1} &=
    \delta(C_{V})\kgamma{1} r^{*}(C_{0}(s(U))) \\
    &\subseteq [\kgamma{1}L(C_{0}(V))r^{*}(C_{0}(s(U)))] =
    [\kgamma{1}L(C_{0}(U))].
  \end{align*}
  
  iii) By ii), $C_{s(U)}$ is a $C^{*}$-algebra.  Consider
  $\kgamma{1}$ as a Hilbert $C^{*}$-module over
  $r^{*}(C_{0}(G^{0})) \cong C_{0}(G^{0})$. Since
  $\delta(C_{G^{0}})\kgamma{1} \subseteq \kgamma{1}$ and
  $\delta(c \cdot f)|\eta\rangle_{1} =
  \delta(c)|\eta\rangle_{1} r^{*}(f)$ for all $c \in
  C_{G^{0}}$, $f \in C_{0}(G^{0})$, $\eta \in \gamma$, the
  formula $c \cdot
  |\eta\rangle_{1}:=\delta(c)|\eta\rangle_{1}$ defines a
  faithful field of representations $C_{G^{0}} \to \mathcal{
    L}(\kgamma{1})$ in the sense of \cite[Theorem
  3.3]{blanchard}. Consequently, $C_{G^{0}}$ is a continuous
  $C_{0}(G^{0})$-algebra and $C_{s(U)}$  a continuous
  $C_{0}(s(U))$-algebra.
 
  iv) Let $c,c' \in C_{U}$ and $f,f' \in C_{0}(U)$ such that
  $0\leq f,f'$ and $f+f'\leq 1$. Then $\|c \cdot f+c' \cdot
  f'\|^{2} = \| c^{*}c \cdot g^{2} + c^{*}c'\cdot gg' +
  c'{}^{*}c \cdot g'g + c'{}^{*}c' \cdot g'{}^{2}\|$, where
  $g=s_{U*}(f)$, $g'=s_{U*}(f')$. Since
  $g^{2}+gg''+g'g+g'{}^{2}\leq 1$ and
  $c^{*}c,c'{}^{*}c',c^{*}c',c'{}^{*}c' \in C_{U^{-1}U}$,
  which is a continuous $C_{0}(s(U))$-algebra and hence a
  convex Banach $C_{0}(s(U))$-module, we get
  $\|cf+c'f'\|^{2} \leq \max\{\|c\|,\|c'\|\}^{2}$.  Finally,
  the norm $\|c_{u}\|^{2} =
  \|(c^{*}c)_{u^{-1}u}\|$
  depends continuously on $u \in U$ because $C_{U^{-1}U}$ is
  a continuous $C_{0}(s(U))$-algebra.
\end{proof}

\begin{proof}[Proof of Proposition  \ref{proposition:fell-construction}]
  Using Lemma \ref{lemma:fell-sections} and \cite{dupre},
  one easily verifies that there exists a continuous Fell
  bundle $\mathcal{ F}$ on $G$ with an isometric isomorphism
  $\iota_{U} \colon \Gamma_{0}(\mathcal{ F}|_{U}) \to C_{U}$
  of Banach $C_{0}(U)$-modules for each $U \in \mathfrak{G}$
  such that for all $U,V \in \mathfrak{G}$, the following
  properties hold. First, the map $\Gamma_{0}(\mathcal{
    F}|_{U}) \hookrightarrow \Gamma_{0}(\mathcal{ F}|_{V})
  \xrightarrow{\iota_{V}} C_{V}$ is equal to
  $\Gamma_{0}(\mathcal{ F}|_{U}) \xrightarrow{\iota_{U}}
  C_{U} \hookrightarrow C_{V}$ if $U\subseteq V$, and
  second, $\iota_{U}(f)^{*}=\iota_{U^{-1}}(f^{*})$,
  $\iota_{UV}(f g) = \iota_{U}(f)\iota_{V}(g)$ for all $f
  \in \Gamma_{0}(\mathcal{ F}|_{U})$, $g \in
  \Gamma_{0}(\mathcal{ F}|_{V}) $.   Define $\iota \colon
  \Gamma_{f}(\mathcal{ F}) \to C$ as follows.  Given $a
  =\sum_{i} a_{i} \in \Gamma_{f}(\mathcal{ F})$, where
  $a_{i} \in \Gamma_{0}({\cal F}|_{U_{i}})$ and $U_{i}
  \in\mathfrak{G}$, let $\iota(a) = \sum_{i}
  \iota_{U_{i}}(a_{i})$.  Using the preceding two properties
  of $\iota$, one easily verifies that $\iota$
  is well-defined and a $*$-homo\-morphism.
\end{proof}
Denote by $p_{0} \colon \Gamma_{f}({\cal F}) \to
\Gamma_{0}({\cal F}^{0})$ the restriction.
\begin{proposition}\label{proposition:etale-expectation}
  There exists a faithful conditional expectation $p \colon
  [\iota(\Gamma_{f}({\cal F}))] \to C_{G^{0}}$ such that $p
  \circ \iota = \iota_{G^{0}} \circ p_{0}$.
\end{proposition}
In the following lemma, $f
h_{\xi,\xi'}$ denotes the pointwise product of functions
$f,h_{\xi,\xi'} \in C_{c}(G)$, where $h_{\xi,\xi'}$ was
defined in \eqref{eq:hxixi}.
\begin{lemma} \label{lemma:fell-eq}
  Let $\xi,\xi' \in C_{c}(G)$, $c \in C$, $f \in C_{c}(G)$.
  \begin{enumerate}
  \item $\langle \eta|_{1}\delta(\langle
    j(\xi)|_{2}\delta(c)|j(\xi')\rangle_{2})|\eta'\rangle_{1}
    = \langle j(\xi)|_{2}\Delta(\langle
    \eta|_{1}\delta(c)|\eta'\rangle_{1})|j(\xi')\rangle_{2}$
    for all $\eta,\eta' \in \gamma$.
  \item $\langle j(\xi)|_{2} \Delta(L(f))|j(\xi')\rangle_{2} =
    L(f h_{\xi,\xi'})$.
  \item $ \langle j(\xi)|_{2} \delta(c \cdot
    f)|j(\xi')\rangle_{2} = c  \cdot (f h_{\xi,\xi'})$ if $c
    \in C_{U}$ and $f \in C_{0}(U)$,  where $U \in
    \mathfrak{G}$.
  \end{enumerate}
\end{lemma}
\begin{proof}
  i) Let $d=\langle
  j(\xi)|_{2}\delta(c)|j(\xi')\rangle_{2}$. Then
  $\delta(d)=\langle j(\xi)|_{3} (\delta \ast
  \Id)(\delta(c)) |j(\xi')\rangle_{3} = \langle j(\xi)|_{3}
  (\Id \ast \Delta)(\delta(c)) |j(\xi')\rangle_{3}$ and
   $ \langle \eta|_{1} \delta(d)|\eta'\rangle_{1} =
  \langle j(\xi)|_{2} \Delta(\langle
  \eta|_{1}\delta(c)|\eta'\rangle_{1}) |j(\xi')\rangle_{2}$.

  ii) This is a special case of Lemma \ref{lemma:fell-eq-1}.

  iii) Let $\eta,\eta' \in \gamma$. Since $c \in C_{U}$, we
  have $\langle\eta|_{1}\delta(c)|\eta'\rangle_{1} =L(g)$
  for some $g \in C_{0}(U)$.  Let $\xi'' =
  r^{*}(s_{U*}(f))\xi$ and denote by $d_{l},d_{r} \in C$ the
  left and the right hand side of the equation in iii),
  respectively.  Then $d_{l} = \langle
  j(\xi)|_{2}\delta(c)|j(\xi'')\rangle_{2}$, and by 
  i) and ii),
  \begin{align*}
    \langle \eta|_{1}\delta(d_{l}) |\eta'\rangle_{1} &=
    \langle j(\xi)|_{2}
    \Delta(\langle\eta|_{1}\delta(c)|\eta'\rangle_{1})
    |j(\xi'')\rangle_{2} = \langle
    j(\xi)|_{2}L(g)|j(\xi'')\rangle_{2} = L(gh_{\xi,\xi''}), \\
    \langle
    \eta|_{1}\delta(d_{r})|\eta'\rangle_{1}
    &= \langle\eta|_{1}\delta(c)|\eta'\rangle_{1}
    r^{*}(s_{U*}(fh_{\xi,\xi'})) =
    L(g)L(s_{U*}(fh_{\xi,\xi'})).
\end{align*}
We can conclude that $\langle
\eta|_{1}\delta(d_{l})|\eta'\rangle_{1}=\langle
\eta|_{1}\delta(d_{r})|\eta'\rangle_{1}$ because for all
$x\in G$,
\begin{align*}
  (gh_{\xi,\xi''})(x) =g(x)
  \int_{G^{r(x)}}\overline{\xi(y)}f(x)\xi'(x^{-1}y)
  \intd\lambda^{r(x)}(y)
  = g(x) (s_{U*}(fh_{\xi,\xi'}))(s(x)).
\end{align*}
Since $\eta,\eta' \in \gamma$ were arbitrary and $\delta$ is
injective, we must have $d_{l}=d_{r}$. 
\end{proof}
\begin{proof}[Proof of Proposition \ref{proposition:etale-expectation}]
  Given a subset $U \subseteq G$, denote by $\chi_{U}$ its
  characteristic function.  Using the same formulas as for
  elements of $C_{c}(G)$, we can define a map $j(\xi) \colon
  \frakK \to H$ and the function $h_{\xi,\xi'}$ for the
  characteristic function $\xi=\xi'=\chi_{G^{0}}$ of $G^{0}
  \subseteq G$, and then Lemma \ref{lemma:fell-eq} still
  holds. Define $p \colon C \to C$ by $c \mapsto \langle
  j(\chi_{G^{0}})|_{2} \delta(c)
  |j(\chi_{G^{0}})\rangle_{2}$.  Then $\|p\| \leq \|
  j(\chi_{G^{0}})\|^{2} = 1$, and the relation
  $h_{\chi_{G^{0}},\chi_{G^{0}}} = \chi_{G^{0}}$ and Lemma
  \ref{lemma:fell-eq} imply that $p|_{C_{G^{0}}} = \Id$ and
  $p|_{C_{U}} = 0$ whenever $U \in \mathfrak{G}$ and $U \cap
  G^{0} = \emptyset$.  Using a partition of unity argument
  and the fact that $G^{0} \subseteq G$ is open and closed,
  we can conclude that $p \circ \iota =\iota_{G^{0}} \circ
  p_{0}$. 

  It remains to show that $p$ is faithful.  Using the
  right-regular representation of $G$, one easily verifies
  that $[C^{*}_{r}(G)'j(\chi_{G^{0}})\frakK]=H$.  Therefore,
  the map $q \colon C^{*}_{r}(G) \to {\cal L}(\frakK)$, $a
  \mapsto j(\chi_{G^{0}})^{*}aj(\chi_{G^{0}})$, is faithful
  in the sense that $q(a^{*}a) \neq 0$ whenever $a \neq
  0$. If $c \in [\iota(\Gamma_{f}({\cal F}))]$ and $p(c^{*}c)=0$, then
  $\eta^{*}p(c^{*}c)\eta = q(\langle
  \eta^{*}|_{1}\delta(c^{*}c)|\eta\rangle_{1}) = 0$ and
  hence $\langle \eta^{*}|_{1}\delta(c^{*}c)|\eta\rangle_{1}
  = 0$ and $\delta(c)|\eta\rangle_{1}=0$ for all $\eta \in
  \gamma$, whence $\delta(c)=0$ and $c=0$ by injectivity of
  $\delta$.
\end{proof}
 Proposition
  \ref{proposition:etale-expectation} and \cite[Fact
  3.11]{kumjian} imply:
  \begin{corollary}
  $\iota$ extends to an embedding $
  C^{*}_{r}({\cal F}) \to C$. \qed
\end{corollary}
We denote the extension above by $\iota$ again.
 \begin{proposition}\label{proposition:fell-iso}
   If $\delta$ is fine, then $\iota \colon C^{*}_{r}({\cal
     F}) \to C$ is a $*$-isomorphism.
 \end{proposition}
\begin{proof}
  We only need to show that $C$ is equal to the linear span
  of all $C_{U}$, where $U \in \mathfrak{G}$. Consider an
  element $d \in C$ of the form $d= \langle j(\xi)|_{2}
  \delta(c) |j(\xi')\rangle_{2}$, where $c \in C, \xi \in
  C_{c}(V),\xi'\in C_{c}(V') \text{ for some } V,V' \in
  \mathfrak{G}$.  Since $G$ is \'etale and $\delta$ is
  fine, the closed linear span of all elements of the form
  like $d$ is equal to $ [\balpha{2}\delta(C)\kalpha{2}] =
  [\balpha{2}\kalpha{2}C] = C$.  We show that $d \in C_{U}$,
  where $U=VV'{}^{-1} \in \mathfrak{G}$, and then the claim
  follows.  Let $\eta,\eta' \in \gamma$. By Lemma
  \ref{lemma:fell-eq},
  \begin{align*}
    \langle \eta|_{1} \delta(d)|\eta'\rangle_{1} &
      \in \langle
      j(\xi)|_{2}\Delta(C^{*}_{r}(G))|j(\xi')\rangle_{2} \subseteq
      [L(C_{c}(G)h_{\xi,\xi'})] \subseteq L(C_{0}(U)).
    \end{align*}
    Using the relation $\delta(d)\kgamma{1} \subseteq
    [\kgamma{1}C^{*}_{r}(G)]$, we find $\delta(d)\kgamma{1}
    \subseteq [\kgamma{1}\bgamma{1} \delta(d)\kgamma{1}]
    \subseteq [\kgamma{1}L(C_{0}(U))]$.  Moreover, since
    $h_{\xi,\xi'} \in C_{c}(U)$, we can choose $g \in
    C_{0}(U)$ with $h_{\xi,\xi'}g=h_{\xi,\xi'}$. Then
    $L(fh_{\xi,\xi'})r^{*}(s_{U*}(g))=L(fh_{\xi,\xi'})$ for each $f
    \in C_{0}(U)$, and hence $\langle
    \eta|_{1}\delta(d\rho_{\gamma}(s_{U*}(g)))|\eta'\rangle_{1} =
    \langle \eta|_{1}\delta(d)|\eta'\rangle_{1}r^{*}(s_{U*}(g)) =
    \langle \eta|_{1}\delta(d)|\eta'\rangle_{1}$. Since
    $\delta$ is injective, we can conclude
    $d=d\rho_{\gamma}(s_{U*}(g)) \in C\rho_{\gamma}(C_{0}(s(U)))$
    and finally $d \in C_{U}$.
\end{proof}
 \begin{proposition}\label{proposition:fell-admissible}
   If $\delta$ is fine, then ${\cal F}$ is admissible.
 \end{proposition}
 \begin{proof}
   The proof is similar to the proof of Lemma
   \ref{lemma:groupoid-cx-g} i).  By
   \ref{lemma:fell-sections} iii), $\Gamma_{0}({\cal F}^{0})
   \cong C_{G^{0}}$ is a continuous $C_{0}(G^{0})$-algebra.
   Let $u \in G^{0}$, denote by $I_{u} \subset C_{0}(G^{0})$
   the ideal of all functions vanishing at $u$, and assume
   that ${\cal F}_{u} = 0$. Then $\Gamma_{0}({\cal
     F}^{0})=[\Gamma_{0}({\cal F}^{0})I_{u}]$ and
   $[C^{*}_{r}({\cal F})]=[C^{*}_{r}({\cal
     F})\Gamma_{0}({\cal F}^{0})] = [C^{*}_{r}({\cal
     F})I_{u}]$, whence $C=[C\rho_{\gamma}(I_{u})]$. Define
   $j(\chi_{G^{0}})$ as in the proof of Proposition
   \ref{proposition:etale-expectation}. Then
   $[\delta(C)\kgamma{1}C^{*}_{r}(G)]=[\kgamma{1}C^{*}_{r}(G)]$
   and
   \begin{align*}
     [r^{*}(C_{0}(G^{0}))C^{*}_{r}(G)]=
     [\bgamma{1}\kgamma{1}C^{*}_{r}(G)] &=
     [\bgamma{1}\delta(C I_{u})\kgamma{1}C^{*}_{r}(G)] \\ &=
     [\bgamma{1}\kgamma{1}r^{*}(I_{u})C^{*}_{r}(G)] =
     [r^{*}(I_{u})C^{*}_{r}(G)],
   \end{align*}
   whence $[j(\chi_{G^{0}})^{*}C^{*}_{r}(G)j(\chi_{G^{0}})]=I_{u} \neq
   C_{0}(G^{0})$, a contradiction.
 \end{proof}
The construction of the Fell bundle is functorial with
respect to  the
following class of morphisms.
\begin{definition}
  A morphism $\rho$ of coactions
  $(C^{\gamma}_{K},\delta_{C})$ and
  $(D^{\epsilon}_{L},\delta_{D})$ of $C^{*}_{r}(G)$ is {\em
    strongly nondegenerate} if $[\rho(C)D_{G^{0}}]=D$.
\end{definition}
\begin{proposition} \label{proposition:fell-morphism-coact} Let
  $\pi$ be a strongly nondegenerate morphism of fine
  coactions $(C^{\gamma}_{K},\delta_{C})$,
  $(D^{\epsilon}_{L},\delta_{D})$ with associated Fell
  bundles $\mathcal{ F}$, $\mathcal{ G}$ and
  $*$-homomorphisms $\iota_{\mathcal{F}}$, $\iota_{{\cal
      G}}$. Then there exists a unique morphism $T$ from
  ${\cal F}$ to ${\cal G}$ such that $\iota_{{\cal G}} \circ
  T_{*} = \pi \circ \iota_{\mathcal{F}}$.
\end{proposition}
\begin{proof}
Let $U,V \in \mathfrak{G}$.  Then  $\pi(C_{U})D_{V} \subseteq
  D_{UV}$ because
  \begin{align*}
    \delta_{D}(\pi(C_{U})D_{V}) |\epsilon\rangle_{1} &=
    ((\pi \ast
    \Id)(\delta_{C}(C_{U})))\delta_{D}(D_{V})|\epsilon\rangle_{1}
    \\ & \subseteq ((\pi \ast \Id)(\delta_{C}(C_{U})))
    |\epsilon\rangle_{1}L(C_{0}(V)) \\ & \subseteq
    |\epsilon\rangle_{1}L(C_{0}(U))L(C_{0}(V)) =
    |\epsilon\rangle_{1} L(C_{0}(UV)) \\ \text{and } 
    \pi(C_{U})D_{V}
    &\subseteq [\pi(C\rho_{\gamma}(C_{0}(s(U)))) D_{V}] 
     \subseteq
    [\pi(C)D\rho_{\epsilon}(C_{0}(s(UV)))],
\end{align*}
where the last inclusion follows similarly as in the proof
of Lemma \ref{lemma:fell-sections} ii).  Define a map
$S_{U,V} \colon \Gamma_{0}(\mathcal { F}|_{U}) \times
\Gamma_{0}(\mathcal { G}|_{V}) \to \Gamma_{0}(\mathcal{
  G}|_{UV})$ by $(f,g) \mapsto \iota_{\mathcal{
    G}}^{-1}(\pi(\iota_{\mathcal { F}}(f))\iota_{\mathcal{
    G}}(g))$, let $(x,y) \in (U \times V) \cap \GsrG$, and
denote by $I_{x} \subseteq \Gamma_{0}(\mathcal { F}|_{U})$,
$I_{y} \subseteq \Gamma_{0}(\mathcal { G}|_{V})$, $I_{xy}
\subseteq \Gamma_{0}(\mathcal { G}|_{UV})$ the subspaces of
all sections vanishing at $x,y$, and $xy$,
respectively. Using Lemma \ref{lemma:fell-sections} i), one
easily verifies that $S_{U,V}$ maps $I_{x} \times
\Gamma_{0}(\mathcal { G}|_{V})$ and $\Gamma_{0}(\mathcal{
  F}|_{U}) \times I_{y}$ into $I_{xy}$. Hence, there exists
a unique map $S_{x,y} \colon \mathcal{ F}_{x} \times
\mathcal{ G}_{y} \to \mathcal{ G}_{xy}$ such that
$S_{x,y}(f(x),g(y)) = (S_{U,V}(f,g))(xy)$ for all $f \in
\Gamma_{0}(\mathcal { F}|_{U})$, $g \in \Gamma_{0}(\mathcal{
  G}|_{V})$, and this map depends on $(x,y)$ but not on
$(U,V)$.  For each $x \in G$ and $c \in \mathcal{ F}_{x}$,
define $T(c) \colon \mathcal{ G}|_{G^{s(x)}} \to \mathcal{
  G}|_{G^{r(x)}}$ by $T(c)d = S_{x,y}(c,d)$ for each $y \in
G^{s(x)}$, $d \in \mathcal { G}_{y}$. One easily checks that
then $T$ is a continuous map from $\mathcal{ F}$ to
$\mathcal{ M(G)}$ which satisfies conditions i) and ii) of
Definition \ref{definition:fell-morphism}, and that the
representation $\tilde \pi :=\iota_{{\cal G}}^{-1} \circ \pi
\circ \iota_{{\cal F}} \colon C^{*}_{r}({\cal F}) \to
M(C^{*}_{r}({\cal G}))$ satisfies $\tilde \pi(f)g = (T \circ
f)g$ for all $f \in \Gamma_{c}({\cal F})$, $g \in
\Gamma_{c}({\cal G})$. We show that $T$ also satisfies
condition iii) of Definition
\ref{definition:fell-morphism}. Since $\pi$ is strongly
nondegenerate, $D=[\pi(C)D_{G^{0}}]$, that is,
$C^{*}_{r}({\cal G})=[\tilde \pi(C^{*}_{r}({\cal
  F}))\Gamma_{0}({\cal G}^{0})]$ and hence $\Gamma^{2}({\cal
  G},\lambda^{-1})=[\tilde \pi(C^{*}_{r}({\cal
  F}))\Gamma_{0}({\cal G}^{0})]$. In particular,
${\cal G}_{x}=[T({\cal F}_{x}){\cal G}_{s(x)}]$ for each $x
\in G$ because $G_{s(x)}$ is discrete.
\end{proof}

\paragraph{The unit and counit of the adjunction}
Denote by $\bfcoact_{C^{*}_{r}(G)}^{as}$ the category of
very fine left-full coactions of $C^{*}_{r}(G)$ with 
all strongly nondegenerate morphisms.  Then the functor
$\bfFA \colon \bffell_{G}^{a} \to
\bfcoact^{a}_{C^{*}_{r}(G)}$ constructed in the preceding
section actually takes values in
$\bfcoact_{C^{*}_{r}(G)}^{as}$:
\begin{lemma}
  Let $T$ be a morphism of admissible Fell bundles ${\cal
    F},{\cal G}$ on $G$. Then the morphism $\bfFA T$ from
  $\bfFA {\cal F}$ to $\bfFA {\cal G}$ is strongly
  nondegenerate.
\end{lemma}
\begin{proof}
  Immediate from Proposition \ref{proposition:fell-morphism} ii).
\end{proof}
The constructions in Proposition
\ref{proposition:fell-construction} and
\ref{proposition:fell-morphism-coact} yield a functor $\bfGA
\colon \bfcoact_{C^{*}_{r}(G)}^{as} \to \bffell^{a}_{G}$.  We
now obtain an embedding
$(\bfFA,\bfGA,\check\eta,\check\epsilon)$ of $\bffell_{G}^{a}$
into $\bfcoact^{as}_{C^{*}_{r}(G)}$ as a full and coreflective
subcategory.
\begin{proposition} \label{proposition:fell-coaction-eta}
  Let $\mathcal { F}$ be an admissible Fell bundle,
  $(\pi_{\mathcal{ F}}(C^{*}_{r}(\mathcal{ F}))_{K_{{\cal
        F}}}^{\gamma_{{\cal F}}},\delta_{\mathcal {
      F}})=\bfFA {\cal F}$ the associated fine coaction, and
  $\mathcal { G}=\bfGA\bfFA {\cal F}$ and $\iota_{\mathcal {
      G}} \colon C^{*}_{r}(\mathcal { G}) \to \pi_{\mathcal
    { F}}(C^{*}_{r}(\mathcal { F}))$ the Fell bundle and the
  $*$-homomorphism associated to this coaction as
  above. Then there exists a unique isomorphism
  $\check\eta_{\mathcal { F}} \colon \mathcal { F} \to
  \mathcal { G}$ such that $\iota_{\mathcal { G}}\circ
(\check\eta_{\mathcal { F}})_{*} = \pi_{\mathcal {
      F}}$.
\end{proposition}
\begin{proof}
Let $(C^{\gamma}_{K},\delta)=(\pi_{\mathcal{
      F}}(C^{*}_{r}(\mathcal{ F}))_{K_{{\cal
        F}}}^{\gamma_{{\cal F}}},\delta_{\mathcal { F}})$
  and $U \in \mathfrak{ G}$. We show that
  $C_{U}=\pi_{\mathcal { F}}(\Gamma_{0}(\mathcal{ {\cal
      F}}|_{ U}))$.  Note that $[\{ h_{\xi,\xi'} \mid \xi
  \in C_{c}(r(U)), \xi' \in C_{c}(U)\}] = C_{0}(U)$, where
  the functions $h_{\xi,\xi'}$ were defined in
  \eqref{eq:hxixi}.  Using Lemma \ref{lemma:fell-eq}, we
  can conclude
  \begin{align*} 
    C_{U} = [\langle j(C_{c}(r(U))) |_{2} \delta(\pi_{\mathcal{
        F}}(\Gamma_{c}(\mathcal { F})) )|j(C_{c}(U))
    \rangle_{2}].
  \end{align*}
By Lemma 
  \ref{lemma:fell-eq-1}, we have  for all $\xi \in C_{c}(r(U))$, $f
  \in \Gamma_{c}(\mathcal { F})$, $\xi' \in C_{c}(U)$,  
  \begin{align*}
    \langle j(\xi)|_{2}\delta(\pi_{\mathcal{
        F}}(f))|j(\xi')\rangle_{2} = \pi_{{\mathcal { F}}}(f 
    h_{\xi,\xi'}) \in \pi_{\mathcal { F}}(\Gamma_{0}(\mathcal{
      F}|_{U})),
  \end{align*}
  where $f h_{\xi,\xi'}$ denotes the pointwise product.
  Consequently, $C_{U} = \pi_{\mathcal {
      F}}(\Gamma_{0}(\mathcal { F}|_{U}))$.  Since $U \in
  \mathfrak{G}$ was arbitrary, we can conclude that there
  exists an isomorphism $\check \eta_{{\cal F}} \colon
  \mathcal { F} \to \mathcal { G}$ of Banach bundles such
  that $\iota_{\mathcal { G} } \circ (\check \eta_{{\cal
      F}})_{*} = \pi_{\mathcal { F}} \colon \Gamma_{c}({\cal
    F}) \to C$. Using the fact that $(\check \eta_{{\cal
      F}})_{*}$ is a $*$-homomorphism and that $G$ is
  \'etale, one easily concludes that $\check \eta_{{\cal
      F}}$ is an isomorphism of Fell bundles.
\end{proof} 

\begin{proposition} \label{proposition:fell-coaction-epsilon}
  Let $({\cal C},\delta)$ be a very fine coaction of
  $C^{*}_{r}(G)$, where ${\cal C}=C^{\gamma}_{K}$, and let
  $\mathcal { F}, \iota\colon C^{*}_{r}(\mathcal { F}) \to
  C$ be the associated Fell bundle and $*$-isomorphism. Then
  there exists a unique strongly nondegenerate morphism
  $\check\epsilon_{({\cal C},\delta)}$ from $(\pi_{\mathcal {
      F}}(C^{*}_{r}(\mathcal { F}))^{\gamma_{\mathcal{
        F}}}_{K_{\mathcal { F}}},\delta_{\mathcal { F}})$ to
  $({\cal C},\delta)$ such that $\check\epsilon_{({\cal
      C},\delta)} \circ \pi_{\mathcal { F}} = \iota$.
\end{proposition}
\begin{lemma}
  Let $U \in \mathfrak{G}$, $\xi \in C_{c}(U)$, $\eta \in
  \gamma$, and $\omega = |\eta\rangle_{1}j(\xi) \in \gamma
  \rt \alpha \subseteq \mathcal{ L}(\frakK,\rHrange)$.
  \begin{enumerate}
  \item There exists a $C_{0}(G^{0})$-weight $\phi \colon
    \Gamma_{0}({\cal F}^{0}) \to C_{0}(G^{0})
    \subseteq \mathcal{ L}(\frakK)$, $f \mapsto
    \omega^{*}\delta(\iota(f))\omega$.
  \item There exists a unique isometry $S_{\omega} \colon
    K_{\phi} = \Gamma^{2}(\mathcal { F},\nu;\phi) \to
    \rHrange$ such that $S_{\omega}\hat j_{ \phi}(f) =
    \delta(\iota(f))\omega$ for all $f \in
    \Gamma_{c}(\mathcal { F})$. Furthermore,
    $S_{\omega}\pi_{\phi}(f)= \delta(\iota(f))S_{\omega}$
    for all $f \in \Gamma_{c}(\mathcal{ F})$.
  \item $S_{\omega}j_{\phi}(\Gamma_{c}(\mathcal { F}))
    \subseteq \gamma \rt \alpha$.
  \end{enumerate}
\end{lemma}
\begin{proof}
  i) First, note that $\omega^{*}\delta(C_{G^{0}})\omega
  \subseteq
  [\alpha^{*}\bgamma{1}\kgamma{1}L(C_{0}(G^{0}))\alpha] =
  [\alpha^{*}\bgamma{1}\kgamma{1}\alpha] = C_{0}(G^{0})
  \subseteq \mathcal{ L}(\frakK)$. Second, observe that for
  all $c \in C_{G^{0}}$, $f\in C_{0}(G^{0})$,
  \begin{align*}
    \phi(cf)  &=
    j(\xi)^{*}\langle\eta|_{1} \delta(c f)
    |\eta\rangle_{1}j(\xi) \\ &=    j(\xi)^{*}\langle\eta|_{1}
    \delta(c) |\eta\rangle_{1} r^{*}(f)j(\xi) = j(\xi)^{*}
    \langle \eta|_{1}\delta(c)|\eta\rangle_{1} j(\xi)f = \phi(c)f.
  \end{align*}

  ii) As before, denote by $p_{0} \colon \Gamma_{f}(\mathcal
  { F}) \to \Gamma_{0}(\mathcal { F}^{0})$ the restriction.
  Let $U \in \mathfrak{G}$, $f,f' \in \Gamma_{c}(\mathcal {
    F})$, and $g=f^{*}f'$. Using the relation $\supp
  h_{\xi,\xi} \subseteq G^{0}$ and Lemma
  \ref{lemma:fell-eq}, we find
  \begin{align*}
    \omega^{*} \delta(\iota(f))^{*}\delta(\iota(f'))\omega
    &= \eta^{*} \langle
    j(\xi)|_{2}\delta(\iota(g))|j(\xi)\rangle_{2}
    \eta^{*}  \\ &=
    \eta^{*}  \iota(g \cdot  h_{\xi,\xi}) \eta \\ &=
    \omega^{*}\delta(\iota(p_{0}(g)))\omega^{*} =
    \phi(p_{0}(g)) =  \langle
    f|f'\rangle_{\Gamma^{2}(\mathcal { F},\lambda^{-1},\phi)}.
  \end{align*}
  The existence of $S_{\omega}$ follows. Finally, $
  S_{\omega}\pi_{\phi}(f) = \delta(\iota(f))
  S_{\omega}$ because $ S_{\omega}\pi_{\phi}(f) \hat
  j_{\phi}(g) = S_{\omega}\hat j_{\phi}(fg) =
  \delta(\iota(fg))\omega =
  \delta(\iota(f)) S_{\omega}\hat
  j_{\phi}(g) $ for all $f,g \in \Gamma_{c}(\mathcal{ F})$.

  iii) Let $V \in \mathfrak{G}, f \in \Gamma_{c}({\cal
    F}|_{V}), \zeta \in C_{c}(G^{0})$, and define $\zeta'
  \in L^{2}(G^{0},\mu)$ by
  $\zeta'(s(x))=\zeta(r(x))D^{1/2}(x)$ for all $x \in V$ and
  $\zeta'(y)=0$ for all $y \in G^{0} \setminus s(V)$. Then
  $(j_{\phi}(f)\zeta)(x) = f(x)\zeta(r(x)) = (\hat
  j_{\phi}(f)\zeta')(x)$ for all $x \in G$ and therefore
  \begin{align*}
    S_{\omega}j_{\phi}(f)\zeta = S_{\omega}\hat
    j_{\phi}(f)\zeta' = \delta(\iota(f))\omega \zeta' =
    \delta(\iota(f))|\eta\rangle_{1}j(\xi) \zeta'.
  \end{align*}
  Since $f\in \Gamma_{c}(\mathcal { F}|_{V})$, there exist
  $f' \in L(C_{0}(V))$, $\eta' \in \gamma$ such that
  $\delta(\iota(f))|\eta\rangle_{1} =
  |\eta'\rangle_{1}L(f')$. Now,
  \begin{align*}
    S_{\omega}j_{\phi}(f)\zeta =
    \delta(\iota(f))|\eta\rangle_{1} j(\xi)\zeta' =
    |\eta'\rangle_{1}L(f')j(\xi)\zeta' = |\eta'\rangle
    j(L(f')\xi)\zeta
  \end{align*}
  because $(L(f')j(\xi)\zeta')(z)=0$ for $z \not\in VU$ and
  \begin{align*}
    (L(f')j(\xi)\zeta')(xy) &=
    D^{-1/2}(x)f'(x)\xi(y)\zeta'(r(y)) =
    f'(x)\zeta(r(x))\xi(y) = j(L(f')\xi)\zeta(xy)
  \end{align*}
  for all $(x,y) \in (V \times U) \cap \GsrG$. Thus,
  $S_{\omega}j_{\phi}(f)\zeta \in \gamma \rt
  \alpha$. The claim follows.
\end{proof}
\begin{proof}[Proof of Proposition \ref{proposition:fell-coaction-epsilon}]
  Since $\pi_{\mathcal { F}}$ is injective, we can define
  $\check\epsilon=\check\epsilon_{({\cal C},\delta)}:=\iota
  \circ \pi_{\mathcal { F}}^{-1}$.  We show that $\delta
  \circ \check\epsilon$ is a morphism from $\pi_{\mathcal {
      F}}(C^{*}_{r}(\mathcal { F}))^{\gamma_{\mathcal{
        F}}}_{K_{\mathcal { F}}}$ to $\delta(C)^{\gamma \rt
    \alpha}_{\rHrange}$.  For each $C_{0}(G^{0})$-weight
  $\phi$ on $\Gamma_{0}(\mathcal { F}^{0})$, denote by
  $p_{\phi}\colon K_{\mathcal { F}} \to K_{\phi}$ the
  canonical projection.  Let $\mathcal{ S} \subseteq
  \mathcal{ L}(K_{\mathcal { F}},\rHrange)$ be the closed
  linear span of all operators of the form
  $S_{\omega}p_{\phi}$, where $U,\xi,\eta,\omega,\phi$ are
  as in the lemma above.  Then
  $Sa=\delta(\check\epsilon(a))$ for each $S \in \mathcal{
    S}$, $a \in \pi_{{\cal F}}(C^{*}_{r}(\mathcal { F}))$,
  and $[\mathcal{ S}\gamma_{\mathcal { F}}] =
  [\delta(\iota(\Gamma_{c}(\mathcal { F})))(\gamma \rt
  \alpha)] = \gamma \rt \alpha$. The claim follows. Since
  $\delta$ is an isomorphism from ${\cal C}$ to
  $\delta(C)^{\gamma \rt \alpha}_{\rHrange}$, we can
  conclude that $\check\epsilon$ is a morphism from
  $\pi_{\mathcal { F}}(C^{*}_{r}(\mathcal{
    F}))^{\gamma_{\mathcal { F}}}_{K_{\mathcal { F}}}$ to
  ${\cal C}$.  The relation $(\check\epsilon \ast \Id) \circ
  \delta= \delta \circ \check\epsilon$ follows from the fact
  that
  \begin{align*}
    \langle j(\xi)|_{2} \delta(\check\epsilon(g)\cdot
    f)|j(\xi')\rangle_{2} &= \check\epsilon(g) \cdot (f
    h_{\xi,\xi'}) \\ &= \check\epsilon(g \cdot (f h_{\xi,\xi'})) =
    \check\epsilon\big( \langle j(\xi)|_{2} \delta(\pi_{\mathcal {
        F}}(g)) |j(\xi')\rangle_{2}\big)
  \end{align*}
  for all $U \in \mathfrak{G}$, $g \in \Gamma_{c}(\mathcal {
    F}|_{U})$, $f \in C_{c}(U)$, $\xi,\xi' \in C_{c}(G)$ by
  Lemma \ref{lemma:fell-eq}.
\end{proof}
\begin{corollary}
  Every very fine coaction of $C^{*}_{r}(G)$ is left-full.
\end{corollary}
\begin{proof}
  Let $({\cal C},\delta)$ be a very fine coaction of
  $C^{*}_{r}(G)$, let $\check\epsilon_{({\cal C},\delta)}$
  and $(\pi_{\mathcal { F}}(C^{*}_{r}(\mathcal {
    F}))^{\gamma_{\mathcal{ F}}}_{K_{\mathcal {
        F}}},\delta_{\mathcal { F}})$  as above, and let
  $I:=\{ T \in {\cal L}_{s}((K_{{\cal F}},\gamma_{{\cal
      F}}),(K,\gamma)) \mid Tx=\check\epsilon_{({\cal
      C},\delta)}(x)T$ for all $x \in \pi_{{\cal
      F}}(C^{*}_{r}({\cal F}))\}$.  Then
  $\gamma=[I\gamma_{{\cal F}}]$ because
  $\check\epsilon_{({\cal C},\delta)}$ is a morphism, and
  since $\delta_{{\cal F}}$ is left-full,
  \begin{align*} { [\delta(C)\kgamma{1}C^{*}_{r}(G)]} &{=
      [(\check\epsilon_{({\cal C},\delta)} \ast
      \Id)(\delta_{{\cal F}}( \pi_{{\cal F}}(C^{*}_{r}({\cal
        F})))) (I \btensor \Id)
      |\gamma_{{\cal F}}\rangle_{1}C^{*}_{r}(G)]} \\
    &= [(I \btensor \Id)\delta_{{\cal F}}(\pi_{{\cal
          F}}(C^{*}_{r}({\cal F})))|\gamma_{{\cal
          F}}\rangle_{1}C^{*}_{r}(G)] \\ & = [(I \btensor
      \Id)|\gamma_{{\cal F}}\rangle_{1}C^{*}_{r}(G)] =
      [\kgamma{1}C^{*}_{r}(G)]. \qedhere
  \end{align*}
\end{proof}
\begin{theorem}
  $(\bfFA,\bfGA,\check\eta,\check\epsilon)$ is an embedding of
  $\bffell^{a}_{G}$ into $\bfcoact_{C^{*}_{r}(G)}^{as}$ as a
  full and coreflective subcategory.
\end{theorem}
\begin{proof}
  One easily verifies that $\bfGA$ is faithful and that the
  families $(\check\eta_{\mathcal{ F}})_{\mathcal { F}}$ and
  $(\check\epsilon_{({\cal C},\delta)})_{({\cal C},\delta)}$
  are natural transformations as desired. Since $\check\eta$
  is a natural isomorphism, $\bfFA$ is full and faithful
  \cite[IV.3 Theorem 1]{maclane}.
\end{proof}

\def\cprime{$'$}

\end{document}